\documentclass[12pt,A4,leqno]{amsart}
\usepackage{amsfonts}
\usepackage{mathrsfs}
\usepackage[T1]{fontenc}
\usepackage{amssymb,amscd}
\usepackage{upgreek}
\usepackage{color}
\usepackage{tikz-cd}
\usepackage{dsfont}
\usepackage{geometry}
\usepackage{epsf}	
\usepackage{graphicx}
\usepackage{extarrows}
\usepackage[curve]{xypic}
\usepackage{hyperref} 
\usepackage{url}
\hypersetup{colorlinks,linkcolor={blue},citecolor={blue},urlcolor={blue}} 
\usepackage[capitalise]{cleveref}

\theoremstyle{plain}
\newtheorem{thm}{Theorem}[section]
\newtheorem{lem}[thm]{Lemma}
\newtheorem{prop}[thm]{Proposition}
\newtheorem{cor}[thm]{Corollary}

\theoremstyle{definition}
\newtheorem{defi}[thm]{Definition}
\newtheorem{exam}[thm]{Example}
\newtheorem{rem}[thm]{Remark}

\newtheorem{warn}[thm]{Warning}

\newcommand{\R}{\mathbb R}
\newcommand{\Z}{\mathbb Z}

\newcommand{\nn}{\vskip 0.2cm}
\newcommand{\n}{\vskip 0.1cm}

\renewcommand{\geq}{\geqslant}
\renewcommand{\leq}{\leqslant}

\makeatletter
\renewcommand{\@secnumfont}{\bfseries}
\makeatother

\makeatletter
\def\section{%
  \@startsection{section}{1}
    {\z@}
    {2.0ex plus 0.8ex minus .1ex}
    {1.0ex plus .2ex}
    {\bfseries\large\centering\MakeUppercase}%
}
\makeatother

\begin{document}
\title{Weighted homology theory of orbifolds 
and Weighted Polyhedra}
\author{*Yin Wei, $^\dagger$Lisu Wu, and $^\ddagger$Li Yu}
\address{*School of Mathematics, Nanjing University\\ Nanjing\\210093\\ P.R.China}
\email{2986343993@qq.com}
\address{$^\dagger$College of Mathematics and Systems Science, Shandong University of 
Science and Technology, Tsingtao, 266590, P.R.China}
\email{wulisu@sdust.edu.cn} 
\address{$^\ddagger$School of Mathematics, Nanjing University\\ Nanjing\\210093\\ P.R.China}
\email{yuli@nju.edu.cn}


\date{\today}


\begin{abstract}
  We introduce two new homology theories of orbifolds   from some special type of triangulations adapted
  to an orbifold, called AW-homology and DW-homology. 
  The main idea in the definitions of these two homology theories is that we use divisibly weighted simplices as the building blocks
  of an orbifold and encode the orders of the local groups of the orbifold in the boundary maps of their chain complexes so that these
  two theories can reflect some structural information of the singular set of the orbifold. We
  prove that AW-homology and DW-homology groups are invariants of compact orbifolds under orbifold isomorphisms and more generally under certain type of homotopy equivalences of orbifolds. 
   Moreover, we find that there exists a natural graded commutative product in the cohomology groups corresponding to the DW-homology, which generalizes the cup product in the ordinary simplicial cohomology. In addition, we introduce a broader class of objects called weighted polyhedra and develop our AW-homology and DW-homology theory in this broader setting. When a weighted polyhedron is based on
     a compact orientable homology $n$-manifold, we prove that its AW-homology and DW-homology satisfy a generalized version of Poincar\'e duality with respect to its
     DW-cohomology and AW-cohomology, respectively.
     Our goal is to generalize the whole simplicial (co)homology theory to any triangulable topological space with a suitable weight function.
\end{abstract}

\maketitle

\section{Introduction}

 Orbifolds were first introduced by I.~Satake in~\cite{Sa56} where they were originally called $V$-manifolds.
 Roughly speaking, an $n$-dimensional orbifold is a topological space locally modeled on the quotient spaces of open subsets in the Euclidean space $\mathbb{R}^n$ by some finite group actions.
 Orbifolds have arisen naturally
in many ways in mathematics. For example, the orbit space of any proper action by a
discrete group on a manifold has the structure of an orbifold,
this applies in particular to moduli spaces (see Thurston~\cite{Thurs77} and Scott~\cite{Scott83} for more examples).\n

 Homology
and cohomology theories of orbifolds have been defined
and studied in many different ways.  For example, equivariant cohomology of global quotients, \v{C}ech and de Rham cohomology in~\cite{Sa56}, simplicial cohomology in Moerdijk and Pronk~\cite{MoePro99},
Chen-Ruan cohomology in~\cite{ChenRuan04}, t-singular homology in Takeuchi and Yokoyama~\cite{TakYok06}, ws-singular cohomology in Takeuchi and Yokoyama~\cite{TakYok07}, string homology in Lupercio, Uribe and Xicotencatl~\cite{LupUriXic08,Yas18} 
and so on. Note that many theories only study  orbifolds that have a manifold covering. \n

 Here we introduce two new homology theories for any compact orbifold called AW-homology and DW-homology, and moreover study the cohomology theories associated to these two homology theories.
The definitions of the AW-homology and DW-homology are based on an important fact that any orbifold admits an appropriate triangulation with respect to the natural weight function induced from the orbifold structure.
Then inspired by the definition of weighted simplicial homology from Dawson~\cite{Daw90}, we encode the orders of the local groups of an orbifold in the boundary maps of the AW-homology and DW-homology so that these two homology theories may reflect the structural information of the singular set of the orbifold. Moreover, we observe that the
weighted simplicial homology in~\cite{Daw90} is invariant under barycentric subdivisions of a weighted simplicial complex with respect to some natural rule. This makes us believe that AW-homology and DW-homology are invariants of orbifolds and worthy of study.
\n

Another key fact supporting the two new homology theories is that the weighted simplicial chain complexes of a divisibly weighted simplex and its barycentric subdivision are both acyclic. So divisibly weighted simplices play the analogous role in weighted simplicial homology as the ordinary simplices do in the ordinary simplicial homology.  
This analogy allows us to prove 
a generalized version of the \emph{simplicial approximation theorem} for (weight-preserving)
morphisms in the category of divisibly weighted simplicial complexes using the algebraic method of acyclic carrier (see Munkres~\cite[Section\,13]{Munk84}). Based on this, we can prove the invariance of AW-homology and DW-homology
under isomorphisms of orbifolds and more generally under certain type of homotopy equivalences of orbifolds.
  \n

 Moreover, the definitions of AW-homology and DW-homology can be extended to a much wider class of objects called \emph{weighted polyhedra}. Roughly speaking, a weighted polyhedron is a topological space $X$ with a weight at each point, and $X$ admits a
 triangulation consisting of (divisibly) weighted simplices such that the weight of $X$ at any point $x$ agrees with
 the weight of the \emph{carrier} of $x$ (i.e. the unique simplex containing $x$ in its relative interior).
 We will develop the whole theory of
 AW-homology and DW-homology in the category
 of weighted polyhedra. Our work can be considered as a natural generalization of the ordinary simplicial homology theory. 
 
  \n

  The paper is organized as follows. In Section~\ref{Sec:Preliminaries}, we will review some basic definitions and constructions in the theories of orbifolds and weighted simplicial homology. In addition, we will introduce a new notion called descending weighted simplicial complex and compare it with the usual weighted simplicial complex. In Section~\ref{Sec:Homology-Theory}, we introduce the definitions of AW-homology and DW-homology of a compact orbifold.
  In Section~\ref{Sec:Pseudo-Orbi}, we define the notion of weighted polyhedron and extend the definitions
  of AW-homology and DW-homology to all weighted polyhedra.  In Section~\ref{Sec:Div-Wei-Simplex}, we prove some basic properties of a divisibly weighted simplex.
  In Section~\ref{Sec:Invariance}, we prove that AW-homology and DW-homology groups are invariants of  weighted polyhedra under isomorphisms and more generally under certain type of homotopy equivalences of weighted polyhedra.
  In Section~\ref{Sec:Relation}, we prove some basic relations between AW-homology, DW-homology and
  the ordinary simplicial homology under some special coefficients. 
  It turns out that the free abelian part of AW-homology and DW-homology of a weighted polyhedron is isomorphic to that of the ordinary simplicial homology. So it is the torsion part of these two homology theories
 that can really give us some new information of a weighted polyhedron.
  In Section~\ref{Sec:Example}, we compute
  AW-homology and DW-homology of some concrete examples such as circles, disks and compact surfaces with isolated singular points. In Section~\ref{Sec:Product-Cohomology}, we construct a product on the DW-cohomology groups
  with integral coefficients and prove that it is graded commutative (see Theorem~\ref{Thm:Graded-Comm-Cup}). This construction generalizes the cup product in the ordinary simplicial cohomology.
  In addition, we can define a weighted cap product between DW-homology classes and DW-cohomology
  classes, which generalizes the ordinary cap product. But we remark that there is no parallel product structures in AW-cohomology. In Section~\ref{Sec:Poincare-Duality}, we study some relations between AW-(co)homology and DW-(co)homology groups.
  We find that
  when a weighted polyhedron is based on a compact orientable homology $n$-manifold,  
 its AW-(co)homology and DW-(co)homology will satisfy
 some duality relations. These relations generalize the classical Poincar\'e duality between the homology and cohomology groups of a compact orientable manifold. This is another reason why we insist on studying AW-homology and DW-homology at the same time.\n
 
  Our future direction of study is to build the theory of 
  weighted singular homology for weighted polyhedra
  that corresponds to AW-homology or DW-homology.
 Another direction is to compute the AW-homology and DW-homology of some higher dimensional orbifolds that arises naturally in mathematics. In addition,
   we want to know whether the AW-homology of an orbifold is isomorphic to the t-singular homology introduced in~\cite{TakYok06}.
  From their definitions, it is hard to see any relation between the two theories.
  But the calculation in Example~\ref{Exam:surface-sing} shows that for a compact $2$-orbifold with only isolated singular points, the AW-homology group and $t$-singular homology
 actually agree (see Remark~\ref{Rem:Agree}).   
 \n

  \section{Preliminaries} \label{Sec:Preliminaries}
  
  \subsection{Definition of orbifolds}
  \ \n  
  
   We first briefly review some basic definitions concerning orbifolds. The reader is referred to~\cite{Sa56,Sa57,AdemLeiRuan07,Choi12} for more detailed discussion of these definitions and some other notions 
  in the study of orbifolds.\n
  
   Let $M$ be a paracompact Hausdorff space.      
   \begin{itemize}
 \item An \emph{orbifold chart} on $M$ is given by a connected open subset $\tilde{U}$ of $\R^n$ for some integer $n\geq 0$, a finite group $G$ acting smoothly and effectively on $\tilde{U}$, and a map $\varphi: \tilde{U}\rightarrow M$, such that $\varphi$ is $G$-invariant ($\varphi\circ g=\varphi$ for all $g\in G$) and induces a homeomorphism from $\tilde{U}\slash G$ onto an
 open subset $U=\varphi(\tilde{U})$ of $M$.\n 
 
 \item An \emph{embedding} $\lambda: (\tilde{U},G,\varphi)\hookrightarrow (\tilde{V},H,\psi)$ between two orbifold charts on $M$ is a smooth embedding $\lambda: \tilde{U}\hookrightarrow \tilde{V}$ with $\psi\circ\lambda = \varphi$. \n

 \item An \emph{orbifold atlas} on $M$ is a family 
 $\mathcal{U}=\{ (\tilde{U},G,\varphi) \}$ of such charts, which cover $M$ and are locally compatible in the following sense: given any two orbifolds charts $(\tilde{U},G,\varphi)$ for $U=\varphi(\tilde{U})\subseteq M$ and 
 $(\tilde{V}, H, \psi)$ for $V=\psi(\tilde{V})\subseteq M$, and any point $x\in U\cap V$, there exists an open neighborhood $W\subseteq U\cap V$ of $x$ and a chart
 $(\tilde{W}, K,\chi)$ with $W=\chi(\tilde{W})$ such that there are
 embeddings $(\tilde{W}, K,\chi)\hookrightarrow
 (\tilde{U},G,\varphi)$ and $(\tilde{W}, K,\chi)\hookrightarrow (\tilde{V}, H, \psi)$. Since $M$ is paracompact, we can choose an orbifold atlas $\mathcal{U}=\{ (\tilde{U},G,\varphi) \}$ on $M$
 such that $\{\varphi(\tilde{U})\}$ is a 
 locally finite open cover of $M$.\n

 \item An atlas $\mathcal{U}$ on $M$ is said to \emph{refine} another
 atlas $\mathcal{V}$ if for every chart in $\mathcal{U}$
 there exists an embedding into some chart of $\mathcal{V}$.\n
 
 \item Two orbifold atlases on $M$ are said to be 
 \emph{equivalent} if they have a common refinement. 
 An \emph{orbifold} (of dimension $n$) is such a space $M$ with an equivalence class of atlases $\mathcal{U}$.\n
 \end{itemize}
 
  We will generally write $\mathcal{M} = (M,\mathcal{U})$ for the orbifold represented by the space $M$ and a chosen atlas $\mathcal{U}$ and say that $\mathcal{M}$ is \emph{based on} $M$. In addition, we call $\mathcal{M}$ \emph{compact} if $M$ is compact. 
   By the fact that a smooth action is locally smooth (see Bredon~\cite[p.\,308]{Bredon72}), any
orbifold of dimension $n$ has an atlas consisting of ``linear'' charts, i.e. charts of
the form $(\R^n, G,\varphi)$ where the finite group $G$ acts on $\R^n$ via orthogonal linear transformations.
\n 
   Let $\mathcal{M} = (M,\mathcal{U})$ be an orbifold of dimension $n$ where $\mathcal{U}$ is an atlas consisting of linear charts. For each point $x \in M$, choose a linear chart $(\R^n, G,\varphi)$ around $x$, with $G$ a finite subgroup of the orthogonal
group $\mathrm{O}(n,\R)$. Let $\tilde{x}$ be a point
with $\varphi(\tilde{x})=x$, and $G_x=\{ g\in G\,|\, g\cdot \tilde{x} = \tilde{x}\}$ the \emph{isotropy subgroup} 
at $\tilde{x}$. Up to conjugation, $G_x$ is a well defined
subgroup of $\mathrm{O}(n,\R)$, called the \emph{local group} at $x$. \n

 \begin{itemize}
  \item The order $|G_x|$ of $G_x$ is independent on the local chart around $x$.\n
  
  \item The set 
  $\{ x \in M\, | \, G_x \neq 1 \}$ is called the \emph{singular set} of $M$, denoted by $\Sigma M$.
  A point in $\Sigma M$ is called a \emph{singular point}, and a point in $M\backslash \Sigma M$ is called
  a \emph{regular point}.
 \n
  
  \item The space $M$ carries a \emph{natural
stratification} whose strata are the connected components of the sets
 $\Sigma_H (M) := \{ x\in M\,|\, (G_x)=(H)\}$,
  where $H$ is any finite subgroup of $\mathrm{O}(n,\R)$
  and $(H)$ is its conjugacy class.
    \end{itemize}
    \n
    
 \begin{exam}[Good Orbifold] \label{Exam:Good-Orbifold}
   If a discrete group $\Gamma$ acts smoothly and properly discontinuously on a manifold $X$, then the orbit space $X\slash \Gamma$ has a natural orbifold structure where the local group of an orbit $\overline{x}=\Gamma x$, $x\in X$, is isomorphic to the isotropy group $\Gamma_x$ of $x$.
Such an orbifold $X\slash \Gamma$ is usually called \emph{good} (or \emph{developable}). An orbifold is called \emph{bad} if it is not good. 
\end{exam}

\begin{exam}[Effective Quotient] \label{Exam:Eff-Quot}
  If a compact Lie group $G$ acts smoothly, effectively, and almost freely on a manifold $X$, then the orbit space
$X\slash G$ has a natural orbifold structure due to 
the ``slice theorem'' (see~\cite[Ch.~II]{Bredon72}). 
The orbifold $X\slash G$ is also referred to as an \emph{effective quotient} (see~\cite[Definition 1.7]{AdemLeiRuan07}).
 \end{exam}
 
 \begin{defi}[Isomorphism of Orbifolds] \label{Def:orbi-isom}
  Two orbifolds $\mathcal{M} = (M,\mathcal{U})$ and 
  $\mathcal{N} = (N,\mathcal{V})$ are called \emph{isomorphic}
  if there exists a homeomorphism $f: M\rightarrow N$ such that:
  for any $x\in M$, there are orbifold charts $(\tilde{U},G,\varphi)$ around $x$ and $(\tilde{V},H,\psi)$ around $y=f(x)$
  where $f$ maps $U=\varphi(\tilde{U})$ onto $V=\psi(\tilde{V})$ and can be lifted to an equivariant homeomorphism 
  $\tilde{f}:\tilde{U}\rightarrow \tilde{V}$, i.e. there exists 
  a group isomorphism $\rho: G\rightarrow H$ so that
  $\tilde{f}(g\cdot z) = \rho(g)\cdot \tilde{f}(z)$ for all
  $z\in \tilde{U}$ and $g\in G$. Such a homeomorphism $f$ is called an \emph{isomorphism} from $\mathcal{M}$ to $\mathcal{N}$, denoted by $f: \mathcal{M}\rightarrow \mathcal{N}$.
 \end{defi} 
  
   \begin{defi}\label{Def:Product-Orbi}
    The \emph{product of two orbifolds} $\mathcal{M}=(M,\mathcal{U})$ and $\mathcal{N}=(N,\mathcal{V})$, denoted by $\mathcal{M}\times \mathcal{N} = (M\times N,\mathcal{U}\times \mathcal{V})$, is an orbifold whose underlying space is $M\times N$ and atlas given by
    $\mathcal{U}\times \mathcal{V}=\{ (\tilde{U}\times\tilde{V},G\times H,\varphi\times\psi) \}$ where
$\mathcal{U}=\{ (\tilde{U},G,\varphi) \}$ and
$\mathcal{V}=\{ (\tilde{V},H,\psi) \}$.   
   \end{defi}
  
   Note that the projection $p: M\times N\rightarrow M$, $(x,y)\mapsto x$ is not a weight-preserving map unless $\mathcal{N}$ has no singular points.  \n

\subsection{Homology of weighted simplicial complexes}
\label{Subsec:Weigh-Simp-Hom}
 \ \n
 
  Weighted simplicial complexes and weighted homology groups were first studied by 
  Dawson in~\cite{Daw90}, which were used to construct nonstandard homology theories for categories of a combinatorial nature, such as preconvexity spaces. In the recent years, many interesting applications of weighted simplicial homology are found in computational topology and topological data analysis; see Ren, Wu and Wu~\cite{RenWuWu18, RenWuWu21},
  Wu, Ren, Wu and Xia~\cite{WuRenWuXia20} and
Baccini, Geraci and Bianconi~\cite{BGB22}.\n
  
  If not specified otherwise, the coefficients of homology groups in our following discussions are 
 always assumed to be integers $\Z$.\n
 
 A (positively) \emph{weighted simplicial complex}
 is a pair $(K, w)$ where $K$ is a simplicial complex and  $w$ is a function assigning a positive integer to each simplex of $K$, which satisfy: for any simplices $\sigma, \sigma'$ in $K$,
  \[ \sigma'\ \text{is a face of}\ \sigma \Rightarrow w(\sigma')\, |\, w(\sigma).  \]
  The function $w$ is called a \emph{weight} on $K$.
  In the following, we also use the vertex set of a simplex $\sigma$ to denote the simplex.
  \n

  For each $n\geq 0$, let $C_n(K)$ denote the free abelian group generated by all the oriented $n$-simplices of $K$.
  By abuse of notation, we will use the symbol $\sigma$ to denote a simplex or an oriented simplex in different occasions.
  Then the weight function $w$ determines a
  boundary operator $\partial^w$ on $C_*(K)$ by:
  \begin{equation} \label{Equ:weighted-Boun-AW}
  \partial^{w}: C_n(K)\rightarrow C_{n-1}(K), \ \  \sigma \mapsto \sum^n_{j=0} \frac{w(\sigma)}{w(\partial_j(\sigma))} (-1)^j \partial_j \sigma 
   \end{equation}
  where the face operators $\partial_j$ are defined as in the standard simplicial homology, i.e.
 for an oriented simplex $\sigma=[v_0,\cdots,v_n]$, $\partial_j\sigma=[v_0,\cdots,\widehat{v}_j,\cdots,v_n]$. It is easy to check that $$\partial^w\circ\partial^w=0.$$ 
  Then $(C_*(K), \partial^w)$ is called the \emph{weighted simplicial chain complex} of $(K,w)$, whose
    homology group is called the \emph{weighted simplicial homology} of $(K,w)$, denoted by $H_*(K,\partial^w)$.
 \n
 
Moreover, for a simplicial subcomplex $A$ of $K$, 
the map $\partial^w$ induces a boundary 
map $\overline{\partial}^w$ on the relative chain complex
$C_*(K,A)=C_*(K)\slash C_*(A)$. Then similarly, we can define the \emph{relative weighted simplicial homology} 
$H_*(K,A, \overline{\partial}^w)$ as
\[ H_*(K,A, \overline{\partial}^w):= H_*(C_*(K,A),\overline{\partial}^w). \] \n

 Given two weighted simplicial complexes $(K,w)$ and $(K',w')$, a simplicial map $\varrho: K\rightarrow K'$ is called a \emph{morphism} from $(K,w)$ to $(K',w')$
   if
   \begin{equation} \label{Equ:Mor-Ascend}
   w'(\varrho(\sigma)) \mid w(\sigma), \ \text{for every
   simplex $\sigma$ of $K$}.
   \end{equation} 
   We also denote such a map by
   $\varrho: (K,w)\rightarrow (K',w')$.
    In particular, $\varrho$ is called \emph{weight-preserving} if $w'(\varrho(\sigma))=w(\sigma)$ for any simplex
   $\sigma$ in $K$.
      It is easy to check that a morphism $\varrho$ induces a chain map 
   \[ \varrho_{\#}:
   (C_*(K),\partial^w)\rightarrow (C_*(K'), \partial^{w'}),\ \  \varrho_{\#}(\sigma) = \frac{w(\sigma)}{w'(\varrho(\sigma))} \varrho(\sigma), \ \forall\sigma\in K, \]
  Then $\varrho$ further induces a homomorphism on the weighted homology, denoted by 
  $$ \varrho_*:  H_*(K,\partial^w)\rightarrow H_*(K',\partial^{w'}).$$
   The reader is referred to~\cite{Daw90} for more discussion of
 the categorical properties of weighted simplicial complexes and weighted homology groups.\n

 \subsection{Descending weighted simplicial complexes}
\label{Subsec:Descend-Weigh-Simp-Hom}
 \ \n
 
 In the definition of weighted simplicial complex, the weights of the simplices are ascending when their dimensions go up. But we find that another type of weight functions on simplicial complexes defined below are more suitable for the study of orbifolds. 
 
\begin{defi}
 A \emph{descending weighted simplicial complex}
 is a pair $(L, \overline{w})$ where $L$ is a simplicial complex and $\overline{w}$ is a function assigning a positive integer to each simplex of $L$, which satisfy: for any simplices $\sigma, \sigma'$ in $L$,
  \[ \sigma'\ \text{is a face of}\ \sigma \Rightarrow \overline{w}(\sigma)\, |\, \overline{w}(\sigma').  \]
  
  We call $\overline{w}$ a \emph{descending weight} on $L$. There is a natural boundary operator $\partial^{\overline{w}}: C_n(L)\rightarrow C_{n-1}(L)$, $n\geq 0$, defined by:
  \begin{equation} \label{Equ:weighted-Boun-DW}
   \partial^{\overline{w}}: \sigma \longmapsto \sum^n_{j=0} \frac{\overline{w}(\partial_j(\sigma))}{\overline{w}(\sigma)} (-1)^j \partial_j \sigma. 
   \end{equation}
  It is easy to prove $\partial^{\overline{w}}\circ\partial^{\overline{w}}=0$.  We call $(C_*(L), \partial^{\overline{w}})$ the
  \emph{weighted simplicial chain complex} of $(L,\overline{w})$, whose homology group is called the \emph{weighted simplicial homology group} of $(L,\overline{w})$, denoted by 
  $H_*(L,\partial^{\overline{w}})$.
 \n
  \end{defi}
  
  A \emph{morphism} from $(L,\overline{w})$ to $(L',\overline{w}')$ is defined to be a simplicial map $\overline{\varrho}: L\rightarrow L'$ that satisfies:
    \begin{equation} \label{Equ:Mor-Descend}
      \overline{w}(\sigma) \mid \overline{w}'(\overline{\varrho}(\sigma)), \ \text{for every
   simplex $\sigma$ of $L$}. 
   \end{equation}
   In particular, $\overline{\varrho}$ is called \emph{weight-preserving} if $\overline{w}'(\overline{\varrho}(\sigma))=\overline{w}(\sigma)$ for any simplex $\sigma$ of $L$.
     It is easy to see that a morphism
   $\overline{\varrho}$ induces a chain map 
   \[ \overline{\varrho}_{\#}:
   \big(C_*(L),\partial^{\overline{w}} \big)\rightarrow \big(C_*(L'), \partial^{\overline{w}'}\big), \ \text{where}\
    \overline{\varrho}_{\#}(\sigma) = \frac{\overline{w}'(\overline{\varrho}(\sigma))}{\overline{w}(\sigma)} \overline{\varrho}(\sigma), \ \forall\sigma\in L. \] 
  Then $\overline{\varrho}$ further induces a homomorphism on the homology
  groups, denoted by 
   $$ \overline{\varrho}_*:  H_*(L,\partial^{\overline{w}})\rightarrow H_*(L',\partial^{\overline{w}'}).$$
  
  \begin{rem}
    The definition~\eqref{Equ:Mor-Descend} suggests that descending weighted simplicial complexes are more suitable for the study of orbifolds. 
 Indeed, suppose a discrete group $G$ acts smoothly and properly discontinuously on manifolds $X$ and $Y$. Then any equivariant map $\varphi : X\rightarrow Y$ induces a map on their orbit spaces, denoted by $ \overline{\varphi}: X/G\rightarrow Y/G$ (see Example~\ref{Exam:Good-Orbifold}).
 It is natural to define the weight of an orbit $\overline{x}\in X\slash G$ (or $\overline{y}\in Y\slash G$) to be $|G_x|$ (or $|G_y|$).
 Then since for any $x\in X$, the isotropy group $G_x$ is a subgroup of $G_{\varphi(x)}$,
 the weight of $\overline{x}$ always divides
 the weight of $\overline{\varphi}(\overline{x})=\overline{\varphi(x)}$, which corresponds to the condition in~\eqref{Equ:Mor-Descend}.  
  \end{rem}
   
   The following are some conventions that will be used in the rest of the paper.
\begin{itemize}
\item[\textbf{(CV-1)}] A simplicial complex is always assumed to be finite (i.e. having only finitely many vertices) if not specified otherwise.\n

\item[\textbf{(CV-2)}] To avoid ambiguity, we call a weighted simplicial complex $(K,w)$ defined by Dawson~\cite{Daw90} an \emph{ascending weighted simplicial complex} and call $w$ an \emph{ascending weight}.
\n

\item[\textbf{(CV-3)}] When we say $(K,\mu)$ is a weighted simplicial complex, $\mu$ could be either an ascending weight or a descending one.\n
 
 \item[\textbf{(CV-4)}] We say that an ascending weighted simplicial complex and a descending weighted one are of \emph{opposite type}, so do we call their weight functions.\n
 
  \item[\textbf{(CV-5)}] If $(K,\mu)$ is a weighted simplicial complex and $L$ is a subcomplex of $K$, we also use $(L,\mu)$ to denote the weighted simplicial complex
  $(L,\mu|_L)$ and, use $H_*(L,\partial^{\mu})$
  instead of the cumbersome notation $H_*(L,\partial^{\mu|_L})$ to denote its weighted homology groups.  
  \end{itemize}
 \n
 For a weighted simplicial complex $(K,\mu)$, let
 \begin{equation}\label{Equ:N-mu}
   N_{\mu} := \text{the least common multiple of all the weights $\mu(\sigma)$}, \ \sigma\in K. 
   \end{equation}
 
 \begin{lem} \label{Lem:Adjoint-Weight}
 For any weighted simplicial complex $(K,\mu)$, there exists a  weighted simplicial complex $(K,\mu^*)$
 of opposite type such that
 \begin{equation} \label{Equ:Homology-Adjoint-Equal}
   H_i(K,\partial^{\mu^*}) \cong H_i(K,\partial^{\mu}) \ \text{for all}\ i \geq 0.
   \end{equation}
  \end{lem}
  \begin{proof}
 Define another weight $\mu^*$ on $K$ by 
\begin{equation} \label{Equ:Adjoint-Weight}
 \mu^*(\sigma) = \frac{N_{\mu}}{\mu(\sigma)}, \ \forall \sigma\in K.
 \end{equation}
 
 By the definitions of $\partial^{\mu}$ and $\partial^{\mu^*}$, $(C_*(K),\partial^{\mu})$ and $(C_*(K),\partial^{\mu^*})$ are in fact isomorphic chain complexes. So
 $H_*(K,\partial^{\mu^*}) \cong H_*(K,\partial^{\mu})$. 
  \end{proof}
 
 \begin{defi}[Adjoint of a Weight]\label{Def:Adjoint-Wt}
 For a weighted simplicial complex $(K,\mu)$, we call the weight function $\mu^*$ defined in~\eqref{Equ:Adjoint-Weight} the \emph{adjoint} of $\mu$. 
  \end{defi}
 
 By Lemma~\ref{Lem:Adjoint-Weight}, descending weighted simplicial complexes can be considered as mirror objects of
 ascending weighted ones. But not all properties of these two categories of objects are the same.
 We will see later in Section~\ref{Sec:Product-Cohomology} that there is a natural
   product structure on the weighted cohomology groups with integral coefficients on a  descending weighted simplicial complex.
  But there is no such a product for an ascending weighted simplicial complex. 
  \n

 \subsection{Augmentation of weighted simplicial complex}
 \label{Subsec:Augmentation}
 \ \n
 A (abstract) chain complex $\mathcal{C}_*=(C_p,\partial_p)$ is called \emph{non-negative} if $C_p=0$ for all $p<0$.
 In this paper, all chain complexes are assumed to be non-negative if not specified otherwise.\n
 
   The \emph{augmentation} of a non-negative chain complex 
   $\mathcal{C}_* =\{ \partial_n: C_n\rightarrow C_{n-1} \}_{n\geq 1}$ is a map $\varepsilon: C_0\rightarrow \Z$ such that $\varepsilon\circ \partial_1=0$. The \emph{augmented chain complex} $(\mathcal{C}_*,\varepsilon)$ is the
chain complex obtained from $\mathcal{C}_*$ by adjoining the group $\Z$ in dimension $-1$ and
using $\varepsilon$ as the boundary operator in dimension $0$. \n

 Let $(\mathcal{C}_*,\varepsilon)$ and $(\mathcal{C}'_*,\varepsilon')$ be two augmented chain complexes.
 We call a chain map $\phi = \{ \phi_n\}: \mathcal{C}_* \rightarrow \mathcal{C}'_*$ \emph{augmentation-preserving} if $\epsilon' \circ \phi_0 = \epsilon$.
 \n

   Let $(K,\mu)$ be a weighted simplicial complex.
   We define an augmentation of $(C_*(K),\partial^{\mu})$, denoted by $\varepsilon^{\mu}: C_0(K,\partial^{\mu})\rightarrow \Z$, as follows:
   \begin{itemize}
   \item If $\mu$ is ascending, define
   \begin{equation} \label{Equ:Aug-1}
      \varepsilon^{\mu}(v_i) := \mu(v_i) \ \text{for any vertex} \ v_i \ \text{of}\ K. 
     \end{equation}

     \item If $\mu$ is descending, its adjoint $\mu^*$ is an ascending weight on $K$. Define
     \begin{equation} \label{Equ:Aug-2}
      \varepsilon^{\mu}(v_i) := \mu^*(v_i) = \frac{N_{\mu}}{\mu(v_i)} \ \text{for any vertex} \ v_i \ \text{of}\ K.
     \end{equation}  
     \end{itemize}  
   The reason why we use $N_{\mu}$ as the numerator here is that we want to use integral coefficients for our weighted homology.\n
  It is easy to check that the above defined $ \varepsilon^{\mu}$
  is an augmentation on $(C_*(K),\partial^{\mu})$ in both cases. Then we call the homology of the augmented chain complex
  $\big( C_*(K),\partial^{\mu}, \varepsilon^{\mu} \big)$
  the \emph{reduced weighted simplicial homology} of
  $(K,\mu)$, denoted by $\widetilde{H}_*(K,\partial^{\mu})$.
  The following lemma is easy to see from our definitions.
  
  \begin{lem} \label{Lem:Aug-ext}
  Let $(K,\mu)$ and $(K',\mu')$ be two weighted simplicial complexes of the same type.  If $\varphi: (K,\mu)\rightarrow (K',\mu')$ is a morphism, then $\varphi$ extends to an
  augmentation-preserving chain map $\widetilde{\varphi}$ between the augmented chain complexes
  $\big( C_*(K),\partial^{\mu},  \varepsilon^{\mu} \big)$
  and $\big( C_*(K'),\partial^{\mu'},  \varepsilon^{\mu'} \big)$ where at degree $-1$, $\widetilde{\varphi}: \Z\rightarrow \Z$ is defined by: for any $n \in\Z$,
  $$\widetilde{\varphi}(n) = 
  \begin{cases}
   n ,  &  \text{if $(K,\mu)$ and $(K',\mu')$ are both ascending}; \\
   \frac{N_{\mu'}}{N_{\mu}}\cdot n,  &  \text{if $(K,\mu)$ and $(K',\mu')$ are both descending}.
 \end{cases} $$
  \end{lem}
  
  \n
  
  \subsection{Cartesian product of weighted simplicial complexes}
  \ \n
  
   Let $(K,\mu)$ and $(K',\mu')$ be two weighted simplicial complexes and let
  $$v_1\prec\cdots\prec v_m, \ \ v'_1\prec \cdots \prec v'_{n}$$
  be some total orderings of the vertices of $K$ and
  $K'$, respectively. Then the Cartesian product $K\times K'$ is a simplicial complex whose vertex set is 
  $$\{ (v_i,v'_j)\,|\, 1\leq i \leq m, 1\leq j \leq n \}.$$
  Moreover, all the simplices in $K\times K'$ are of the form  
  \begin{equation} \label{Equ:Product-Simplices}
    \{ (v_{i_1},v'_{j_1}),\cdots, (v_{i_s},v'_{j_s}) \},
  \ v_{i_1}\preccurlyeq \cdots \preccurlyeq v_{i_s},\ v'_{j_1}\preccurlyeq \cdots \preccurlyeq v'_{j_s},
  \end{equation}
  where $\{v_{i_1},\cdots, v_{i_s}\}$ is a simplex in $K$
  and $\{ v'_{j_1},\cdots, v'_{j_s} \}$ is a simplex in $K'$.
  \n
  
 \begin{defi}[Cartesian product of weighted simplicial complexes]\label{Def:Product-Wt-Complexes}
   Let $(K,\mu)$ and $(K',\mu')$ be weighted simplicial complexes which are both ascending or both descending. 
  The \emph{Cartesian product} of $(K,\mu)$ and $(K',\mu')$ with respect to some total orderings of the vertices of $K$ and $K'$
   is a weighted simplicial complex $(K\times K',\mu\times \mu')$, where for a simplex
  $\{(v_{i_1},v'_{j_1}),\cdots, (v_{i_s},v'_{j_s})\}$ of
  $K\times K'$,
  \begin{equation} \label{Equ:Product-Weight}
   \mu\times \mu'\big( \{(v_{i_1},v'_{j_1}),\cdots, (v_{i_s},v'_{j_s})\} \big):= \mu(\{v_{i_1},\cdots, v_{i_s}\})\cdot \mu'(\{v'_{j_1},\cdots, v'_{j_s}\}).  
   \end{equation}
   It is easy to see that
  $(K\times K',\mu\times \mu')$ has the same type as $(K,\mu)$ and $(K',\mu')$.  
  \end{defi}
  
  Notice that the simplicial complex structure of $K\times K'$ depends on the ordering of vertices of $K$ and $K'$, so does $(K\times K',\xi\times \xi')$. But 
  we usually omit the ordering of vertices in our notation.
  
  \begin{rem}
   There is another meaningful weight function $  \overline{\mu\times \mu'}$ on $K\times K'$ defined by:
   for each simplex
  $\{(v_{i_1},v'_{j_1}),\cdots, (v_{i_s},v'_{j_s})\}$ of
  $K\times K'$ as in~\eqref{Equ:Product-Simplices},
    $$ \overline{\mu\times \mu'} (\{(v_{i_1},v'_{j_1}),\cdots, (v_{i_s},v'_{j_s})\}) := \mathrm{lcm}\big(\mu(\{v_{i_1},\cdots, v_{i_s}\}), \mu'(\{v'_{j_1},\cdots, v'_{j_s}\}) \big),$$
    where ``lcm'' is the abbreviation for least common multiple. In~\cite{Daw90}, $(K\times K',\overline{\mu\times \mu'})$ is called the
  product of $(K,\mu)$ and $(K',\mu')$ (see~\cite[Proposition 1.1]{Daw90}).
  But the Cartesian product in Definition~\ref{Def:Product-Wt-Complexes} is a more suitable notion for our study of orbifolds. The reason is that the local group of
  the product of two orbifolds $\mathcal{M}=(M,\mathcal{U})$ and $\mathcal{N}=(N,\mathcal{V})$ at a point $(x,y) \in M\times N$ is the product of the local group
  of $\mathcal{M}$ at $x$ and the local group of $\mathcal{N}$ at $y$. Therefore, the weight of $(x,y)$ in the product orbifold is the product (not the least common multiple) of the weights of $x$ in $\mathcal{M}$ and $y$ in $\mathcal{N}$.
  \end{rem}

  \begin{prop} \label{Prop:K-I-Product}
   Let $(K,\mu)$ be a weighted simplicial complex. Then
   there is an isomorphism 
    $H_*(K,\partial^{\mu}) \cong H_*(K\times [0,1],  \partial^{\mu\times\mathbf{1}})$.
    \end{prop}
 \begin{proof}
 Let $i_0: K\hookrightarrow K\times \{0\}\subset K\times [0,1]$ be the inclusion and  $p: K \times [0,1]\rightarrow K$ be the projection.
  It is easy to check that they induce chain maps
   $$(i_0)_{\#}:  \big(C_*(K),\partial^{\mu}\big) \rightarrow
   \big(C_*(K\times [0,1]),
    \partial^{\mu\times\mathbf{1}}\big),$$
   $$\ \ \ p_{\#}: \big( C_*(K\times [0,1]), \partial^{\mu\times\mathbf{1}} \big) \rightarrow \big(C_*(K),\partial^{\mu} \big).  $$
    
    \textbf{Claim:} $(i_0)_{\#}$ and $p_{\#}$ are chain homotopy inverse of each other.\n
    
    First of all, clearly $p_{\#}\circ (i_0)_{\#}=\mathrm{id}_{C_*(K)}$. It remains to prove that $(i_0)_{\#}\circ p_{\#}$ is chain homotopic to the identity map of $C_*(K\times [0,1])$.    
    Let the vertices of $K\times [0,1]$ be
    $$ \text{$v_1,\cdots, v_m \in K\times\{0\}$; \
    $v'_1,\cdots, v'_m\in K\times \{1\}$ where
    $p(v'_j)=v_j$, $1\leq j \leq m$.}$$ Then
    any simplex of $K\times [0,1]$ can be written as one of the following forms:
    $$\{ v_{j_1},\cdots, v_{j_r} \}\in K\times\{0\} , \ \ 
   \{v'_{j_1},\cdots, v'_{j_r} \} \in K\times \{1\} ,
     \ 1\leq j_1< \cdots  < j_r \leq m;$$
    $$\{ v_{j_1},\cdots, v_{j_s}, v'_{j_{s+1}},\cdots, v'_{j_r} \}, \ 1\leq j_1< \cdots < j_s \leq j_{s+1} <\cdots< j_r \leq m $$
    where $\{v_{j_1},\cdots, v_{j_r}\}$ is a simplex in $K$.
    Note that this is a simplified way to write the simplices of $K\times [0,1]$ given by~\eqref{Equ:Product-Simplices}.\n
    
    Note that $i_0$ and $p$ are both weight-preserving by the definition of $\mu\times\mathbf{1}$ in~\eqref{Equ:Product-Weight}.    
  Moreover,
  $$i_0\circ p\,(\{v_{j_1},\cdots, v_{j_r} \}) = 
  i_0\circ p\,(\{v'_{j_1},\cdots, v'_{j_r} \}) = \{v_{j_1},\cdots, v_{j_r} \};$$
   $$i_0\circ p\, (\{v_{j_1},\cdots, v_{j_s}, v'_{j_{s+1}},\cdots, v'_{j_r}\}) = \begin{cases}
   \{v_{j_1},\cdots,v_{j_r}\},  &  \text{if $j_s<j_{s+1}$}; \\
   \{v_{j_1},\cdots, v_{j_s}, v_{j_{s+2}},\cdots,v_{j_r}\},  &  \text{if $j_s=j_{s+1}$}.
 \end{cases} 
 $$
    So the chain map $(i_0)_{\#}\circ p_{\#}$ vanishes at $[v_{j_1},\cdots, v_{j_s}, v'_{j_s},\cdots, v'_{j_r}]$
    because of the dimension reason.  \n

    Next, define $\big\{ P_n : 
  C_n (K\times [0,1])
  \rightarrow  C_{n+1}(K\times [0,1]) \big\}_{n\geq 0}$ by 
  $$ P_{r-1}([v_{j_1},\cdots, v_{j_r}])= 0; $$
  $$P_{r-1}([v'_{j_1},\cdots, v'_{j_r}]) = \sum^r_{t=1} (-1)^{t-1} [v_{j_1},\cdots, v_{j_t}, v'_{j_t},\cdots, v'_{j_r}], \ r\geq 1;$$
   \begin{align*}
   &\ P_{r-1} ( [v_{j_1},\cdots, v_{j_s}, v'_{j_{s+1}},\cdots, v'_{j_r} ] ) \\
   =& \begin{cases}
   \underset{s+1\leq t\leq r}{\sum} (-1)^{t-1} 
  [v_{j_1},\cdots, v_{j_s},\cdots, v_{j_t}, v'_{j_t},\cdots, v'_{j_r} ],  &  \text{if $j_s < j_{s+1}$}; \\
  \ 0,  &  \text{if $j_s=j_{s+1}$}.
 \end{cases}  
 \end{align*}
 It is routine to check that $\{ P_n\}_{n\geq 0}$ is a chain homotopy between $(i_0)_{\#}\circ p_{\#}$ and the identity map $\mathrm{id}_{C_*(K\times [0,1])}$.  
   \end{proof}

  \n
   
   \subsection{Divisibly weighted simplicial complex}
   \label{Subsec:Div-Weighted-Complex}
   \ \n

  Let $(K,w)$ be an ascending weighted simplicial complex. A simplex $\sigma$ in $K$ is called \emph{divisibly weighted} (see Dawson~\cite{Daw90}) if all the vertices of $\sigma$ can be ordered as
 $\{v_0,\cdots, v_k\}$ such that
   \begin{equation} \label{Equ:Weight-cond}
     w(v_0)\, |\, \cdots\, | \, w(v_k) = w(\sigma).
   \end{equation}
 
 If all the simplices in $(K,w)$ are divisibly weighted, 
  the weight function $w$ is called \emph{divisible}. \n
 
 Similarly, for a descending weighted simplicial complex
  $(L,\overline{w})$ as well, we call a simplex $\sigma$ in $L$  \emph{divisibly weighted} if all the vertices of $\sigma$ can be ordered as
 $\{v_0,\cdots, v_k\}$ such that
   \begin{equation} \label{Equ:Weight-cond-decre}
    \overline{w}(\sigma) = w(v_0)\, |\, \cdots\, | \, w(v_k).
   \end{equation}
 
 If all the simplices in $(L,\overline{w})$ are divisibly weighted, 
  the weight function $\overline{w}$ is also called \emph{divisible}. 
 \n
 
 In general, we call a weighted simplicial complex  
 $(K,\mu)$ \emph{divisibly weighted} if its weight function $\mu$ (either ascending or descending) is 
 divisible.
Clearly, a divisible weight on a simplicial complex $K$ is completely determined by its 
   value on the vertices of $K$.  
  \n
  
 The following convention will be used in the rest of the paper. \n
 
 \begin{itemize}
 \item[\textbf{(CV-6)}] We will use Greek letters $\xi$ and $\eta$ to refer to a divisible weight on a simplicial complex $K$ to distinguish it from a general weight function $\mu$.\n
   \end{itemize}
   
   \begin{lem} \label{Lem:Morph-Vertex}
       Suppose $(K,\xi)$ and $(K',\xi')$ are two divisibly weighted simplicial complexes of the same type.  Then
   a simplicial map $\varphi: K\rightarrow K'$ is a morphism from $(K,\xi)$ to $(K',\xi')$ if and only if
   for every vertex $v$ of $K$,
   $$ \begin{cases}
   \xi'(\varphi(v)) \mid \xi(v),  &  \text{if $\xi$ and $\xi'$ are ascending}; \\
   \xi(v) \mid \xi'(\varphi(v)),  &  \text{if $\xi$ and $\xi'$ are descending}.
 \end{cases} $$    
 In particular, $\varphi$ is weight-preserving if and only if
 $\xi'(\varphi(v))=\xi(v)$ for every vertex $v$ of $K$.
   \end{lem}   
   \begin{proof}
    This follows easily from the definitions of
  divisibly weighted simplicial complex and the morphism of weighted simplicial complex.   
   \end{proof}
   
 \begin{prop}
  \label{Prop:Prod-Div-Weighted-Complex}
     The Cartesian product of two divisibly weighted simplicial complexes of the same type is also divisibly weighted.
  \end{prop}
  \begin{proof}
   Let $(K,\xi)$ and $(K',\xi')$ be two divisibly weighted simplicial complexes of the same type.    
  Let
  $v_1\prec\cdots\prec v_m$ and $v'_1\prec \cdots \prec v'_{n}$
   be total orderings of the vertices of $K$ and
  $K'$, respectively, such that
  \begin{equation} \label{Equ:vertex-order}
     \xi(v_1) \leq \cdots \leq \xi(v_m), \ \ 
    \xi'(v'_1) \leq \cdots \leq \xi'(v'_{n}).
  \end{equation}  
  Any simplex in $K\times K'$ is of the form
  $$\{(v_{i_1},v'_{j_1}),\cdots, (v_{i_s},v'_{j_s})\},\
  v_{i_1}\preccurlyeq \cdots \preccurlyeq v_{i_s},\ v'_{j_1} \preccurlyeq \cdots \preccurlyeq v'_{j_s} $$
  where $\{v_{i_1},\cdots, v_{i_s}\}$ is a simplex of $K$
   and $\{v'_{j_1},\cdots, v'_{j_s}\}$ is a simplex of $K'$.
    Then since $(K,\xi)$ and $(K',\xi')$ are both divisibly weighted, both
    $(\{v_{i_1},\cdots, v_{i_s}\},\xi)$ and 
  $(\{v'_{j_1},\cdots, v'_{j_s}\},\xi')$ are
   divisibly weighted simplices.\n
   
    If $(K,\xi)$ and $(K',\xi')$ are both ascending, then we have
  \begin{align} \label{Equ:Divide-WT}
     \xi(v_{i_1}) \mid \cdots \mid \xi(v_{i_s}) &=\xi(\{v_{i_1},\cdots, v_{i_s}\}), \\ 
  \xi'(v'_{j_1}) \mid \cdots  \mid \xi'(v'_{j_s})&=\xi'(\{v'_{j_1},\cdots, v'_{j_s}\}). \notag 
  \end{align}
   So by Definition~\ref{Def:Product-Wt-Complexes}, the weight $\xi\times \xi'$ on $K\times K'$ is given by
  \begin{align*}
    (\xi\times \xi')\big(\{(v_{i_1},v'_{j_1}),\cdots, (v_{i_s},v'_{j_s})\}\big) &= \xi(\{v_{i_1},\cdots, v_{i_s}\}) \cdot \xi'(\{v'_{j_1},\cdots, v'_{j_s}\})\\
    &= \xi(v_{i_s})\cdot \xi'(v'_{j_s}),
    \end{align*} 
      In particular, 
    $$(\xi\times \xi')((v_i,v'_j))=  \xi(v_i)\cdot\xi'(v'_j), \ 1\leq i\leq m, \ 1\leq j \leq n.$$  
    Then from~\eqref{Equ:Divide-WT}, we obtain
   $$ (\xi\times \xi')((v_{i_1},v'_{j_1})) \mid \cdots \mid
   (\xi\times \xi')((v_{i_s},v'_{j_s})) = (\xi\times \xi')\big(\{(v_{i_1},v'_{j_1}),\cdots, (v_{i_s},v'_{j_s})\}\big), $$
 which means that $\big(\{(v_{i_1},v'_{j_1}),\cdots, (v_{i_s},v'_{j_s})\},\xi\times \xi'\big)$ is a divisibly weighted simplex. So $(K\times K',\xi\times \xi')$
    is a divisibly weighted simplicial complex. \n
    
    Similarly, if $(K,\xi)$ and $(K',\xi')$ are 
    both descending, the weight $\xi\times \xi'$ on $K\times K'$ is given by
  \[ (\xi\times \xi')\big(\{(v_{i_1},v'_{j_1}),\cdots, (v_{i_s},v'_{j_s})\}\big) = \xi(v_{i_1})\cdot\xi'(v'_{j_1}).
    \]
 It is also easy to check that $(K\times K',\xi\times \xi')$ is divisibly weighted.       
  \end{proof}

 \begin{defi}[Inversion of a Divisible Weight] \label{Def:Associated-Weight}
 If $\xi$ is an ascending divisible weight on a simplicial complex 
 $K$, then $\xi$ canonically determines a descending divisible weight
 $\widehat{\xi}$ on $K$ by: for any simplex $\sigma$ in $K$, 
   \begin{equation} \label{Equ:Weight-cond-Des}
    \widehat{\xi}(\sigma)= \xi(v_0)\, |\, \cdots\, | \, \xi(v_k) = \xi(\sigma)
   \end{equation}
   where $\{v_0,\cdots, v_k\}$ is the vertex set of 
    $\sigma$. We call $\widehat{\xi}$
   the \emph{inversion of $\xi$}. \n
   
  Similarly, if $\eta$ is a descending divisible weight on $K$, we obtain an ascending divisible weight
 $\widehat{\eta}$ on $K$ defined by
   \begin{equation} \label{Equ:Weight-cond-Ins}
  \eta(\sigma) =  \eta(v_0)\, |\, \cdots\, | \, \eta(v_k) = \widehat{\eta}(\sigma).
   \end{equation}
     We call $\widehat{\eta}$ the \emph{inversion of $\eta$}. Note that a divisible weight on $K$ always agrees with its inversion on every vertex of $K$.
   \end{defi}
   
      The following lemma is immediate from our definition.
      
       \begin{lem} \label{Lem:Inver-Adjoint}
     For a divisible weight $\xi$ on a simplicial complex,
     the adjoint of the inversion of $\xi$ agrees with the inversion of the adjoint of $\xi$, i.e. $(\widehat{\xi})^* = \widehat{\xi^*}$.
   \end{lem}

  By definition, the inversion of
   a divisible weight $\xi$ is different from its adjoint  (see Figure~\ref{p:Inversion-Adjoint} for example).
   In addition, the inversion of the inversion of $\xi$ always goes back to $\xi$, while the adjoint of the adjoint of $\xi$ may differ from $\xi$ by a constant factor. 
   
    \begin{figure}[h]
         \includegraphics[width=0.75\textwidth]{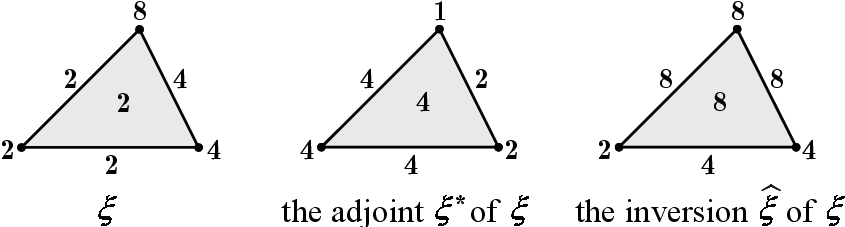}\\
          \caption{The adjoint and the inversion of a divisible weight $\xi$}\label{p:Inversion-Adjoint}
      \end{figure}
   
  A divisible weight $\xi$ on $K$ and its inversion $\widehat{\xi}$ clearly determine each other since they have the same values on the vertices. But generally speaking, the chain complexes $(C_*(K),\partial^{\xi})$ and
  $(C_*(K),\partial^{\widehat{\xi}})$ are not isomorphic,
  neither are the weighted homology groups $H_*(K,\partial^{\xi})$ and
  $H_*(K,\partial^{\widehat{\xi}})$.
Indeed, it is possible that two divisible weighted
simplicial complexes $(K,\xi)$ and $(K',\xi')$ of the same type satisfy $H_*(K,\partial^{\xi})\cong H_*(K',\partial^{\xi'})$ while $H_*(K,\partial^{\widehat{\xi}})\ncong H_*(K',\partial^{\widehat{\xi'}})$ (see Example~\ref{Exam:1-dim-orbifold}). \n

\begin{warn}
If a simplicial map $\varrho: K\rightarrow K'$ determines a morphism from $(K,\xi)$ to $(K',\xi')$, it does not
necessarily induce a morphism from $(K,\widehat{\xi})$ to $
(K',\widehat{\xi'})$.
This is because the morphisms for ascending
vs. descending weighted simplicial complexes satisfy
the opposite conditions (see~\eqref{Equ:Mor-Ascend} and~\eqref{Equ:Mor-Descend}). So
only when the map
 $\varrho$ is weight-preserving can it determine  morphisms $(K,\xi)\rightarrow (K',
  \xi')$ and
  $(K,\widehat{\xi})\rightarrow (K',
  \widehat{\xi'})$ simultaneously.
 \end{warn}

\subsection{Barycentric subdivision of Weighted Simplicial Complex}
   \ \n
   For any weighted simplicial complex $(K,\mu)$, there is a natural divisible weight on the barycentric subdivision $Sd(K)$ of $K$ defined as follows.

 \begin{defi}[Barycentric Subdivision of Weighted Simplicial Complex] \label{Def:Bary-Subdiv-Weighted}
\ \n
 For a weighted simplicial complex $(K,\mu)$, define the weight of the barycenter $b_{\sigma}$ of any simplex $\sigma$ in $K$ to be $\mu(\sigma)$. Since any simplex in $Sd(K)$ is of the form $\{b_{\sigma_0},\cdots, b_{\sigma_l}\}$ where
   $\sigma_0\subsetneq \cdots\subsetneq \sigma_l \in K$,
   we define a weight function $Sd(\mu)$ on $Sd(K)$ by:
    \begin{equation} \label{Equ:Subdiv-w}
     Sd(\mu)\big(\{ b_{\sigma_0},\cdots, b_{\sigma_l}\}\big)=\mu(\sigma_l),\ \sigma_0\subsetneq \cdots\subsetneq \sigma_l \in K.
     \end{equation}
   
   Note that we have either
   $\mu(\sigma_0)\, | \, \cdots \, | \,\mu(\sigma_l)$ or 
    $\mu(\sigma_l) \, |\, \cdots\, |\, \mu(\sigma_0)$ depending on whether $\mu$ is ascending or descending. Then it is easy to 
    check that $Sd(\mu)$ is a divisible weight on $Sd(K)$ which is ascending (or descending) if so is $\mu$.  
     We call $(Sd(K),Sd(\mu))$ the \emph{barycentric subdivision} of $(K,\mu)$.\n
     
  Another useful fact of $(Sd(K),Sd(\mu))$ is that for an $n$-simplex $\sigma$ in $K$, every $n$-simplex $\tau$ in $Sd(\sigma)$ is of the form $\{b_{\sigma_0},\cdots, b_{\sigma_n}\}$ where
$\sigma_0\subsetneq \cdots\subsetneq \sigma_n=\sigma$.
So we always have 
\begin{equation} \label{Equ-Sd-n-simplex}
 Sd(\mu)(\tau) = \mu(\sigma).
 \end{equation}
 
 For any integer $m\geq 1$, let $(Sd^m(K),Sd^m(\mu))$ 
 denote the $m$ iterated barycentric subdivisions of $(K,\mu)$.
  \end{defi}

 \subsection{Contractible weighted simplicial complex}
 \ \n
 
  The following definitions are introduced by Dawson~\cite[p.\,236]{Daw90} for ascending weighted simplicial complexes. But clearly they can be defined for descending weighted simplicial complexes as well.

  \begin{defi}[Contiguous Morphisms] \label{Def:Contigu}
  Suppose $(K,\mu)$ and $(K',\mu')$ are two weighted simplicial complexes of the same type. Two morphisms $$\varrho_0,\varrho_1: (K,\mu)\rightarrow (K',\mu')$$ are called \emph{contiguous}, denoted by $\varrho_0 \underset{c}{\simeq} \varrho_1$,  
  if there exists a morphism
  $$ F: (K \times [0, 1], \mu \times \mathbf{1}) \rightarrow (K',\mu')\  \text{with}\
  F(x, 0) =\varrho_0(x),\, F(x,1)= \varrho_1(x),\ \forall x\in K.$$
   Here $([0,1],\mathbf{1})$ is a weighted simplicial complex where the weight of every 
   simplex of $[0,1]$ is $1$, and $(K\times [0,1],\mu\times \mathbf{1})$ is the Cartesian product $(K,\mu)$ with
   $([0,1],\mathbf{1})$ (see Definition~\ref{Def:Product-Wt-Complexes}). 
   Moreover, if $\varrho_0$, $\varrho_1$ and $F$ are all weight-preserving, we say that 
   $\varrho_0$ and $\varrho_1$ are \emph{strongly contiguous}.
  \end{defi}

 We warn that the term ``contiguous'' is also used in Munkres's textbook~\cite{Munk84} to describe a relation between two simplicial maps. But its meaning is slightly different from what is defined above. The following lemma
shows a connection between these two notions.

\begin{lem} \label{Lem:Contiguity}
Let $\varrho,\varrho': (K,\xi)\rightarrow (K',\xi')$ be two morphisms between divisibly weighted simplicial complexes $(K,\xi)$ and $(K',\xi')$. 
   If for every simplex $\sigma$ of $K$, $\varrho(\sigma)$ and $\varrho'(\sigma)$ span a simplex in $K'$, then
    $\varrho$ and $\varrho'$ are contiguous. 
   If moreover, $\varrho$ and $\varrho'$ are both weight-preserving, then $\varrho$ and $\varrho'$ are strongly contiguous.
\end{lem}
\begin{proof}
 First of all, triangulate $K\times [0,1]$
with respect to a total ordering of the vertices of $K$. 
For a simplex $\sigma=\{v_0,\cdots,v_n\}$ of $K$, we can
define a simplicial map $F_{\sigma}: \sigma\times [0,1]\rightarrow K'$ by 
   $$F_{\sigma} (v_i\times \{0\}) = \varrho(v_i),\ \ 
   F_{\sigma} (v_i\times \{1\}) = \varrho'(v_i), \ i=0,\cdots,n. $$
  More specifically, for any $0\leq s \leq n$, the simplex of $\sigma\times [0,1]$ spanned by 
  $$ v_0\times\{0\},\cdots, v_s\times\{0\}, v_s\times\{1\}, \cdots, v_n\times \{1\} $$ 
  is mapped linearly to the unique simplex of $K'$, denoted by
  $\gamma_{\sigma}$, spanned by 
  $$\varrho(v_0),\cdots,
  \varrho(v_s), \varrho'(v_s),\cdots, \varrho'(v_n).$$
  Such a simplex $\gamma_{\sigma}$ exists because 
  $\varrho(\sigma)$ and $\varrho'(\sigma)$ span a simplex in $K'$. 
   Moreover, $F_{\sigma}: \big(\sigma\times [0,1],\xi\times \mathbf{1}\big)\rightarrow (K',\xi')$ is a morphism of weighted simplicial complexes. This follows from Lemma~\ref{Lem:Morph-Vertex} and the assumption that
  $\xi$ and $\xi'$ are both divisibly weighted and $\varrho$ and $\varrho'$ are both morphisms. In addition, if $\tau$ is a face of $\sigma$, the map $F_{\tau}$ is clearly the restriction
   of $F_{\sigma}$ to $\tau\times [0,1]$.
     So we obtain a set of
    compatible simplicial maps $\{F_{\sigma}:  \sigma\times [0,1]\rightarrow K' \}_{\sigma\in K}$ which together defines a morphism 
    $$F: \big( K \times [0,1],
   \xi\times \mathbf{1}\big) \rightarrow (K',\xi'),$$
   where $F|_{ K\times \{0\}} = \varrho$ and
   $F|_{K\times \{1\}} = \varrho'$. 
   So $\varrho$ is contiguous to $\varrho'$.\n
   
    Moreover, if $\varrho$ and $\varrho'$ are both weight-preserving, each $F_{\sigma}$ and hence the whole map $F$ is weight-preserving. So $\varrho$ is strongly contiguous to $\varrho'$.
    \end{proof}

  \begin{thm}[\text{\cite[Theorem 2.5]{Daw90}}] \label{Thm:Contig}
   If two morphisms $\varrho_0,\varrho_1: (K,\mu)\rightarrow (K',\mu')$ of weighted simplicial complexes 
   are contiguous, then the chain maps $(\varrho_0)_{\#}$
   and $(\varrho_1)_{\#}$ are chain homotopic and hence  $$ (\varrho_0)_*= (\varrho_1)_*:  H_*(K,\partial^{\mu})\rightarrow H_*(K',\partial^{\mu'}).$$ 
  \end{thm}

 \begin{defi}[Contractible weighted simplicial complex] \label{Def:Contrac-WSC} \ \n
  A weighted simplicial complex $(K, \mu)$ is called \emph{contractible} if there exists a sequence 
  $\varrho_0, \varrho_1, \ldots, \varrho_n: (K,\mu) \rightarrow (K, \mu)$ of morphisms of weighted simplicial complexes such that $\varrho_0$ is the identity on $(K, \mu)$, $\varrho_i$ is contiguous to $\varrho_{i-1}$ for each $1\leq i\leq n$, and $\varrho_n$ is a constant map. 
 \end{defi}

   \begin{prop}[\text{\cite[Corollary 2.5.1]{Daw90}}] \label{Prop:Contrac-WSC}
     If a weighted simplicial complex $(K,\mu)$ is
     contractible, then  
     $$
   H_j(K,\partial^{\mu}) \cong \begin{cases}
   \Z ,  &  \text{if $j=0$}; \\
   0 ,  &  \text{if $j \geq 1$}.
 \end{cases}
 $$
   \end{prop}
   
   The proofs of Theorem~\ref{Thm:Contig} and Proposition~\ref{Prop:Contrac-WSC} in~\cite{Daw90} are for
   ascending weighted simplicial complexes. 
   But the parallel proofs for descending weighted simplicial complexes can be written down analogously.\nn
   
  Suppose $K$ is a simplicial complex and $L$ is a subcomplex of $K$. A \emph{simplicial retraction} from $K$ to $L$ is a simplicial map $r: K\rightarrow L$
  whose restriction to $L$ is the identity map $\mathrm{id}_L$ of $L$. So if $i_L: L\hookrightarrow K$ be the inclusion map, then $r\circ i_L = \mathrm{id}_L$. The following corollary follows directly from Theorem~\ref{Thm:Contig}.

  \begin{cor}\label{Cor:Retract-Weight-Complex}
    Let $(K,\mu)$ be a weighted simplicial complex and $L$ be a simplicial subcomplex of $K$.  
    Suppose there is a simplicial retraction $r: K\rightarrow L$ such that 
    \begin{itemize}
    \item $r: (K,\mu)\rightarrow (L,\mu|_L)$ is a morphism;\n
    \item $i_L \circ r :(K,\mu)\rightarrow (K,\mu)$ is contiguous to $\mathrm{id}_K$.
    \end{itemize}
     Then 
    $r_*: H_*(K,\mu)\rightarrow H_*(L,\mu|_{L})$
    is an isomorphism.    
  \end{cor}
  
  We can readily think of $(L,\mu|_{L})$ in the above corollary as the
  deformation retract of $(K,\mu)$ in the category
  of weighted simplicial complexes.
  
\subsection{Some notations of simplicial complexes}
  \ \n
   Let $K$ be a simplicial complex. We fix some notations for our discussions in the rest of the paper. 
 \begin{itemize}
  \item For any $n\geq 0$, denote by $K^{(n)}$ the $n$-skeleton of $K$. In particular, $K^{(0)}$ is the vertex set of $K$.\n

   \item  Let $|K| \subseteq \R^N$ denote a \emph{geometrical realization} of $K$ where
  each simplex $\sigma\in K$ determines a geometric simplex $|\sigma|$ in $|K|$. So
   $ |K| = \bigcup_{\sigma\in K} |\sigma|$.\n

    \item For any simplex $\sigma$ of $K$, let
      $\mathrm{star}_K(\sigma)\subseteq K$ denote the 
     \emph{star} of $\sigma$ in $K$, i.e.
      $$ \mathrm{star}_K(\sigma) = \{ \tau\in K \,|\, \tau\cup\sigma\in K\}. $$    
    Besides, let $\mathrm{link}_K(\sigma)\subseteq K$ denote the 
     \emph{link} of $\sigma$ in $K$, i.e.
     $$ \mathrm{link}_K(\sigma) = \{ \tau\in K \,|\, \tau\cup\sigma\in K, \tau\cap\sigma=\varnothing\}. $$

     \item  For any simplex $\sigma$ of $K$, let the \emph{open star} of $\sigma$ in $|K|$ be  
     $$ \mathrm{St}(\sigma,K)=\bigcup_{\sigma\subseteq\tau} |\tau|^{\circ},$$ 
          where $|\tau|^{\circ}$ is the relative interior of $|\tau|$. Moreover, let $\overline{\mathrm{St}}(\sigma,K)$ denote the \emph{closure} of $\mathrm{St}(\sigma,K)$ in $|K|$. Then we have
          $$ \overline{\mathrm{St}}(\sigma,K) = |\mathrm{star}_K(\sigma)|, \ \ 
          \mathrm{St}(\sigma,K) = |\mathrm{star}_K(\sigma)| \backslash |\mathrm{link}_K(\sigma)|.  $$

    \item For any subcomplex $L$ of $K$, the
    \emph{simplicial neighborhood} of $L$ in $K$ is
     \begin{equation} \label{Equ:N-L-K}
       N(L,K):=\bigcup_{\sigma\in L} \mathrm{star}_K(\sigma)=\bigcup_{v\in L^{(0)}} \mathrm{star}_K(v) \subseteq K. 
       \end{equation}
     In addition, let  (see Figure~\ref{p:W-filtration} for example)
      \begin{equation} \label{Equ:S-L-K} 
      S(L,K):= \{ \sigma \in N(L,K) \,|\, \sigma\cap L=\varnothing \}. 
      \end{equation}
        Note that any simplex in $N(L,K)$ is the join of a simplex of $L$
  with some simplex of $S(L,K)$.\n
  
  \item Let $U(L,K)=|N(L,K)|\backslash |S(L,K)|$ which is an open neighborhood of $|L|$ in $|K|$. 
  More specifically,
 \begin{equation} \label{Equ:U-L-K}
   U(L,K) = \bigcup_{\sigma\in L} \mathrm{St}(\sigma, K)=\bigcup_{v\in L^{(0)}} \mathrm{St}(v, K).
   \end{equation}

     \item For any point $x\in |K|$, let $\mathrm{Car}_K(x)$ denote the unique simplex of $K$ such that $x$ is contained in the relative interior $|\mathrm{Car}_K(x)|^{\circ}$ of $|\mathrm{Car}_K(x)|$. Then 
     $$x\in \mathrm{St}(v,K) \Longleftrightarrow v\in \mathrm{Car}_K(x).$$ 
     
       \item For a simplicial map $\varrho: K\rightarrow L$, let
     $\overline{\varrho}: |K| \rightarrow |L|$ denote the continuous map determined by $\varrho$. \n

     \end{itemize}

  \begin{lem} \label{Lem:Retraction}
  If $L$ is a subcomplex of a simplicial complex $K$, then there is a piecewise linear deformation retraction 
   from $|N(L,K)|\backslash |L|$
   onto $|S(L,K)|$.
  \end{lem}
  \begin{proof}
   We can check this in the relative interior of each simplex of $N(L,K)$. Let
   $\sigma=\tau * \tau'$ be the join of a simplex
   $\tau$ of $L$ with a simplex
   $\tau'$ of $S(L,K)$. There is a canonical
   deformation retraction from $|\sigma|^{\circ}$
   onto $|\tau'|^{\circ}$ along
   the rays emitted from the points of $|\tau|^{\circ}$ (see Figure~\ref{p:Radial-Retraction}).
  These deformation retractions in all the simplices of $N(L,K)$
   fitted together define a piecewise linear deformation retraction 
   from $|N(L,K)|\backslash |L|$ onto $|S(L,K)|$.
    \end{proof}
   \begin{figure}[h]
         \includegraphics[width=0.18\textwidth]{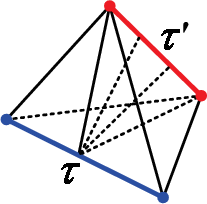}\\
          \caption{A linear retraction}\label{p:Radial-Retraction}
      \end{figure}

 \section{AW-homology and DW-homology of an orbifold}
   \label{Sec:Homology-Theory}

   Let $\mathcal{M}=(M,\mathcal{U})$ be an orbifold. It is well known that (see Goresky~\cite{Gor78} and Verona~\cite{Veron84}) there exists a triangulation $\mathcal{T}$ of $M$ such that the closure of every stratum of $M$ is a simplicial subcomplex of $\mathcal{T}$. Then the relative interior of each simplex of $\mathcal{T}$ is contained in a single stratum of $M$. A detailed proof of this result can be found in Choi~\cite[Section\,4.5]{Choi12}.
By replacing $\mathcal{T}$ by a stellar subdivision, one can assume that
the cover of closed simplices in $\mathcal{T}$ refines the cover of $M$ induced by the atlas $\mathcal{U}$.
For a simplex $\sigma$ in such a triangulation, the isotropy groups of all the interior points
of $\sigma$ are the same, and are subgroups of the isotropy groups of the boundary
points of $\sigma$. By taking a further subdivision of $\mathcal{T}$, we may assume that for any simplex $\sigma$ in $\mathcal{T}$,

\begin{itemize}
 \item[($\scalebox{1.3}{$\star$}$)] there is
one face $\tau$ of $\sigma$ such that the local group is constant on $\sigma\backslash \tau$, and possibly larger
on $\tau$. 
\end{itemize}

We call such a triangulation $\mathcal{T}$ \emph{adapted to $\mathcal{U}$} (see~\cite{MoePro99}). By abuse of notation, we also use $\mathcal{T}$
to refer to the simplicial complex defined by $\mathcal{T}$.  The property ($\scalebox{1.3}{$\star$}$) of $\mathcal{T}$ is crucial for us to define weighted simplicial homology of an orbifold.  
  \n

  \begin{prop}[{\cite[Proposition 1.2.1]{MoePro99}}]
  For any orbifold $\mathcal{M} = (M,\mathcal{U})$, there
  always exists an adapted triangulation.
\end{prop}

Note that any simplex $\sigma$ in a triangulation $\mathcal{T}$ adapted to $\mathcal{M}$ will have a vertex $v \in \sigma$ with the minimal
local group, i.e. $G_v \subseteq G_x$, for all $x \in \sigma$. Let
  \begin{equation} \label{Equ:weight-orbifold}
     \mathbf{w}(\sigma) = \mathrm{min}\{ |G_{v}|\,;\, v
      \ \text{is a vertex of}\  \sigma \}.
  \end{equation} 
  
    \begin{lem} \label{Lem:vertex_order}
  For any simplex $\sigma$ in a triangulation $\mathcal{T}$ adapted to $\mathcal{M}$, we can order the vertices of $\sigma$ 
   to be $\{v_0,\cdots,v_k\}$ so that 
   $\mathbf{w}(\sigma)=\mathbf{w}(v_0) \mid  \cdots \mid \mathbf{w}(v_k)$. So for any face $\sigma'$ of 
  $\sigma$ in $\mathcal{T}$, we have $\mathbf{w}(\sigma) \mid \mathbf{w}(\sigma')$.
   \end{lem} 
   \begin{proof}
   If the local group of $\mathcal{M}$ on the simplex $\sigma$ is constant, the lemma clearly holds.   
    Otherwise, there exists a proper face $\tau$ of $\sigma$ such that
     the local group is constant on $\sigma\backslash \tau$, and larger on $\tau$.  Let $\{v_0,\cdots, v_{s-1}\}$ be all the vertices of $\sigma\backslash \tau$ and $\{v_{s},\cdots, v_k\}$ be all the vertices of $\tau$. If the local group is constant on $\tau$, then 
     we have 
     $G_{v_0}=\cdots = G_{v_{s-1}}\subsetneq G_{v_s} = \cdots = G_{v_k}$.     
     Otherwise,  there exists a proper face $\tau'$ of $\tau$ such that
     the local group is constant on $\tau\backslash \tau'$, and larger on $\tau'$. By reordering the vertices of $\tau$ if necessary, we can assume that $v_{s},\cdots, v_{s'}$ are all the vertices of $\tau\backslash \tau'$ where $s \leq s'< k$. So
      $G_{v_0}=\cdots = G_{v_{s-1}}\subsetneq G_{v_{s}} = \cdots = G_{v_{s'}}$. By iterating the above argument, we can order the vertices of $\sigma$ to be 
      $\{v_0,\cdots,v_k\}$ so that $G_{v_0}\subseteq \cdots \subseteq G_{v_k}$. Then we have 
      $\mathbf{w}(\sigma)=\mathbf{w}(v_0) \mid  \cdots \mid \mathbf{w}(v_k)$ since $v_0$ is the vertex with the minimal local group. \n
      
      If $\sigma'$ is a face of $\sigma$, then the vertex set of $\sigma'$ is a subset of $\{v_0,\cdots,v_k\}$, say
      $\{v_{i_0},\cdots, v_{i_s}\}$, $0\leq i_0 <\cdots < i_s \leq k$. Then
      $ \mathbf{w}(\sigma)=\mathbf{w}(v_0) \mid \mathbf{w}(v_{i_0})=\mathbf{w}(\sigma')$.       
   \end{proof}

   By Lemma~\ref{Lem:vertex_order}, $\mathbf{w}$ is a 
   descending weight function on $\mathcal{T}$ and
   $\mathbf{w}$ is divisible. Let
   $\widehat{\mathbf{w}}$ be the inversion of $\mathbf{w}$ (see~\eqref{Equ:Weight-cond-Des}) which is an ascending weight on $\mathcal{T}$.
From Lemma~\ref{Lem:vertex_order}, we can easily see that:
 for any simplex $\sigma$ in $\mathcal{T}$,
 \begin{equation} \label{Equ:weight-des}
     \widehat{\mathbf{w}}(\sigma) = \mathrm{max}\{ |G_{v}|\,;\, v \ \text{is a vertex of}\  \sigma \}.
  \end{equation} 
      
\begin{defi} \label{Def:AW-DW-Orbifold}
For an orbifold $\mathcal{M}=(M,\mathcal{U})$ with an 
adapted triangulation $\mathcal{T}$,
      \begin{itemize}
  \item We call $H_*(\mathcal{T},\partial^{ \widehat{\mathbf{w}}})$ the \emph{AW-homology} of $\mathcal{M}$, denoted by $H^{AW}_*(\mathcal{M})$.\n
    
  \item We call $H_*(\mathcal{T},\partial^{\mathbf{w}})$ the  \emph{DW-homology} of $\mathcal{M}$, denoted by $H^{DW}_*(\mathcal{M})$.\n
    \end{itemize}
    
Here ``AW-homology'' and ``DW-homology'' are the abbreviations for \emph{ascending weighted homology} and 
  \emph{descending weighted homology}, respectively.
\end{defi} 
     
   We will prove in Section~\ref{Sec:Invariance} that
   $H^{AW}_*(\mathcal{M})$ and $H^{DW}_*(\mathcal{M})$ 
   are independent on the adapted
   triangulation $\mathcal{T}$ (see Corollary~\ref{Cor:Indepen}). An important thing here is that one should use divisibly weighted simplices (not arbitrary simplices) as the building blocks of an orbifold when computing the AW-homology and DW-homology. \n
      
       For any point $x\in M$, there exists a unique simplex in $\mathcal{T}$, denoted by $\mathrm{Car}_{\mathcal{T}}(x)$, which contains $x$ in its relative interior.

\begin{lem} \label{Lem:Weight-Equal}
  For any point $x\in M$, we have 
   $|G_x| = \mathbf{w}(\mathrm{Car}_{\mathcal{T}}(x))$, i.e. $|G_x|$ is equal to the smallest weight of the vertices of $\mathrm{Car}_{\mathcal{T}}(x)$. 
   
\end{lem}
\begin{proof}
 Let the vertex set of 
 $\mathrm{Car}_{\mathcal{T}}(x)$ be $\{v_0,\cdots, v_k\}$
 where $\mathbf{w}(v_0)\, |\, \cdots\, | \, \mathbf{w}(v_k)$. If the local group is constant on $\mathrm{Car}_{\mathcal{T}}(x)$, the lemma clearly holds.
 Otherwise, by the property ($\scalebox{1.3}{$\star$}$) of $\mathcal{T}$, there exists a proper face $\tau$ of 
 $\mathrm{Car}_{\mathcal{T}}(x)$ such that
 the local group is constant on $\mathrm{Car}_{\mathcal{T}}(x)\backslash \tau$. This implies $v_0\notin \tau$. On the other hand,  since $x$ is in the relative interior of $\mathrm{Car}_{\mathcal{T}}(x)$, we also have
 $x\notin \tau$. Hence $|G_x| = \mathbf{w}(v_0)= \mathbf{w}(\mathrm{Car}_{\mathcal{T}}(x))$ by the definition of $\mathbf{w}$ in~\eqref{Equ:weight-orbifold}. 
\end{proof}      

 Lemma~\ref{Lem:Weight-Equal} implies that
  the descending weight $\mathbf{w}$ 
  is a natural choice for the study of an orbifold 
  (also see Example~\ref{Exam:orbifold-descending}). 
  Moreover, Lemma~\ref{Lem:Weight-Equal} tells us that
      the order of the local group of each point of an orbifold
      is determined by the weights of the vertices of
      a triangulation adapted to the orbifold. This property motivates us to define the notion of weighted polyhedron in the next section.

\section{Weighted polyhedron} \label{Sec:Pseudo-Orbi}
   \begin{defi}[Weighted Space]
    A \emph{weighted space} is a pair $(X,\lambda)$ where
  $X$ is a topological space and $\lambda: X\rightarrow \Z_+$ is a function
  from $X$ to the set of positive integers $\Z_+$.
  \n
  Moreover, two weighted spaces $(X,\lambda)$ and $(X',\lambda')$ are called \emph{isomorphic} if there exists a homeomorphism $f: X\rightarrow X'$ such that $f$ is weight-preserving, i.e. $\lambda'(f(x))=\lambda(x)$ for any $x\in X$.\nn
  
   Note that if $X$ is connected,
  $\lambda$ is a continuous function if and only if it is constant. So in most cases, $\lambda$ is not a continuous function.
   \end{defi}

  The \emph{Cartesian product of two weighted spaces} 
 $(X,\lambda)$ and $(X',\lambda')$ is a weighted space
 $(X\times X', \lambda\times \lambda')$ where
 $(\lambda\times \lambda')(x,x') = \lambda(x)\cdot\lambda'(x')$
 for all $x\in X$, $x'\in X'$.

  \begin{defi}[Geometrical Realization of Weighted Simplicial Complex] \label{Def:Realization-WSC}\ \n
  Any weighted simplicial complex $(K,\mu)$ determines
  a map $\lambda_{\mu}: |K| \rightarrow \Z_+$ by:
      \begin{equation} \label{Equ:lambda-mu-relation}
         \lambda_{\mu} (x) = \mu(\mathrm{Car}_K(x)), \ \forall  
      x\in |K|.
      \end{equation}
     
     We call $(|K|,\lambda_{\mu})$ the \emph{geometrical realization} of $(K,\mu)$.
  \end{defi}

  \begin{exam}
  If $(K,\xi)$ is a divisibly weighted simplicial complex, we obtain two weighted spaces $(|K|,\lambda_{\xi})$ and $(|K|,\lambda_{\widehat{\xi}})$,
  where $\widehat{\xi}$ is the inversion of $\xi$ (see~\eqref{Equ:Weight-cond-Des}).
  \end{exam}

     The definition below is inspired by
 the property of triangulations adapted to an orbifold proved in Lemma~\ref{Lem:Weight-Equal}.

  \begin{defi}[Weighted Polyhedron] \label{Def:Weighted-Polyhedron} 
    A weighted space $(X,\lambda)$ is called a 
     \emph{weighted polyhedron} if there exists a
     weighted simplicial complex $(K,\mu)$ such that $(X,\lambda)$ is isomorphic
    to $(|K|,\lambda_{\mu})$, and
    we call $(K,\mu)$  a \emph{weighted triangulation} of $(X,\lambda)$. 
    Moreover, if the weight function $\mu$ is divisible,
    we call $(K,\mu)$ a \emph{divisibly weighted triangulation}. In addition, we say that the \emph{type} of weighted polyhedron $(X,\lambda)$ is 
 \emph{ascending} (or \emph{descending}) 
    if its weighted triangulation is an ascending (or descending) weighted simplicial complex.
\end{defi}

 A weighted polyhedron $(X,\lambda)$ is called \emph{compact}
 if the space $X$ is compact. If $(X,\lambda)$ is compact, then the range of $\lambda$ is a finite set. Let
 \begin{equation}\label{Equ:N-lambda}
 N_{\lambda}:= \ \text{the least common multiple of all the integers in} \ \{\lambda(x),x\in X\}. 
  \end{equation}
 Similarly to Definition~\ref{Def:Adjoint-Wt}, define
 \begin{equation}\label{Equ:adjoint-lambda}
 \lambda^*: X\rightarrow \Z_+, \ x\mapsto \frac{N_{\lambda}}{\lambda(x)}.
  \end{equation}
 We call $(X,\lambda^*)$ the \emph{adjoint} of $(X,\lambda)$. It is clear that if $(K,\mu)$  is a weighted triangulation of $(X,\lambda)$, then 
 $(K,\mu^*)$ is a weighted triangulation of $(X,\lambda^*)$. So $(X,\lambda^*)$ and $(X,\lambda)$ are of different types. \nn
 
 \begin{itemize}
\item[\textbf{(CV-7)}]  In the rest of the paper, we
always 
 assume that a weighted polyhedron is compact
 if not specified otherwise. 
\end{itemize}
\n

Moreover, in the definition of
 weighted polyhedron, we can actually require the weight
 function $\mu$ on $K$ to be divisible. 
 This is suggested by the following lemma.

 \begin{lem} \label{Lem:Barycen-Sub-Iso-Weight}
 For any weighted simplicial complex $(K,\mu)$, the
 weighted spaces $(|K|,\lambda_{\mu})$ and
 $\big(|Sd(K)|, \lambda_{Sd(\mu)}\big)$ are isomorphic.
\end{lem}
\begin{proof}
 For any nonempty simplex $\sigma$ of $K$, we can write $|\sigma|^{\circ}$ as a union
 $$ |\sigma|^{\circ} =  \underset{\sigma_0\subsetneq \cdots\subsetneq \sigma_l =\sigma}{\bigcup_{\{b_{\sigma_0}\cdots b_{\sigma_l}\}\in Sd(K)}} |\{b_{\sigma_0}\cdots b_{\sigma_l}\}|^{\circ}. $$
 
 For any point $x\in |\sigma|^{\circ}$, it follows from Definition~\ref{Def:Realization-WSC} of $\lambda_{\mu}$ that
$$\lambda_{\mu}(x) =  \mu(\sigma) = \lambda_{\mu}(b_{\sigma}). $$

 Meanwhile, $x$ must lie in the relative interior 
 of some simplex $\{ b_{\sigma_0}\cdots b_{\sigma_l}\}$
 of $Sd(K)$
 with $\sigma_0\subsetneq \cdots\subsetneq \sigma_l =\sigma$. So by the definition of $Sd(\mu)$ (see~\eqref{Equ:Subdiv-w}), 
 $$  \lambda_{Sd(\mu)} (x) = Sd(\mu)\big(\{b_{\sigma_0}\cdots b_{\sigma_l}\}\big) = \mu(\sigma_l) = \mu(\sigma). $$ 
 So the two weighted spaces $(|K|,\lambda_{\mu})$ and
 $\big(|Sd(K)|, \lambda_{Sd(\mu)}\big)$ are isomorphic.
\end{proof}

By Definition~\ref{Def:Bary-Subdiv-Weighted}, the barycentric subdivision of a weighted simplicial complex
 is always a divisibly weighted simplicial complex. So we obtain the follow corollary immediately from Lemma~\ref{Lem:Barycen-Sub-Iso-Weight}.
 
\begin{cor}\label{Cor:Div-Weight-Triangl}
 Any weighted polyhedron
   has a divisibly weighted triangulation. 
\end{cor}

\begin{lem} \label{Lem:Product-Weighted-Poly}
Suppose
 $(X,\lambda)$ and $(X',\lambda')$ are two weighted polyhedra of the same type, then their Cartesian product
 $(X\times X', \lambda\times \lambda')$ is also
 a weighted polyhedron with the same type as
 $(X,\lambda)$ and $(X',\lambda')$.
 \end{lem}
 \begin{proof}
 Let $(K,\mu)$ and $(K',\mu')$ be weighted triangulations of $(X,\lambda)$ and  $(X',\lambda')$, respectively.
 We claim that the Cartesian product $(K\times K',\mu\times \mu')$
 is a weighted triangulation of $(X\times X', \lambda\times \lambda')$. Indeed,
 suppose a point $(x,x')\in K\times K'$ is carried by
 a simplex $
   \{(v_{i_1},v'_{j_1}),\cdots, (v_{i_s},v'_{j_s})\}$ of $K\times K'$
  where
  $\{v_{i_1},\cdots, v_{i_s}\}$ is a simplex in $K$
  and $\{v'_{j_1},\cdots, v'_{j_s}\}$ is a simplex in $K'$.
  Then $x$ is carried by $\{v_{i_1},\cdots, v_{i_s}\}$
  while $x'$ is carried by $\{v'_{j_1},\cdots, v'_{j_s}\}$.
  So by our definitions, 
\begin{align*}
  \mu\times\mu'\big(\{(v_{i_1},v'_{j_1}),\cdots, (v_{i_s},v'_{j_s})\}\big) &\overset{\eqref{Equ:Product-Weight}}{=} \mu(\{v_{i_1},\cdots, v_{i_s}\})\cdot \mu'(\{v'_{j_1},\cdots, v'_{j_s}\}) \\
   &\overset{\eqref{Equ:lambda-mu-relation}}{=} \lambda(x)\lambda'(x') = \lambda\times\lambda'((x,x')).
\end{align*}  
 Clearly, $(K\times K',\mu\times \mu')$ has the same type as $(K,\mu)$ and $(K',\mu')$ by definition.
   \end{proof} 

\begin{exam} \label{Exam:orbifold-descending}
   Any orbifold $\mathcal{M}=(M,\mathcal{U})$ canonically defines a weighted space $(M,\lambda_{\mathcal{M}})$ where
   for any point $x$ in $M$,
   $$\lambda_{\mathcal{M}}(x) =\text{the order $|G_x|$ of the local group $G_x$ of}\ x.$$
   
    We claim that $(M,\lambda_{\mathcal{M}})$ is a weighted polyhedron. Indeed, we can first take an adapted triangulation $\mathcal{T}$ of $\mathcal{M}$ along with the induced descending weight $\mathbf{w}$ defined 
   by~\eqref{Equ:weight-orbifold}.  By Definition~\ref{Def:Realization-WSC} and Lemma~\ref{Lem:Weight-Equal}, we have
       \[ \lambda_{\mathbf{w}}(x)= \mathbf{w}(\mathrm{Car}_{\mathcal{T}}(x)) = |G_x| =
       \lambda_{\mathcal{M}}(x), \ \forall x\in M .\]
   Therefore, $(M,\lambda_{\mathcal{M}})$ is isomorphic to $(|\mathcal{T}|,\lambda_{\mathbf{w}})$ as weighted spaces. Hence $(\mathcal{T},\mathbf{w})$
   is a divisibly weighted triangulation of
   $(M,\lambda_{\mathcal{M}})$.  
 \end{exam}

  The above example shows that
  an orbifold is naturally a weighted polyhedron of descending type. So it suggests us to consider
  descending type weighted polyhedra as the generalization of orbifolds below. 
   
 \begin{defi}[Pseudo-orbifold]
 A weighted polyhedron $(X,\lambda)$ is called a \emph{pseudo-orbifold} if it is of descending type and
 $ X_{reg}=\{x \in X \mid \lambda(x)=1\}$ is a dense open subset
 of $X$. We call any point in $X_{reg}$ a \emph{regular point} of $(X,\lambda)$, and any point in $X \backslash X_{reg}$ a \emph{singular point}.
\end{defi}

 \begin{defi}[W-continuous map] \label{Def:W-cont-map}
   Suppose $(X,\lambda)$ and $(X',\lambda')$ are
   ascending (or descending) type weighted polyhedra. 
   A continuous map $f:X\rightarrow X'$ is called a \emph{W-continuous map} from
   $(X,\lambda)$ to $(X',\lambda')$ if 
   $$ \lambda'(f(x)) \mid \lambda(x) \ \big(\text{or}\ \lambda(x) \mid \lambda'(f(x)) \big),\ \text{for all}\
    x\in X. $$  
    We also denote such a map by $f: (X,\lambda)\rightarrow (X',\lambda')$ to indicate its W-continuity.
    Note that the definition of a W-continuous map depends
on the type of the weighted polyhedra involved.
Moreover, $f$ is called \emph{weight-preserving} if 
$$\lambda'(f(x))=\lambda(x);\ \text{for all}\ x\in X.$$

We say that $(X,\lambda)$ is \emph{isomorphic} to $(X',\lambda')$ if there exists a weight-preserving homeomorphism from $(X,\lambda)$ to $(X',\lambda')$.  
 \end{defi}

  \begin{exam} \label{Exam:Geo-Realiztion}
    If $\varrho: (K,\mu)\rightarrow (K',\mu')$ is a
   morphism between two weighted simplicial complexes of the same type, then the induced map on the weighted polyhedra $\overline{\varrho}: (|K|,\lambda_{\mu})\rightarrow
 (|K'|,\lambda_{\mu'})$ is W-continuous, called the \emph{geometrical realization} of $\varrho$.  
 In particular, if $\varrho$ is weight-preserving, then
 so is $\overline{\varrho}$. 
 \end{exam}
  
    \begin{defi}[W-Homotopy]\label{Def:W-homotopy-Equiv} Let $(X,\lambda)$ and $(X',\lambda')$ be 
    weighted polyhedra of the same type and
       $f, g: (X,\lambda)\rightarrow (X',\lambda')$ be two  
     W-continuous maps. A \emph{W-homotopy} from $f$ to $g$ relative
to $A\subseteq X$ is a W-continuous map
$$H : (X \times [0, 1],\lambda\times \mathbf{1}) \rightarrow (X',\lambda')$$ which satisfies:
 \begin{itemize}
   \item $H(x, 0) =f(x)$ for all $x\in X$;\n
   \item $H(x,1)=g(x)$ for all $x\in X$;\n
   \item $H(a,t)=f(a)=g(a)$ for all $a\in A$ and 
   $t\in [0,1]$.\n
   \end{itemize}
   
   Here $([0,1],\mathbf{1})$ is considered to be a weighted polyhedron, and $(X\times [0,1],\lambda\times \mathbf{1})$ is the product $(X,\lambda)$ with
   $([0,1],\mathbf{1})$.
   \n
   If there exists a W-homotopy from $f$ to $g$,
   we say that $f$ is \emph{W-homotopic to} $g$.
   If moreover $H$ is weight-preserving, we say
   that $f$ is \emph{strongly W-homotopic to} $g$ (a priori $f$ and $g$ are both weight-preserving).
    \end{defi}
    
    The following lemma is obvious from
    the above definition.
     
     \begin{lem}
    If two morphisms $\varrho_0,\varrho_1: (K,\mu)\rightarrow (K',\mu')$ of weighted simplicial complexes are (strongly) contiguous, then the maps $\overline{\varrho}_0,
     \overline{\varrho}_1 : (|K|,\lambda_{\mu})\rightarrow
 (|K'|,\lambda_{\mu'})$ are (strongly) W-homotopic.
 \end{lem}
    
   \begin{rem}
    ``W-continuous map'' and ``W-homotopy''
   have been defined by Takeuchi and Yokoyama~\cite[Sec.\,5]{TakYok07} for orbifolds which require all the maps to be weight-preserving. Here we redefine these two terms for the purpose of our study.
    \end{rem}
    
  \begin{defi}[W-Homotopy Equivalence] \label{Def:W-Homo-Equiv} 
  Two weighted polyhedra $(X,\lambda)$ and $(X',\lambda')$ of the same type are called \emph{W-homotopy equivalent} if there exist W-continuous maps
  $$f: (X,\lambda)\rightarrow (X',\lambda'), \ \ 
  g : (X',\lambda')\rightarrow (X,\lambda)$$
such that $g \circ f$ and $f \circ g$
 are W-homotopic to $\mathrm{id}_{X}$ and 
 $\mathrm{id}_{X'}$, respectively. The map $f$ is called a
 \emph{W-homotopy equivalence} from $(X,\lambda)$ to $(X',\lambda')$, and $g$ is called a \emph{W-homotopy inverse} of $f$.  Note that the definition of a W-homotopy equivalence depends
on the type of the weighted polyhedra involved. \n
 Moreover,
  $(X,\lambda)$ and $(X',\lambda')$ are called 
  \emph{strongly W-homotopy equivalent} if in the above definition,  $f$ and $g$ are both weight-preserving, and $g \circ f$ and $f \circ g$
 are strongly W-homotopic to $\mathrm{id}_{X}$ and 
 $\mathrm{id}_{X'}$, respectively. If so, we call $f$ and $g$ \emph{strong W-homotopy equivalences}.
\end{defi}

 \begin{exam}
  A weighted polyhedron $(X,\lambda)$ is always strongly W-homotopy equivalent to 
  $(X\times [0,1],\lambda\times \mathbf{1})$. 
 \end{exam}
 
 \begin{exam} \label{Exam:Global-Quot-Homo-Equiv}
    Suppose a compact Lie group $G$ acts smoothly, effectively, and almost freely on manifolds $X$ and $Y$. If $ \varphi : X\rightarrow Y$ is an equivariant
    homotopy equivalence, then $\varphi$ induces a 
    W-homotopy equivalence between the orbit spaces $\overline{\varphi}: X/G \rightarrow Y/G$, where 
   the orbifolds $X\slash G$ and $Y\slash G$ (see Example~\ref{Exam:Eff-Quot}) are considered as descending type weighted polyhedra.
 \end{exam}

  \begin{defi}[AW-homology and DW-homology of Weighted Polyhedron] \label{Defi:Pseudo-Orbi-Homol} \ \n
  
  Let $(X,\lambda)$ be a weighted polyhedron. Choose 
  a divisibly weighted triangulation $(K,\xi)$ of $(X,\lambda)$.
   \begin{itemize}
   \item If $(X,\lambda)$ is of ascending type, we define
    $$H^{AW}_*(X,\lambda)= H_*(K,\partial^{\xi}), \ \  H^{DW}_*(X,\lambda)= H_*(K,\partial^{\widehat{\xi}}).$$
    
    \item If $(X,\lambda)$ is of descending type, we define 
    $$H^{AW}_*(X,\lambda)= H_*(K,\partial^{\widehat{\xi}}), \ \  H^{DW}_*(X,\lambda)= H_*(K,\partial^{\xi}).$$
    \end{itemize}
    
  We call $H^{AW}_*(X,\lambda)$ and $H^{DW}_*(X,\lambda)$ the 
 \emph{AW-homology} and \emph{DW-homology} 
 of $(X,\lambda)$, respectively. Note that the geometrical realization $(|K|,\lambda_{\widehat{\xi}})$ of $(K,\widehat{\xi})$ is not $(X,\lambda)$ anymore.
\end{defi}

 In Section~\ref{Sec:Invariance}, we will prove that
  $H^{AW}_*(X,\lambda)$
 and $H^{DW}_*(X,\lambda)$ are independent on the divisibly weighted triangulation $(K,\xi)$ in their definitions (see Corollary~\ref{Cor:Indepen}).
  Moreover, we will prove in Theorem~\ref{Thm:Homotopy-Invar} that $H^{AW}_*(X,\lambda)$
 and $H^{DW}_*(X,\lambda)$ are invariants under (strong) W-homotopy equivalences of $(X,\lambda)$. Other categorial properties of $H^{AW}_*(X,\lambda)$
 and $H^{DW}_*(X,\lambda)$ follow from the categorial
 properties of weighted simplicial complexes proved
 in~\cite{Daw90}. 
 \n
 
 In addition, for the adjoint $(X,\lambda^*)$ of
 $(X,\lambda)$, it follows from Lemma~\ref{Lem:Adjoint-Weight} and Lemma~\ref{Lem:Inver-Adjoint}
 that the AW-homology and DW-homology of  $(X,\lambda^*)$ are isomorphic to the DW-homology and AW-homology of $(X,\lambda)$, respectively. So when we study AW-homology and DW-homology, it is often enough to assume a weighted polyhedron to be one of the two types when writing the proofs.

   \n
   
  \begin{defi}[W-filtration]  \label{Def:W-filtration}
   \ \n
   For a weighted polyhedron $(X,\lambda)$ where $\lambda: X\rightarrow \Z_+$, let 
     $$O^n(X,\lambda):=\{x\in X\,|\, \lambda(x)=n\} \subseteq X,\ n\in \mathrm{Im}(\lambda).$$
      We call $O^n(X,\lambda)$ a \emph{stratum} of $(X,\lambda)$.
So $X$ is the disjoint union of all its strata:
 $$X = \bigcup_{n\in \mathrm{Im}(\lambda)} O^n(X,\lambda).$$
  \end{defi}
 
  If $(X,\lambda)$ is of ascending (or descending) type, define
  $$ X_n= \ \text{the closure of}\ O^n(X,\lambda) = \bigcup_{n'\leq n} O^{n'}(X,\lambda) \ \ \Big(\text{or}  \ \bigcup_{n'\geq n} O^{n'}(X,\lambda) \Big). $$
  So $(X_n,\lambda):=(X_n,\lambda|_{X_n})$, $n\in\mathrm{Im}(\lambda)$, 
  is a sequence of weighted polyhedra called the \emph{W-filtration} of $(X,\lambda)$.\n

   Similarly, for an ascending (or descending) weighted simplicial complex $(K,\mu)$, define for each $n\in \mathrm{Im}(\mu)$ a subcomplex $K_n$ of $K$ by:
  \begin{equation} \label{Equ:mu-filtration}
     K_n := \bigcup_{\mu(\sigma)\leq n} \sigma \ \ \Big(\text{or} \   \bigcup_{\mu(\sigma)\geq n} \sigma   \Big).
   \end{equation}  
 So $(K_n, \mu )$, $n\in \mathrm{Im}(\mu)$, is a sequence of weighted simplicial complexes
  called the \emph{W-filtration} of $(K,\mu)$.
  Clearly, $(|K_n|,\lambda_{\mu})$, $n\in \mathrm{Im}(\mu)$, is the W-filtration of the weighted polyhedron $(|K|,\lambda_{\mu})$.
  
  \begin{exam}
 In Figure~\ref{p:W-filtration}, we have an ascending
 divisibly weighted simplicial complex $(K,\xi)$
 and its W-filtration where $\mathrm{Im}(\xi)=\{ 2,3,6,12\}$.
    \begin{figure}[h]
         \includegraphics[width=0.94\textwidth]{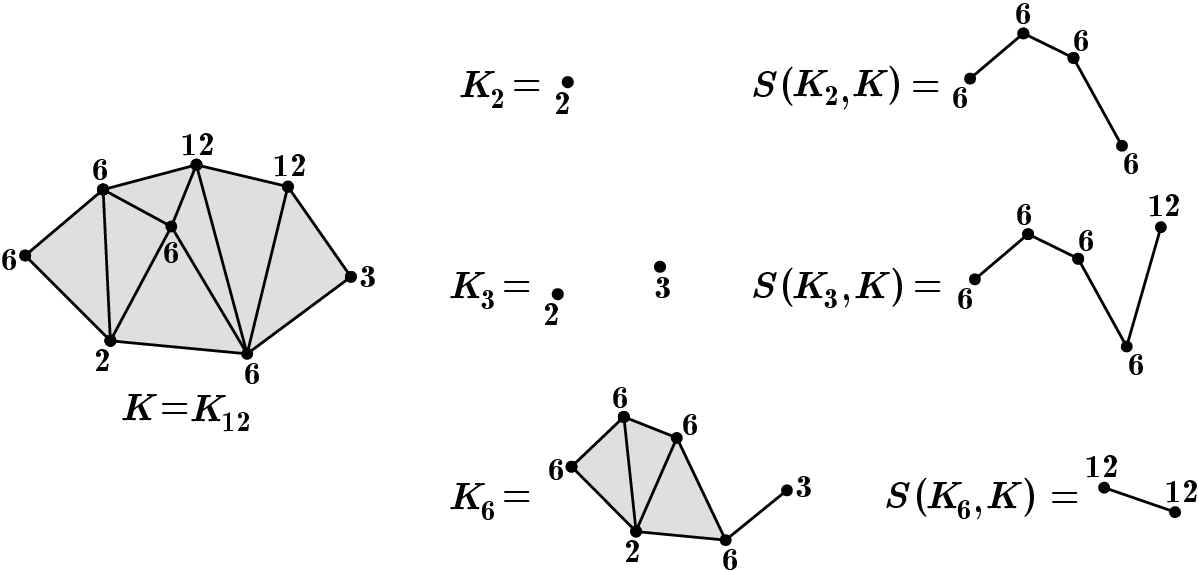}\\
          \caption{The W-filtration of a
          divisibly weighted simplicial complex}\label{p:W-filtration}
      \end{figure}
  \end{exam}
   
   The following lemma will be useful later
   when we prove that a weight-preserving
   continuous map between two weighted polyhedra
   always has a weight-preserving 
   simplicial approximation (see Lemma~\ref{Lem:Weight-Preserv-Simp-Approx}).

 \begin{lem} \label{Lem:Deformation-Retract}
   Suppose $(K,\xi)$ is a divisibly weighted simplicial complex. Then for every $n\in \mathrm{Im}(\xi)$, 
    the piecewise linear deformation retraction defined in Lemma~\ref{Lem:Retraction} from $|N(K_n,K)|  \backslash |K_n|$
     onto $|S(K_n,K)|$ is weight-preserving in $(|K|,\lambda_{\xi})$.
    \end{lem} 
 \begin{proof}
 Let $G^n:
 \big(|N(K_n,K)|  \backslash |K_n| \big)\times [0,1] \rightarrow   |S(K_n,K)|$ be the piecewise linear deformation retraction defined in Lemma~\ref{Lem:Retraction} with $G^n_0=id$ and $G^n_1=r_n$ where
      $$r_n: |N(K_n,K)|  \backslash |K_n| \rightarrow  
       |S(K_n,K)|.$$
       
  To prove that $G^n$ is weight-preserving, it is sufficient to check in the relative interior of each simplex of $N(K_n,K)$.
   Let
   $\sigma=\tau * \tau'$ be the join of a simplex
   $\tau$ of $K_n$ with a simplex
   $\tau'$ of $S(K_n,K)$.
   Assume that $\xi$ is an ascending weight. 
   Let the vertices of $\sigma$ be 
   $\{v_0,\cdots, v_s, v_{s+1},\cdots, v_k\}$
   such that
   $$\tau=\{ v_0,\cdots, v_s\}, \ \ \tau'=\{ v_{s+1},\cdots, v_k\},\ \
   \xi(v_0)\mid\cdots\mid \xi(v_s) \mid \xi(v_{s+1}) \mid\cdots \xi(v_k), $$
   where $\xi(\tau) = \xi(v_s)\leq n$ and $\xi(\sigma)=\xi(\tau')=\xi(v_k)$. By definition,
   $$\lambda_{\xi}(x)= \xi(\sigma), \ \forall
   x\in |\sigma|^{\circ}.$$
   The restriction of $r_n$ to $|\sigma|^{\circ}$
   sends a point $x\in |\sigma|^{\circ}$ to
    $r_n(x)\in |\tau'|^{\circ}$
   whose weight 
   $$\lambda_{\xi}(r_n(x))=\xi(\tau')=\xi(\sigma)=\lambda_{\xi}(x).$$ 
 Therefore, $r_n$ is weight-preserving. It follows that $G^n$ is also weight-preserving.
   The proof for a descending weight is parallel.
 \end{proof}

 To end this section, let us give an example to show that two 
 W-homotopy equivalent weighted polyhedra may not be strongly W-homotopy equivalent. In fact,
 if two weighted polyhedra $(X,\lambda)$ and $(X',\lambda')$ are strongly W-homotopy equivalent, it is necessary that their
  strata $O^n(X,\lambda)$ and $O^n(X',\lambda')$ are homotopy equivalent for every $n$. But a general W-homotopy equivalence may not satisfy this condition.
  
    \begin{figure}[h]
        \begin{equation*}
        \vcenter{
            \hbox{
                  \mbox{$\includegraphics[width=0.46\textwidth]{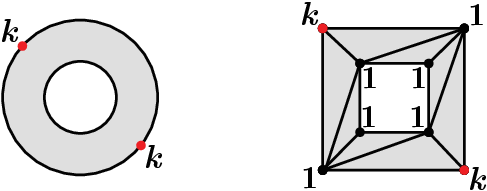}$}
                 }
           }
     \end{equation*}
   \caption{An annulus with two singular points
   on the boundary
       } \label{Fig:Example-AW-DW}
   \end{figure} 
 
\begin{exam} \label{Exam:AW-DW}
  Figure~\ref{Fig:Example-AW-DW} shows a $2$-dimensional pseudo-orbifold based on an annulus with two singular points on the boundary whose weights are $k\geq 2$. Denoted this pseudo-orbifold by $(S^1\times [0,1],\lambda_0)$. The outer circle $S^1\times \{1\}$ with the two singular points is a $1$-dimensional pseudo-orbifold, denoted by $\mathcal{S}^1_{(k,k)}$. The inner circle $S^1\times \{0\}$ consists of regular points. The right picture in Figure~\ref{Fig:Example-AW-DW} gives a divisibly weighted triangulation of $(S^1\times [0,1],\lambda_0)$, denoted by $(K,\xi)$.\n
   
  The deformation retraction from $S^1\times [0,1]$ onto $S^1\times \{1\}$ defines a W-homotopy equivalence from $(S^1\times [0,1],\lambda_0)$ to $\mathcal{S}^1_{(k,k)}$ (as descending type weighted polyhedra).
  But since the regular sets of $(S^1\times [0,1],\lambda_0)$ and $\mathcal{S}^1_{(k,k)}$ are not homotopy equivalent, $(S^1\times [0,1],\lambda_0)$ is not strongly
 W-homotopy equivalent to $\mathcal{S}^1_{(k,k)}$. 
 Another way to see this fact is: the  deformation retraction
   from $S^1\times [0,1]$ onto the inner
    circle $S^1\times \{0\}$ defines a W-homotopy equivalence from $(|K|,\lambda_{\widehat{\xi}})$ 
    to the regular circle $S^1\times \{0\}$ (as ascending type weighted polyhedra).
 This implies
 $$ H^{AW}_*(S^1\times [0,1],\lambda_0)= H_*(K,\partial^{\widehat{\xi}}) \cong 
   \begin{cases}
    \Z ,  &  \text{if $i=0,1$};\\
    0 ,  &  \text{if $i\geq 2$}.
  \end{cases} $$ 
  But by Example~\ref{Exam:1-dim-pseudo-orbifold},
  $H^{AW}_1\big(\mathcal{S}^1_{(k,k)} \big)\cong \Z\oplus \Z_k$. So $(S^1\times [0,1],\lambda_0)$ cannot be strongly
 W-homotopy equivalent to $\mathcal{S}^1_{(k,k)}$.
   \end{exam}
 
 \begin{rem}
  Example~\ref{Exam:AW-DW} shows that a W-homotopy
  equivalence of weighted polyhedra may not preserve
  AW-homology and DW-homology simultaneously.
  An analogous phenomenon occurs in the Chen-Ruan cohomology theory. It was shown in Goldin~\cite{Glodin08} that Chen-Ruan cohomology of global quotients may fail
to be invariant under W-homotopy equivalences. Actually, a strong version of 
homotopy equivalence called ``$T$-representation homotopy''
was introduced in~\cite{Glodin08} to ensure the 
invariance of Chen-Ruan cohomology.
 \end{rem}
  
  \section{Weighted homology of a divisibly weighted simplex}
   \label{Sec:Div-Wei-Simplex}
   
 According to Definition~\ref{Defi:Pseudo-Orbi-Homol},  one should consider divisibly weighted simplices as the building blocks of a weighted polyhedron when we compute the AW-homology and DW-homology groups. So in this section, we first study the properties of divisibly weighted simplices 
 before studying general weighted polyhedra.
 In the following, we consider a simplex $\sigma$ along with all its faces as a simplicial complex.
 
 \begin{lem} \label{Lem:Wt-Div-Simplex}
  Let $\sigma$ be a simplex and $\xi$ be a divisible weight on $\sigma$. Then 
   $$
   H_j(\sigma,\partial^{\xi}) \cong \begin{cases}
   \Z ,  &  \text{if $j=0$}; \\
   0 ,  &  \text{if $j \geq 1$}.
 \end{cases}
 $$
 \end{lem}
 \begin{proof}
 If the weight $\xi$ is ascending, the lemma follows from~\cite[Theorem 3.1]{Daw90}. The argument in~\cite{Daw90} is to show that $(\sigma,\xi)$
 is contractible (see Definition~\ref{Def:Contrac-WSC}).
 Indeed, suppose $\sigma=\{v_0,\cdots,v_n\}$ with $\xi(v_0)\mid \cdots \mid \xi(v_n)$. It is easy to see that
  the identity map $\mathrm{id}_{\sigma}:\sigma\rightarrow \sigma$ is contiguous to the constant map sending the whole simplex $\sigma$ to the vertex $v_0$. If $\xi$ is descending, then the adjoint $\xi^*$ of $\xi$ is an ascending divisible weight on $\sigma$. So the lemma follows from the isomorphism  
  $H_*(\sigma,\partial^{\xi}) \cong H_*(\sigma,\partial^{\xi^*})$ in~\eqref{Equ:Homology-Adjoint-Equal}.
\end{proof}
   
   \begin{rem}
     For a general weight function $\mu$ on a simplex $\sigma$, $H_*(\sigma,\partial^{\mu})$ may have torsion in some dimensions greater than $0$.
   \end{rem}
 
   We can say a bit more about 
  $H_0(\sigma,\partial^{\xi})$ in the following. For any face $\sigma'$ of a simplex $\sigma$, let
       $i_{\sigma\sigma'}: \sigma'\hookrightarrow \sigma$
       be the inclusion. Clearly, the restriction of $\xi$ to 
       $\sigma'$ defines a divisible weight on $\sigma'$. Then $i_{\sigma\sigma'}$
  induces a homomorphism
$$(i_{\sigma\sigma'})_*: H_*(\sigma',\partial^{\xi})\rightarrow H_*(\sigma,\partial^{\xi}).$$ 
  
 \begin{lem}\label{Lem:Weighted-Homl-Simplex}
   Let $\xi$ be an ascending (descending) divisible weight on $\sigma$.
     \begin{itemize}
     \item[(a)] $H_0(\sigma,\partial^{\xi})\cong \Z$ is generated by the
 vertex of $\sigma$ with the minimal (maximal) weight. \n

 \item[(b)] For any face $\sigma'$ of $\sigma$,
 the homomorphism
 $(i_{\sigma\sigma'})_*: H_0(\sigma',\partial^{\xi})\rightarrow H_0(\sigma,\partial^{\xi})$ is injective.
 \end{itemize}   
 \end{lem}
\begin{proof}
 In the following we assume the weight $\xi$ to be ascending. The proof for a descending type weight is completely parallel (or using the isomorphism in Lemma~\ref{Lem:Adjoint-Weight}).\n
 
 (a) Let the vertices of $\sigma$ be $\{v_0,\cdots, v_n\}$ where $ \xi(v_0)\, |\, \cdots\, | \, \xi(v_n) = \xi(\sigma)$. For any $0\leq a<b \leq n$,      
   let $\overline{v_av_b}$ denote the $1$-simplex in $\sigma$ oriented from
   $v_a$ to $v_b$.
     Since $\xi(\overline{v_av_b})=\xi(v_b)$, we have
   $$
     \partial^{\xi}(\overline{v_av_b})=  v_b -
   \frac{\xi(v_b)}{\xi(v_a)} v_a. 
  $$
  
  So for any $0\leq a < b < c\leq n$,  we obtain
  $$ \partial^{\xi}(\overline{v_av_c}) =
  \frac{\xi(v_c)}{\xi(v_b)} 
   \partial^{\xi}(\overline{v_av_b}) +  \partial^{\xi}(\overline{v_bv_c}). $$ 
   
   This implies:
      \begin{align*}
      H_0(\sigma,\partial^{\xi})&=\langle v_0\rangle\oplus \cdots \oplus \langle v_n \rangle \Big\slash \big\langle \partial^{\xi}(\overline{v_av_b}) , 0\leq a<b \leq n \big\rangle \\
       &= \langle v_0\rangle\oplus \cdots \oplus \langle v_n \rangle \Big\slash \big\langle \partial^{\xi}(\overline{v_av_{a+1}}) , 0\leq a \leq n-1 \big\rangle \cong \Z.
    \end{align*}
    The generator of $H_0(\sigma,\partial^{\xi})$ is
    $v_0$ which has the minimal weight in $\sigma$.\n
  
  (b) Suppose the vertex set of
    $\sigma'$ is $\{ v_{i_0},\cdots, v_{i_s}\}$ where
    $0\leq i_0 <\cdots < i_s\leq n$.
    Then by the conclusion in (a), $H_0(\sigma, \partial^{\xi})$ and $H_0(\sigma', \partial^{\xi})$
    are generated by
    $v_{0}$ and $v_{i_0}$, respectively. So
     \begin{align*} 
      \qquad (i_{\sigma\sigma'})_* :  H_0(\sigma',\partial^{\xi})
       & \longrightarrow H_0(\sigma,\partial^{\xi})\\
       [v_{i_0}]\, &\mapsto\, [v_{i_0}]=\Big[\frac{\xi(v_{i_0})}{\xi(v_0)} v_{0}\Big]. \notag
       \end{align*}
      Clearly, this homomorphism is injective.
     \end{proof}
  
 \begin{lem} \label{Lem:reduced-homo-vanish}
   If $\xi$ is a divisible weight on a simplex $\sigma$,
   then the reduced weighted simplicial homology
    $\widetilde{H}_i(\sigma,\partial^{\xi})=0$ for all
    $i \geq 0$.
 \end{lem}
 \begin{proof}
 Let $\partial^{\xi}_i : C_i(\sigma,\partial^{\xi})\rightarrow  C_{i-1}(\sigma,\partial^{\xi})$ be the
 $i$-th weighted boundary map. By Lemma~\ref{Lem:Weighted-Homl-Simplex}, we only need to verify  $\ker(\varepsilon^{\xi})\subset \mathrm{Im}(\partial^{\xi}_1)$ where
 $\varepsilon^{\xi}$ is the augmentation defined by~\eqref{Equ:Aug-1} or~\eqref{Equ:Aug-2}.
 Since $\xi$ is a divisible weight, we can order all the 
  vertices of $\sigma$ to be $\{v_0,\cdots, v_n\}$ such that
  $\xi(v_0) \mid \cdots \mid \xi(v_n)$.
  Let $\alpha=\sum^n_{i=0} m_iv_i\in C_0(\sigma,\partial^{\xi})$ be an element in $\ker(\varepsilon^{\xi})$.\n  
 \begin{itemize}
 \item If $\xi$ is ascending, by the definition of $\varepsilon^{\xi}$ in~\eqref{Equ:Aug-1}, we have
 \[ \varepsilon^{\xi}(\alpha) = \sum^n_{i=0} m_i \xi(v_i)=0 \ \Rightarrow \ m_0 \xi(v_0) = - \sum^n_{i=1} m_i \xi(v_i).\] 
 Since $\xi(\overline{v_0v_i}) = \xi(v_i)$ for any $1\leq i\leq n$, we have
   $$ \partial^{\xi} \Big( \sum^n_{i=1} m_i \overline{v_0v_i}\Big) = \sum^n_{i=1} m_i \Big( v_i - 
    \frac{\xi(v_i)}{\xi(v_0)} v_0\Big)=
    \sum^n_{i=1} m_i v_i -  
    \Big(\sum^n_{i=1} m_i\frac{\xi(v_i)}{\xi(v_0)}\Big) v_0 =\alpha. $$
    
   \item If $\xi$ is descending, by the definition of $\varepsilon^{\xi}$ in~\eqref{Equ:Aug-2}, we have $N=\xi(v_n)$ and
 \[ \varepsilon^{\xi}(\alpha) = \sum^n_{i=0} m_i\frac{N}{ \xi(v_i)}=0 \ \Rightarrow \  \frac{m_n}{\xi(v_n)} = - \sum^{n-1}_{i=0} \frac{m_i}{\xi(v_i)}.\]    
  In this case, $\xi(\overline{v_nv_i}) = \xi(v_i)$ for any $0\leq i \leq n-1$. So we have
   $$ \partial^{\xi} \Big( \sum^{n-1}_{i=0} m_i \overline{v_nv_i}\Big) = \sum^{n-1}_{i=0} m_i \Big( v_i - 
    \frac{\xi(v_n)}{\xi(v_i)} v_n\Big)=
    \sum^{n-1}_{i=0} m_i v_i -  
    \Big(\sum^{n-1}_{i=0} m_i\frac{\xi(v_n)}{\xi(v_i)}\Big) v_n =\alpha. $$
       \end{itemize}
       
    So in both cases, $\alpha$ belongs to $\mathrm{Im}(\partial^{\xi}_1)$. Hence $\ker(\varepsilon^{\xi})\subset \mathrm{Im}(\partial^{\xi}_1)$.
 \end{proof}
 
 \begin{rem} \label{Rem:Reduced-1}
   In Lemma~\ref{Lem:reduced-homo-vanish},
   $\widetilde{H}_{-1}(\sigma,\partial^{\xi})$ may not be trivial. For example, if $\xi$ is an ascending divisible
   weight, it follows from the definition of augmentation $\varepsilon^{\xi}$ in~\eqref{Equ:Aug-1} that $\widetilde{H}_{-1}(\sigma,\partial^{\xi})\cong \Z\slash k_0\Z$ where 
   $k_0 = \mathrm{min}\{ \xi(v)\,|\, v\
   \text{is a vertex of}\ \sigma\}$.
 \end{rem}

  Let $\xi$ be a divisible weight on a simplex $\sigma$. If we do barycentric subdivision to 
  $\sigma$, we naturally obtain two weighted simplicial complexes $(Sd(\sigma),Sd(\xi))$ and $(Sd(\sigma),Sd(\widehat{\xi}))$. Besides, we can take the inversion of 
  $Sd(\xi)$ to obtain yet another
  weight $\widehat{Sd(\xi)}$ on $Sd(\sigma)$. 
  But generally speaking,
   $\widehat{Sd(\xi)}$ does not equal $Sd(\widehat{\xi})$ on $Sd(\sigma)$; see Figure~\ref{p:Subdivision-Simplex} for example.    
    
     \begin{figure}[h]
         \includegraphics[width=0.68\textwidth]{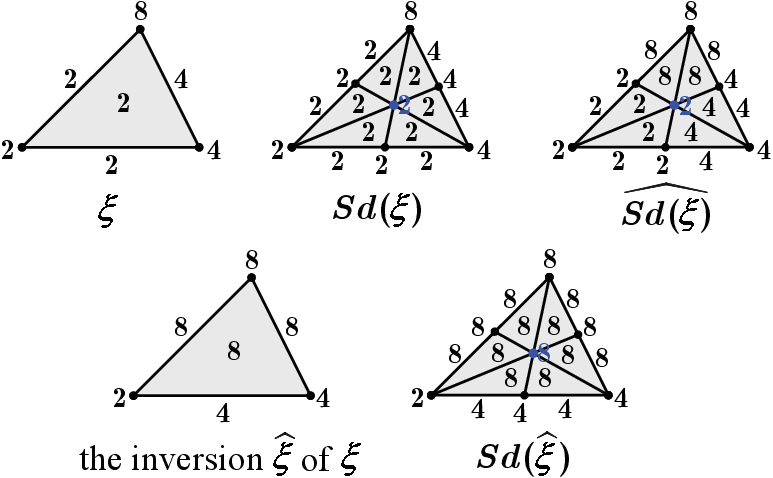}\\
          \caption{Barycentric subdivision of a divisibly weighted simplex}\label{p:Subdivision-Simplex}
      \end{figure}

  \begin{lem} \label{Lem:Wt-Div-Simplex-SD}
  Let $\xi$ be a divisible weight on a simplex $\sigma$. Then 
    $$H_j(Sd(\sigma),\partial^{Sd(\xi)}) \cong
   H_j(Sd(\sigma),\partial^{\widehat{Sd(\xi)}})
   \cong \begin{cases}
   \Z ,  &  \text{if $j=0$}; \\
   0 ,  &  \text{if $j \geq 1$}.
 \end{cases}
 $$ 
 \end{lem}  
 \begin{proof}
   We assume the weight $\xi$ to be descending in the following proof. The argument for an ascending type weight is parallel.\n
   Let us do induction on the dimension of $\sigma$. When $\dim(\sigma)=0$,
   the statement is obviously true.    
   Let $\{v_0,\cdots, v_n\}$ be the vertex set of $\sigma$ with 
    \begin{equation} \label{Equ:vertex-order-weights}
      \xi(\sigma)=\xi(v_0) \mid \cdots \mid \xi(v_n).
    \end{equation}     
     The vertex set of $Sd(\sigma)$ is  
     $\{ b_{\tau}\mid \tau \ \text{is a face of}\ \sigma\}$.  
      Define (see Figure~\ref{p:Simpl-Retract})
     $$ J_0 := \{ b_{\tau}\,|\,  \tau \ \text{is any face of}\ \sigma \ \text{that contains}\ v_0 \},$$
     $$ L_0 := b_{\sigma}\cdot Sd(\sigma\backslash v_0)\ 
     (\text{the cone over}\ Sd(\sigma\backslash v_0)).$$
  In addition, let 
    $i_0: L_0 \hookrightarrow Sd(\sigma)$
    be the inclusion map.
  \begin{figure}[h]
         \includegraphics[width=0.51\textwidth]{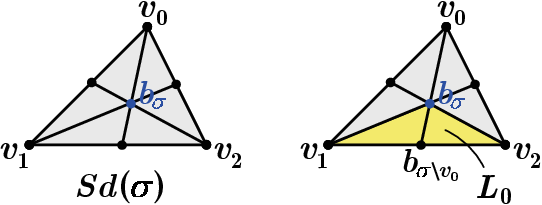}\\
          \caption{A simplicial retraction of $Sd(\sigma)$ where $\dim(\sigma)=2$}\label{p:Simpl-Retract}
      \end{figure}
   
   \n
    
    \textbf{Claim:} There is an isomorphism $H_*(Sd(\sigma),\partial^{Sd(\xi)}) \cong H_*(L_0,\partial^{Sd(\xi)}) $.\n
    
      First of all, we define a simplicial retraction $r_0: Sd(\sigma)\rightarrow  L_0$ by mapping all the vertices in $J_0$ to $b_{\sigma}$. Note that if $\tau$ contains $v_0$,
    then $\xi(\tau)=\xi(v_0)$ since $\xi$ is descending.
   So for every vertex $b_{\tau}$ in $J_0$, we have
      $$Sd(\xi)(b_{\tau})=\xi(\tau)=\xi(v_0)=\xi(\sigma) = Sd(\xi)(b_{\sigma}) . $$  
       So by Lemma~\ref{Lem:Morph-Vertex}, 
       $r_0: (Sd(\sigma),Sd(\xi))\rightarrow (L_0,Sd(\xi))$ is a morphism
       of weighted simplicial complexes (in fact $r_0$ is weight-preserving). Moreover, we can prove\n
       
   \begin{itemize}
   \item $i_0\circ r_0: (Sd(\sigma),Sd(\xi))\rightarrow (Sd(\sigma),Sd(\xi))$ is contiguous to $\mathrm{id}_{Sd(\sigma)}$. 
   \end{itemize}\n
   
   Indeed, there is a natural simplicial map
  $F: Sd(\sigma)\times [0,1]\rightarrow Sd(\sigma)$  
   which extends $\mathrm{id}_{Sd(\sigma)}$ and $i_0\circ r_0$  on the top and the bottom of $Sd(\sigma)\times [0,1]$. Since both $\mathrm{id}_{Sd(\sigma)}$ and $i_0\circ r_0$ preserve the weight
   of each vertex of $Sd(\sigma)$, we can assert that
   $F$ defines a weight-preserving morphism from
   $\big(Sd(\sigma)\times [0,1],\partial^{Sd(\xi)\times \mathbf{1}} \big)$ to $\big( Sd(\sigma),\partial^{Sd(\xi)} \big)$. So by Definition~\ref{Def:Contigu},
   $i_0\circ r_0$ is contiguous to $\mathrm{id}_{Sd(\sigma)}$. Then
   $$(r_0)_*: H_*(Sd(\sigma),\partial^{Sd(\xi)}) \rightarrow H_*(L_0,\partial^{Sd(\xi)}) $$
    is an isomorphism  by Corollary~\ref{Cor:Retract-Weight-Complex}. The claim is proved.\n
       
       Next, we define a simplicial retraction $r'_0: L_0\rightarrow Sd(\sigma\backslash v_0)$ by mapping
      the vertex $b_{\sigma}$ to $b_{\sigma\backslash v_0}$. Let $i'_0: Sd(\sigma\backslash v_0)\hookrightarrow L_0$ be the inclusion map.
     By our assumption~\eqref{Equ:vertex-order-weights}, 
     $Sd(\xi)(b_{\sigma})=\xi(\sigma)=\xi(v_0)$ divides $\xi(\sigma\backslash v_0)=Sd(\xi)(b_{\sigma\backslash v_0})$. Therefore, $r'_0 : (L_0,Sd(\xi))\rightarrow (Sd(\sigma\backslash v_0),Sd(\xi)) $ is a morphism of descending type weighted simplicial complexes. Besides, it is easy to check that\n 
     \begin{itemize}
     \item
     $i'_0\circ r'_0: (L_0,Sd(\xi))\rightarrow  (L_0,Sd(\xi))$ is contiguous to $\mathrm{id}_{L_0}$. 
     \end{itemize} \n
     
     So by Corollary~\ref{Cor:Retract-Weight-Complex} again, 
     the map $r'_0$ induces an isomorphism
     $$ (r'_0)_*: H_*(L_0,\partial^{Sd(\xi)}) \overset{\cong}{\longrightarrow} H_*(Sd(\sigma\backslash v_0),\partial^{Sd(\xi)}).$$
    Then the composition of $(r_0)_*$ and $(r'_0)_*$ gives an isomorphism 
    $$(r'_0)_*\circ (r_0)_*: H_*(Sd(\sigma),\partial^{Sd(\xi)}) \overset{\cong}{\longrightarrow} 
     H_*(Sd(\sigma\backslash v_0),\partial^{Sd(\xi)}).$$
   Finally, by our induction hypothesis on the $(n-1)$-simplex $\sigma\backslash v_0$ with the divisible weight $\xi$, we obtain the desired result
   of $H_*(Sd(\sigma),\partial^{Sd(\xi)})$.\n
   
 In addition, if let $i''_0: Sd(\sigma\backslash v_0)\rightarrow Sd(\sigma)$ be the inclusion map, then we have
  \[ \mathrm{id}_{Sd(\sigma)} \underset{c}{\simeq}
  i_0\circ r_0 = i_0 \circ \mathrm{id}_{L_0}\circ r_0 
  \underset{c}{\simeq}  i_0 \circ i'_0\circ r'_0 \circ r_0  
  =  i''_0\circ r'_0 \circ r_0. \]
  By composing these contiguous relations
  in all the induction steps, we can deduce that
   $(Sd(\sigma),Sd(\xi))$ is contractible
  in the sense of Definition~\ref{Def:Contrac-WSC}.\n 
   
   We can use a similar argument to prove the result for $H_*(Sd(\sigma),\partial^{\widehat{Sd(\xi)}})$. But there is actually a shortcut in this case.    
   Let $\varrho_{0}: Sd(\sigma)\rightarrow Sd(\sigma)$ denote the constant map sending the whole $Sd(\sigma)$ to $b_{\sigma}$. Note that for any face $\tau$ of $\sigma$,
   $$ \widehat{Sd(\xi)}(b_{\sigma}) = Sd(\xi)(b_{\sigma}) = \xi(\sigma) = \xi(v_0) \mid \xi(\tau)=
   Sd(\xi)(b_{\tau}) = \widehat{Sd(\xi)}(b_{\tau}). $$ 
   Since $\xi$ is assumed to be a descending weight on $\sigma$, $\widehat{Sd(\xi)}$ is an ascending weight on $Sd(\sigma)$.
So
 $\varrho_0: (Sd(\sigma),\widehat{Sd(\xi)})\rightarrow (Sd(\sigma),\widehat{Sd(\xi)})$ is a morphism of ascending type weighted simplicial complexes. Moreover, it is easy to check that $\varrho_0$ is contiguous to 
  the identity 
    $\mathrm{id}_{Sd(\sigma)}: (Sd(\sigma),\widehat{Sd(\xi)})\rightarrow (Sd(\sigma),\widehat{Sd(\xi)})$. So
  $(Sd(\sigma),\partial^{\widehat{Sd(\xi)}})$ is contractible in the sense of Definition~\ref{Def:Contrac-WSC}. Therefore, our result follows from Proposition~\ref{Prop:Contrac-WSC}. \n
  We warn that $\varrho_0$ is not a morphism from $(Sd(\sigma),Sd(\xi))$ to itself since $Sd(\xi)$ is a descending weight. So we cannot directly prove that $(Sd(\sigma),Sd(\xi))$ is
 contractible as we do for $(Sd(\sigma),\widehat{Sd(\xi)})$ here.
  \end{proof}
  
 In the proof of Lemma~\ref{Lem:Wt-Div-Simplex-SD},
   the weight function $\xi$ being divisible is essential for the argument. Indeed, if a weight $\mu$ on $\sigma$ is not divisible, $H_j(\sigma,\partial^{\mu})$ could be non-trivial when $j\geq 1$.  
 Later,  we will use Lemma~\ref{Lem:Wt-Div-Simplex-SD} to prove that the
 weighted simplicial homology groups of any divisibly weighted simplicial complex are preserved under barycentric subdivisions.

  \section{Invariance of AW-homology and DW-homology}
    \label{Sec:Invariance}
    
   In this section, we will prove some foundational facts that the AW-homology and DW-homology of a weighted polyhedron are invariant under isomorphisms and more generally under (strong) W-homotopy equivalences. 
   Our argument proceeds along the same line as
   the proof of the topological invariance of
   the ordinary simplicial homology of simplicial complexes in~\cite{Munk84}. \n 
   
  We will first review some basic notions and constructions in the theory of
   simplicial homology. The standard reference of these contents 
   is~\cite[Chapter\,1-2]{Munk84} which is highly recommended to the reader. Then we will show how to modify these constructions to adapt to our study of
   weighted simplicial complexes and weighted polyhedra.

  \subsection{Simplicial approximations of continuous maps}
  \label{Subsec:Simpl-Approx}
  \ \n
    
   \begin{defi}\label{Defi:Simpl-Approx}
   Let $K$ and $L$ be two simplicial complexes
   and $f: |K|\rightarrow |L|$ be a continuous map. 
   If $\varrho: K\rightarrow L$ is a simplicial map such that:
    \[ f(\mathrm{St}(v,K)) \subseteq \mathrm{St}(\varrho(v),L) \] 
    for every vertex $v$ of $K$, then $\varrho$ is called a  \emph{simplicial approximation} of $f$. 
   \end{defi} 
    
    The following are some standard facts on
    simplicial approximation:
    
    \begin{itemize}    
     \item A simplicial map $\varrho : K\rightarrow L$ is a simplicial approximation of $f$ if and only if
     $\overline{\varrho}(x)\in |\mathrm{Car}_L(f(x))|$ for all $x\in |K|$.\n
     
    \item A continuous map $f: |K|\rightarrow |L|$ has a 
     simplicial approximation if and only if 
     $f$ satisfies the \emph{star condition} relative to $K$ and $L$: for every vertex $v$ of $K$, there exists a vertex $u$ of $L$ such that $f(\mathrm{St}(v,K)) \subseteq \mathrm{St}(u,L)$. \n
     
     \item If $\varrho, \varrho' :K\rightarrow L$ are two 
   simplicial approximations of $f$, then the chain maps
   $\varrho_{\#}$ and $\varrho'_{\#}$ are chain homotopic and so $\varrho_*=\varrho'_* : H_*(K)\rightarrow H_*(L)$. 
      \end{itemize}
  
    \begin{thm}[\text{Existence of Simplicial Approximation~\cite[Theorem 16.1, 16.5]{Munk84}}] \label{Thm:Simp-Approx}
   For two simplicial complexes $K$ and $L$ and a continuous map $f: |K|\rightarrow |L|$,
       there always exists a subdivision $K'$ of $K$ such that $f$ has a simplicial approximation $\varrho: K' \rightarrow L$.
    \end{thm}
    
    In Theorem~\ref{Thm:Simp-Approx}, if the simplicial complex $K$ is compact, the subdivision $K'$ can be taken to be $Sd^m(K)$ for some
    large enough integer $m$. But if $K$ is not compact, some generalized barycentric subdivisions (see~\cite[p.\,90]{Munk84}) are needed to construct the complex $K'$ .\n

 \subsection{Acyclic carrier}\ \n
  A chain complex $\mathcal{C}_*=(C_p,\partial_p)$ is called \emph{free} if each $C_p$ is a free abelian group.
 Besides, $\mathcal{C}_*$ is called \emph{acyclic}  if its homology group $H_*(\mathcal{C}_*)$ satisfies
    $$ H_i(\mathcal{C}_*)= \begin{cases}
    0 ,  &  \text{if $i\geq 1$}; \\
    \Z ,  &  \text{if $i=0$}.
 \end{cases} $$
 
  \begin{defi}[Acyclic Carrier]
  Let $(\mathcal{C}_*,\varepsilon)=(C_p,\partial_p,\varepsilon)$ be an augmented chain complex. Suppose $\mathcal{C}_*$ is free; let $\{\sigma_p^\alpha\}$ be a basis for $C_p$, as $\alpha$ ranges over some index set $J_p$. Let $(\mathcal{C}'_*,\varepsilon')=(C'_p,\partial'_p,\varepsilon')$ be an arbitrary augmented chain complex. An \emph{acyclic carrier} from $\mathcal{C}_*$ to $\mathcal{C}'_*$, relative to the given bases, is a function $\Phi$ that assigns to each basis element $\sigma_p^\alpha$, a subchain complex $\Phi (\sigma_p^\alpha)$ of $\mathcal{C}^{\prime}_*$, satisfying the following conditions:
  \begin{itemize}
  \item The chain complex $\Phi(\sigma_p^\alpha)$ is augmented by $\varepsilon'$ and is acyclic.\n
 \item If $\sigma_{p-1}^\beta$ appears in the expression for $\partial_p \sigma_p^\alpha$ in terms of the preferred basis for $C_{p-1}$, then $\Phi(\sigma_{p-1}^\beta)$ is a subchain complex of $\Phi(\sigma_p^\alpha)$.\n
 \end{itemize}
  
  A homomorphism $f: C_p \rightarrow C_q^{\prime}$ is said to be \emph{carried by $\Phi$} if $f(\sigma_p^\alpha)$ belongs to the $q$-dimensional group of the subchain complex $\Phi(\sigma_p^\alpha)$ of $\mathcal{C}^{\prime}$, for each $\alpha$. Moreover, a chain map 
  $\phi=\{\phi_p\}: \mathcal{C}_*\rightarrow \mathcal{C}'_*$ is said to be \emph{carried by $\Phi$} if
   each $\phi_p: C_p \rightarrow C_p^{\prime}$ is carried by $\Phi$.
\end{defi}

  \begin{thm}[\text{Algebraic Acyclic Carrier Theorem~\cite[Theorem 13.4]{Munk84}}] \label{Thm:Acyclic-Carrier} 
  Suppose $(\mathcal{C}_*,\varepsilon)$ and $(\mathcal{C}'_*,\varepsilon')$ are augmented chain complexes where $\mathcal{C}_*$ is free.  Let $\Phi$ be an acyclic carrier
   from $\mathcal{C}_*$ to $\mathcal{C}'_*$ relative to a set of preferred bases for $\mathcal{C}_*$. Then
   there is an augmentation-preserving chain map $\phi: 
   \mathcal{C}_* \rightarrow \mathcal{C}'_*$ carried by $\Phi$. Moreover, any two such chain maps are chain homotopic and the chain homotopy is also carried by $\Phi$.    
  \end{thm} 
  
  Theorem~\ref{Thm:Acyclic-Carrier} is particularly useful in the situation when one wants to prove that two chain maps are chain homotopic but do not want to write down an explicit formula for the chain homotopy.  In addition,
 if we let the chain complexes $\mathcal{C}_*$ and $\mathcal{C}'_*$ in the above theorem be the
  simplicial chain complexes of two simplicial complexes, then we obtain the geometric version of the acyclic carrier theorem (see~\cite[Theorem 13.3]{Munk84}).
 
  \n

 \subsection{Barycentric subdivision}
    \ \n
  For a simplicial complex $K$, let
     $Sd_{\#} : C_*(K)\rightarrow C_*(Sd(K))$ be the ordinary chain map induced by the barycentric subdivision $Sd$ of $K$. We warn that $Sd$ is not a simplicial map.   
   By definition, $Sd_{\#}$
  sends any oriented $n$-simplex $\sigma$ to the sum of all the $n$-simplices in $Sd(\sigma)$ with the same orientation. 
  More precisely, any $n$-simplex in $Sd(\sigma)$ can be written as $\{b_{\sigma_0}\cdots b_{\sigma_n}\}$ where $\sigma_0\subsetneq \cdots \subsetneq \sigma_n = \sigma$ is an ascending sequence of faces of $\sigma$. So
  \begin{equation} \label{Equ:Bary-Subdiv-Form}
   Sd_{\#}(\sigma)= \sum_{\sigma_0\subsetneq \cdots\subsetneq \sigma_n = \sigma} 
 \varepsilon(\sigma_0,\cdots,\sigma_n)\cdot [b_{\sigma_0}\cdots b_{\sigma_n}]. 
 \end{equation}
 where $\varepsilon(\sigma_0,\cdots,\sigma_n)\in \{-1,1\}$
 is properly chosen to make $Sd_{\#}\circ\partial = 
  \partial\circ Sd_{\#}$.\n
  
  For brevity, we omit the sign $\varepsilon(\sigma_0,\cdots,\sigma_n)$ in~\eqref{Equ:Bary-Subdiv-Form}  and write the definition of $Sd_{\#}$ inductively as
   \begin{equation} \label{Equ:Sd-simple}
    Sd_{\#}(\sigma)=b_{\sigma}\cdot Sd_{\#}(\partial \sigma), \ \text{where}\ Sd_{\#}(v)=v, \, \forall v\in K^{(0)},
    \end{equation}    
  where for any $(n-1)$-simplex $\tau$ in $Sd_{\#}(\partial \sigma)$, $b_{\sigma}\cdot \tau$ is the cone of $\tau$ with $b_{\sigma}$.    
 \n
 
  Let $Sd_* : H_*(K)\rightarrow H_*(Sd(K))$ denote the homomorphism induced by $Sd_{\#}$.
    
   \begin{thm}[\text{see~\cite[\S 17]{Munk84}}] \label{Thm:Sd-Homology}
     $Sd_* : H_*(K)\rightarrow H_*(Sd(K))$ is an isomorphism.
   \end{thm}

  \subsection{Barycentric subdivision of a weighted simplicial complex} \ \n
  
  Suppose $(K,\mu)$ is a weighted simplicial complex.
  By Definition~\ref{Def:Bary-Subdiv-Weighted},  $\mu$ induces a divisible weight $Sd(\mu)$
    on $Sd(K)$. In particular, for a simplex $\{b_{\sigma_0},\cdots, b_{\sigma_l}\}$ in $Sd(K)$ where $\sigma_0\subsetneq \cdots\subsetneq \sigma_l \in K$, we have
   \begin{equation} \label{Equ:weight-Sd}
     Sd(\mu)\big(\{b_{\sigma_0},\cdots, b_{\sigma_l}\}\big) = \mu(\sigma_l). 
   \end{equation} 
        
        Note that for an $n$-simplex $\sigma$ of $K$, 
      $Sd_{\#}(\sigma)$ in~\eqref{Equ:Bary-Subdiv-Form} consists of a collection of $n$-simplices having the same weight as $\sigma$. 
 
   \begin{lem} \label{Lem:Sd-Weight-Chain}
   $Sd_{\#}: \big(C_*(K),\partial^{\mu} \big) \rightarrow 
   \big(C_*(Sd(K)),\partial^{Sd(\mu)} \big)$ is a chain map.
  \end{lem}
  \begin{proof}
 We assume the weight function $\mu$ to be descending in the argument below. The proof for an ascending type weight is completely parallel.\n
 
 For an oriented $n$-simplex $\sigma$ of $K$, by definition
 \[ \partial \sigma = \sum^n_{j=0} (-1)^j  \partial_j\sigma, \ \ \partial^{\mu} \sigma \overset{\eqref{Equ:weighted-Boun-DW}}{=} \sum^n_{j=0} (-1)^j \frac{\mu(\partial_j\sigma)}{\mu(\sigma)} \partial_j\sigma. \]
 So we have
    \begin{align*}
    \partial^{Sd(\mu)}\circ Sd_{\#}(\sigma) 
   & \overset{\eqref{Equ:Sd-simple}}{=} \partial^{Sd(\mu)}\big( b_{\sigma}\cdot 
    Sd_{\#}(\partial \sigma ) \big)
    = \partial^{Sd(\mu)}\Big( b_{\sigma}\cdot 
    \sum^n_{j=0} (-1)^j   Sd_{\#}(\partial_j\sigma)  \Big) \\
    &=   \sum^n_{j=0} (-1)^j  \partial^{Sd(\mu)} \big( b_{\sigma}\cdot  Sd_{\#}(\partial_j\sigma)  \big) \\
    &\overset{\divideontimes}{=} 
    \sum^n_{j=0} (-1)^j \Big( \frac{\mu(\partial_j\sigma)}{\mu(\sigma)} Sd_{\#}(\partial_j\sigma) - b_{\sigma}\cdot \partial \,
    Sd_{\#}(\partial_j \sigma )  \Big) \\ 
     (\text{by induction})  &= Sd_{\#} \Big( \sum^n_{j=0} (-1)^j  \frac{\mu(\partial_j\sigma)}{\mu(\sigma)} \partial_j\sigma \Big) - b_{\sigma}\cdot  Sd_{\#}(\partial\partial\sigma ) \\
    & =
    Sd_{\#}\circ \partial^{\mu} \sigma.
   \end{align*}
  The equality $\overset{\divideontimes}{=}$ follows from the facts that for any $(n-1)$-simplex $\tau$ in $Sd(\partial_j \sigma)$:\n
  
  $\bullet$ $Sd(\mu)(\tau)= \mu(\partial_j\sigma)$ by~\eqref{Equ-Sd-n-simplex}, and $Sd(\mu)(b_{\sigma}\cdot \tau)=\mu(\sigma)$ by~\eqref{Equ:weight-Sd};\n
  
 $\bullet$ for each $(n-2)$-face $\theta$ of $\tau$, $Sd(\mu)(b_{\sigma}\cdot \theta) =\mu(\sigma)$ by~\eqref{Equ:weight-Sd}.
 \end{proof}  
    
  For a divisibly weighted simplicial complex
    $(K,\xi)$, the following theorem (which generalizes Theorem~\ref{Thm:Sd-Homology}) tells us that
    weighted simplicial homology of $(K,\xi)$ is preserved under barycentric subdivisions. This result is crucial for us to prove the invariance of
   AW-homology and DW-homology under isomorphisms of weighted polyhedra.

    \begin{thm} \label{Thm:Sd-Homology-Isom-1}
 Let $(K,\xi)$ be a divisibly weighted simplicial complex.\n
 \begin{itemize}
 \item[(i)] $Sd_{\#}: \big(C_*(K),\partial^{\xi}\big) \rightarrow 
   \big(C_*(Sd(K)),\partial^{Sd(\xi)} \big)$ is a chain homotopy equivalence which induces an isomorphism
   $Sd_*: 
   H_*(K,\partial^{\xi})\rightarrow H_* \big(Sd(K),\partial^{Sd(\xi)} \big)$. 
  \item[(ii)] For the inversion weight $\widehat{\xi}$ of $\xi$, there exists a chain map
  $$\widehat{Sd}_{\#}:  C_*(K,\partial^{\widehat{\xi}})\rightarrow C_* \big(Sd(K),\partial^{\widehat{Sd(\xi)}} \big)$$
   which induces an isomorphism 
   $\widehat{Sd}_*: 
   H_*(K,\partial^{\widehat{\xi}})\rightarrow H_* \big(Sd(K),\partial^{\widehat{Sd(\xi)}} \big)$.
   \end{itemize}
 \end{thm}
 \begin{proof} 
   We assume the weight $\xi$ to be of descending type in the following proof. The argument for an ascending type weight is parallel.\n 
 (i) We first construct 
   a special simplicial approximation of $\mathrm{id}_K$.
   Choose a total ordering $\prec$ of all the vertices of $K$ such that $\xi(v)\leq \xi(v')$ for any $v\prec v'$. 
 Let $\sigma=\{v_0,\cdots, v_n\}$ be a simplex in $K$ where $v_0\prec\cdots\prec v_n$.
 Since $(K,\xi)$ is divisibly weighted and $\xi$ is descending, we have
  \begin{equation} \label{Equ:sigma-des-weight}
    \xi(\sigma)=\xi(v_0) \mid \cdots \mid \xi(v_n).
    \end{equation}
   
    Then we define a simplicial map $\pi^{\xi} : Sd(K)\rightarrow K$ by 
    \begin{align} \label{Equ:sigma-cond}
     \pi^{\xi}: Sd(K)^{(0)} &\longrightarrow K^{(0)} \\
             b_{\sigma}\, &\longmapsto \, v_0 \in \sigma.\notag
     \end{align}
   
    Note that $\pi^{\xi} : Sd(K)\rightarrow K$ is a simplicial approximation of the identity map $\mathrm{id}_{|K|}:
    |K|=|Sd(K)| \rightarrow |K|$. Moreover, for each $\sigma$
    of $K$,  
    $$Sd(\xi)(b_{\sigma})=\xi(\sigma)=
    \xi(v_0)=\xi(\pi^{\xi}(b_{\sigma})).$$
    So $\pi^{\xi}$ preserves the weight of every vertex of $Sd(K)$.  
   Then since $(K,\xi)$ and $(Sd(K),Sd(\xi))$ are both divisibly weighted,
    $\pi^{\xi}: (Sd(K),Sd(\xi))\rightarrow (K,\xi) $ is
    a weight-preserving simplicial map by Lemma~\ref{Lem:Morph-Vertex}.
   \n

 \noindent \textbf{Claim-1:} The chain map
   $Sd_{\#}\circ \pi^{\xi}_{\#} :
   \big(C_*(Sd(K)),\partial^{Sd(\xi)} \big) \rightarrow 
   \big( C_*(Sd(K)),\partial^{Sd(\xi)} \big)$
    is carried 
   by an acyclic carrier $\Phi^{\xi}$ on $ C_*(Sd(K),\partial^{Sd(\xi)} )$ defined by
    \begin{equation} \label{Equ:Acyclic-Carrier}
      \Phi^{\xi}([b_{\sigma_0}\cdots b_{\sigma_l}])=
      \big( C_*(Sd(\sigma_l)),\partial^{Sd(\xi)} \big), \ 
      \sigma_0\subsetneq \cdots\subsetneq \sigma_l \in K.
    \end{equation}  
 
  Indeed, $\Phi^{\xi}$ is an acyclic carrier  since each $\big( C_*(Sd(\sigma_l)),\partial^{Sd(\xi)} \big)$ 
  is acyclic by Lemma~\ref{Lem:Wt-Div-Simplex-SD}.
  In addition, each vertex $b_{\sigma_i}$ of the simplex $[b_{\sigma_0}\cdots b_{\sigma_l}]$ in~\eqref{Equ:Acyclic-Carrier} is clearly mapped to a vertex of 
  $\sigma_l$. This implies that 
  $Sd_{\#}\circ \pi^{\xi}_{\#}$ is carried by $\Phi^{\xi}$. So Claim-1 is proved.\n 
      
 Moreover, since $\pi^{\xi}$ is weight-preserving and $Sd(v)=v$ for every vertex $v$ of $K$, the chain map $Sd_{\#}\circ \pi^{\xi}_{\#}$ extends to an augmentation-preserving
 map on the augmented chain complexes
 $\big( Sd(K),\partial^{Sd(\xi)},  \varepsilon^{Sd(\xi)} \big)$.      
        Then since $\mathrm{id}_{C_*(Sd(K))}$ is also an augmentation-preserving chain map carried by $\Phi^{\xi}$,  
     Theorem~\ref{Thm:Acyclic-Carrier} implies
   $$Sd_*\circ \pi^{\xi}_* = \mathrm{id}:
   H_*\big( Sd(K),\partial^{Sd(\xi)} \big)\rightarrow H_*\big( Sd(K),\partial^{Sd(\xi)}\big).$$

 \noindent \textbf{Claim-2:} $\pi^{\xi}_{\#}\circ 
 Sd_{\#} = \mathrm{id}_{C_*(K)}: (C_*(K),\partial^{\xi})\rightarrow (C_*(K),\partial^{\xi})$.\n
  
  Obviously, $\pi^{\xi}_{\#}\circ Sd_{\#}$ agrees with $\mathrm{id}_{C_*(K)}$ on $C_0(K)$.     
    Assume that $\pi^{\xi}_{\#}\circ Sd_{\#}$ agrees with $\mathrm{id}_{C_*(K)}$ on $C_{<n}(K)$, $n\geq 1$. 
   For an oriented $n$-simplex $\sigma=[v_0,\cdots,v_n]$ of $K$, we have
    \begin{align} \label{Equ:pi-Sd}
       \pi^{\xi}_{\#}\circ Sd_{\#}(\sigma) & \overset{\eqref{Equ:Sd-simple}}{=} 
     \pi^{\xi}_{\#} \big(  b_{\sigma}\cdot (Sd_{\#}( \partial\sigma)) \big)  \overset{\eqref{Equ:sigma-cond}}{=} v_0 \cdot \big(  \pi^{\xi}_{\#}\circ Sd_{\#} (\partial\sigma) \big)\\
  (\text{by induction})   &= v_0\cdot \partial \sigma =  v_0\cdot \Big(
  \sum^n_{j=0} (-1)^j [v_0,\cdots, \widehat{v}_j,\cdots,v_n] \Big) \notag \\
  &= v_0\cdot [v_1,\cdots,v_n] =\sigma. \notag
       \end{align}
   So Claim-2 is proved. Then $\pi^{\xi}_{\#}$ is a chain homotopy inverse of $Sd_{\#}$ and so
    \begin{equation} \label{Equ:Iso-Comp-H-2}
   \pi^{\xi}_*\circ Sd_* =\mathrm{id}: H_*(K,\partial^{\xi})\rightarrow
  H_*(K,\partial^{\xi}).
  \end{equation} 
  Therefore, $Sd_*$ is an isomorphism.
  \n
  (ii) Since
  the inversion $\widehat{Sd(\xi)}$ of $Sd(\xi)$ 
  is not equal to  $Sd(\widehat{\xi})$ in general, we cannot derive
  (ii) by simply applying the result in (i) to $\widehat{\xi}$.
  Suppose $\sigma=\{v_0,\cdots, v_n\}$ is an $n$-simplex of $K$ whose vertices are ordered as in~\eqref{Equ:sigma-des-weight}. Then
   \begin{equation} \label{Equ:sigma-hat}
     \xi(\sigma)=\xi(v_0) \mid \cdots \mid \xi(v_n)=\widehat{\xi}(\sigma).
     \end{equation}
 
 For an $n$-simplex $\{b_{\sigma_0}\cdots b_{\sigma_n}\}$ in $Sd(\sigma)$ where
     $\sigma_0\subsetneq \cdots\subsetneq \sigma_n=\sigma \in K$,
     $$\widehat{Sd(\xi)}(b_{\sigma_i})=Sd(\xi)(b_{\sigma_i}) = \xi(\sigma_i), \ 0\leq i \leq n.$$
    Since $\xi$ is assumed to be descending,
    $\xi(\sigma)=\xi(\sigma_n) \mid \xi(\sigma_{n-1})
     \mid \cdots \mid \xi(\sigma_0)$. Hence
    $$ \widehat{Sd(\xi)}(b_{\sigma_n}) \mid \widehat{Sd(\xi)}(b_{\sigma_{n-1}}) \mid \cdots \mid \widehat{Sd(\xi)}(b_{\sigma_0})  .$$
  Then because the inversion 
  $\widehat{Sd(\xi)}$ of $Sd(\xi)$ is an ascending weight, we obtain 
   \begin{equation} \label{Equ:Sd-xi-widehat}
     \widehat{Sd(\xi)}([b_{\sigma_0}\cdots b_{\sigma_n}])
 = \widehat{Sd(\xi)}(b_{\sigma_0})= \xi(\sigma_0).
 \end{equation}
 Note that here $\sigma_0$ is just a vertex of $\sigma$. So according to~\eqref{Equ:sigma-hat}, we obtain
 $$\widehat{Sd(\xi)}([b_{\sigma_0}\cdots b_{\sigma_n}]) =  \xi(\sigma_0) \mid \xi(v_n)=\widehat{\xi}(\sigma). $$
 Then parallel to the definition of $Sd_{\#}$
 in~\eqref{Equ:Bary-Subdiv-Form}, we define a linear map
 $\widehat{Sd}_{\#}$ by:
  \begin{equation} \label{Equ:Bary-Subdiv-Form-widehat}
   \widehat{Sd}_{\#}(\sigma)= \sum_{\sigma_0\subsetneq \cdots\subsetneq \sigma_n = \sigma} 
 \varepsilon(\sigma_0,\cdots,\sigma_n) \cdot
 \frac{\widehat{\xi}(\sigma)}{\widehat{Sd(\xi)}([b_{\sigma_0}\cdots b_{\sigma_n}])}  [b_{\sigma_0}\cdots b_{\sigma_n}]. 
 \end{equation}
 Similarly to~\eqref{Equ:Sd-simple}, we can also write the
 definition of $\widehat{Sd}_{\#}$ inductively as 
  \begin{equation} \label{Equ:Sd-simple-hat}
    \widehat{Sd}_{\#}(\sigma)=b_{\sigma}\cdot \widehat{Sd}_{\#}(\partial^{\widehat{\xi}} \sigma), \ \text{where}\  \widehat{Sd}_{\#}(v)=v, \, \forall v\in K^{(0)}.
    \end{equation} 
 By a similar argument as in the proof of  
 Lemma~\ref{Lem:Sd-Weight-Chain}, we can show that
 $\widehat{Sd}_{\#}$ is a chain map from
 $C_*(K,\partial^{\widehat{\xi}})$ to 
 $C_* \big(Sd(K),\partial^{\widehat{Sd(\xi)}} \big)$.\n
 
 Moreover, we can prove that $\widehat{Sd}_{\#}$ is a chain homotopy equivalence. Again, consider the simplicial map $\pi^{\xi}$ defined in~\eqref{Equ:sigma-cond}. By definition, $\widehat{\xi}$ and $\widehat{Sd(\xi)}$ agrees with $\xi$ and $Sd(\xi)$, respectively, at each vertex. Then since $\pi^{\xi}:Sd(K)\rightarrow K$ preserves the weight of every vertex of $Sd(K)$, 
 $\pi^{\xi}: (Sd(K),
   \widehat{Sd(\xi)})\rightarrow (K,\widehat{\xi})$ is also weight-preserving. Denote the chain map 
   determined by $\pi^{\xi}$ here by 
   $$ \widehat{\pi}^{\xi}_{\#}: C_*(Sd(K),
   \partial^{\widehat{Sd(\xi)}})\rightarrow C_*(K,\partial^{\widehat{\xi}})  $$
   to distinguish it from the chain map
   $\pi^{\xi}_{\#}: C_*(Sd(K),
   \partial^{Sd(\xi)}) \rightarrow  C_*(K,\partial^{\xi})$.\n
  
  We claim that $\widehat{\pi}^{\xi}_{\#}$ is a chain homotopy inverse of $\widehat{Sd}_{\#}$. Indeed, we only need to check that
     \begin{itemize}
     \item  The chain map
   $\widehat{Sd}_{\#}\circ \widehat{\pi}^{\xi}_{\#} :
   \big(C_*(Sd(K)),\partial^{\widehat{Sd(\xi)}} \big) \rightarrow 
   \big( C_*(Sd(K)),\partial^{\widehat{Sd(\xi)}} \big)$
    is carried 
   by an acyclic carrier $\widehat{\Phi}^{\xi}$ on $ C_*(Sd(K),\partial^{\widehat{Sd(\xi)}} )$ defined by
    \begin{equation} \label{Equ:Acyclic-Carrier-widehat}
      \widehat{\Phi}^{\xi}([b_{\sigma_0}\cdots b_{\sigma_l}])=
      \big( C_*(Sd(\sigma_l)),\partial^{\widehat{Sd(\xi)}} \big), \ 
      \sigma_0\subsetneq \cdots\subsetneq \sigma_l \in K.
    \end{equation}

     \item $\widehat{\pi}^{\xi}_{\#}\circ 
 \widehat{Sd}_{\#} = \mathrm{id}_{C_*(K)}: (C_*(K),\partial^{\widehat{\xi}})\rightarrow (C_*(K),\partial^{\widehat{\xi}})$.
     \end{itemize}
   \n
   
  Since the argument is almost the same as that for $Sd_{\#}$ using Lemma~\ref{Lem:Wt-Div-Simplex-SD}, we leave it to the reader. Then $\widehat{Sd}_{\#}$ induces an isomorphism 
  $$\widehat{Sd}_*: 
   H_*(K,\partial^{\widehat{\xi}})\rightarrow H_* \big(Sd(K),\partial^{\widehat{Sd(\xi)}} \big).$$   
   This finishes the proof of the theorem.   
\end{proof}

 \subsection{Invariance of AW-homology and DW-homology} \label{Subsec:Invar-St-Homol}

  \ \n

 In this section, we prove that AW-homology and DW-homology of a weighted polyhedron are invariants
    under isomorphisms and more generally under (strong) W-homotopy equivalences of weighted polyhedra. We need the following notion.

     \begin{defi}[W-simplicial Approximation]
       Suppose $(X,\lambda)$ and $(X',\lambda')$ are two weighted polyhedra of the same type.
       Let $(K,\xi)$ and $(K',\xi')$ be divisibly weighted triangulations of $(X,\lambda)$ and $(X',\lambda')$, respectively. For a continuous map $f: X\rightarrow X'$, if $\varrho: K\rightarrow K'$
    is a simplicial approximation of $f$ and moreover $\varrho$ is a morphism from
  $(K,\xi)$ to $(K',\xi')$, then we call 
  $\varrho$ a \emph{W-simplicial approximation} of $f$. 
     \end{defi}

    The following lemma is a bit surprising at first look. But it follows naturally from the definition of W-continuous map.
    
     \begin{lem} \label{Lem:W-simp-Approx-0}
     Let $(X,\lambda)$ and $(X',\lambda')$ be two weighted polyhedra of the same type. Let $(K,\xi)$ and $(K',\xi')$ be divisibly weighted triangulations of $(X,\lambda)$ and $(X',\lambda')$, respectively.
     For a W-continuous map $f:(X,\lambda)\rightarrow (X',\lambda')$,
      if $\varrho: K\rightarrow K'$
    is a simplicial approximation of $f:X\rightarrow X'$,
    then $\varrho$ must be a W-simplicial approximation
    of $f$ and $\overline{\varrho}: X\rightarrow X'$ is W-homotopic to $f$.
     \end{lem}
     \begin{proof}
        We assume that $(X,\lambda)$ and $(X',\lambda')$ are of ascending type below. The proof for descending type weighted polyhedra is completely parallel. \n
        
    First of all, since $\varrho: K\rightarrow K'$
    is a simplicial approximation of $f:X\rightarrow X'$,  
        \begin{equation} \label{Equ:Car-sub}
  \overline{\varrho}(x)\in |\mathrm{Car}_{K'}(f(x))|, \ \text{for all}\ x\in X=|K|.
  \end{equation}

 Besides, since $f$ is W-continuous, by Definition~\ref{Def:W-cont-map} we obtain
 \begin{equation} \label{Equ:lambda-sub}
  \lambda'(f(x)) \mid \lambda(x), \ \text{for all}\ x\in X. 
  \end{equation}
   So for any vertex $v$ of $K$, we have
  $\lambda'(f(v)) \mid \lambda(v) = \xi(v)$.    
  Moreover, since $\varrho(v)$ is a vertex of 
 $\mathrm{Car}_{K'}(f(v))$ and by our assumption $(K',\xi')$ is ascending, we obtain
   $$ \xi'(\varrho(v)) \mid
   \xi'\big(\mathrm{Car}_{K'}(f(v))\big) = \lambda'(f(v)). $$
  So we obtain 
  $$\xi'(\varrho(v))\mid \lambda(v)=\xi(v).$$   
   This implies that $\varrho$ is a morphism from
   $(K,\xi)$ to $(K',\xi')$ since $(K,\xi)$ and $(K',\xi')$ are both divisibly weighted (see Lemma~\ref{Lem:Morph-Vertex}). \n
   Next, we prove that $\overline{\varrho}$ is W-homotopic to $f$. Indeed,
   the line segment between $\overline{\varrho}(x)$
   and $f(x)$ in $|\mathrm{Car}_{K'}(f(x))|$ for any $x\in X$ determines a homotopy $H$ from $\overline{\varrho}$ to $f$, that is 
   \begin{equation} \label{Equ:Homotopy-varrho-f}
    H(x,t)=(1-t)\overline{\varrho}(x) + tf(x),\
   t\in[0,1], \ x\in X.
   \end{equation}
  Moreover, by~\eqref{Equ:Car-sub},
   $\mathrm{Car}_{K'}(\overline{\varrho}(x))$ is a face of $\mathrm{Car}_{K'}(f(x))$. This implies
   $$\lambda'(\overline{\varrho}(x)) = \xi'\big( \mathrm{Car}_{K'}(\overline{\varrho}(x)) \big) \mid \xi'\big(\mathrm{Car}_{K'}(f(x)) \big)=\lambda'(f(x)).$$ 
 For $0< t\leq 1$, observe that $H(x,t)$ lies in the relative
  interior of $|\mathrm{Car}_{K'}(f(x))|$, which implies that
  $\mathrm{Car}_{K'}(H(x,t))=\mathrm{Car}_{K'}(f(x))$
  and so $\lambda'(H(x,t))=\lambda'(f(x))$.
  So for any $x\in X$ and $t\in [0,1]$, 
   $$ \lambda'(H(x,t))\mid \lambda'(f(x)) \mid \lambda(x) = \lambda((x,t)).$$
  Therefore, $H$ is a W-continuous map and so
   $H$ is a W-homotopy from $\overline{\varrho}$ to $f$.
     \end{proof}

      Note that even if $f: (X,\lambda)\rightarrow (X',\lambda')$ is weight-preserving, a
  W-simplicial approximation $\varrho: (K,\xi)\rightarrow (K',\xi')$ of $f$ is not necessarily weight-preserving by our construction in Lemma~\ref{Lem:W-simp-Approx-0}. But
  we will prove in the following lemma that there always exists a
  weight-preserving simplicial approximation for a 
  weight-preserving map.   
 \n
     
 \begin{lem}\label{Lem:Weight-Preserv-Simp-Approx}
  Suppose $(X,\lambda)$ and $(X',\lambda')$ are two weighted polyhedra of the same type. If $f:(X,\lambda)\rightarrow (X',\lambda')$ is a weight-preserving continuous map, then there exists a weight-preserving simplicial approximation of $f$.
 \end{lem} 
 \begin{proof}  
 We assume $(X,\lambda)$ and $(X',\lambda')$ are ascending type weighted polyhedra in our proof. The proof for the descending type is parallel.  
 By Lemma~\ref{Lem:W-simp-Approx-0}, we can first 
 construct 
 an ordinary W-simplicial approximation of $f$, denoted by 
 $$\varrho: (K,\xi)\rightarrow (K',\xi'),$$
  where $(K,\xi)$ and $(K',\xi')$ can be assumed to be divisibly weighted triangulations of $(X,\lambda)$ and $(X',\lambda')$, respectively (see Corollary~\ref{Cor:Div-Weight-Triangl}). Then in the following, we show how to construct a weight-preserving simplicial approximation of $f$ 
 via iterated barycentric subdivisions of $K$. 
 Note that
 the strata of $(X,\lambda)$ may not be compact. So we cannot directly apply the
 ordinary simplicial approximation theorem
 to each stratum of $(X,\lambda)$ to prove our result.  \n

 Let $(K_n, \xi)$ and $(K'_{n'}, \xi')$, $n\in \mathrm{Im}(\xi)$ and $n'\in\mathrm{Im}(\xi')$, be the W-filtrations of $(K,\xi)$ and $(K',\xi')$, respectively (see Definition~\ref{Def:W-filtration}). Note that $\mathrm{Im}(\xi)$ is a subset of $\mathrm{Im}(\xi')$ since $f$ is weight-preserving.
   Let $n_0$ be the largest integer in $\mathrm{Im}(\xi)$
   so that $\varrho$ is weight-preserving on $K_{n_0}$.
       Let $n_1$ be the
   smallest integer in $\mathrm{Im}(\xi)$ 
   that is greater than $n_0$.\n
   
   Let $u$ be an arbitrary vertex of $K$ with $\xi(u)=n_1$.
   If $\xi'(\varrho(u))=n_1$, then we do nothing.
   Otherwise,
   $\xi'(\varrho(u))< n_1$ since $\xi'(\varrho(u)) \mid \xi(u)$. So $\xi'(\varrho(u))$
   is a vertex of $K'_{n_0}$. Moreover, since $\varrho$
   is a W-simplicial approximation of $f$, we have 
   $$f ( \mathrm{St}(u,K) )  \subseteq  
   \mathrm{St}(\varrho(u),K') \subseteq 
     |N(K'_{n_0},K')|. $$  
  On the other hand, the weight of any point in $\mathrm{St}(u,K)$ is at least $n_1$ (since $\xi$ is an ascending weight).
 So $f$ being weight-preserving implies that (see Figure~\ref{p:Pushing-Star})
     $$f ( \mathrm{St}(u,K) )  \subseteq 
     |N(K'_{n_0},K')|  \backslash |K'_{n_0}|.$$
   
   By Lemma~\ref{Lem:Deformation-Retract}, there is a weight-preserving deformation retraction $G^{n_0}$
   from $|N(K'_{n_0},K')|  \backslash |K'_{n_0}|$ onto $|S(K'_{n_0},K')|$.
   So there exists a $\varepsilon_u\in (0,1)$ such that
    $$ G^{n_0}_{\varepsilon_u}(f(\mathrm{St}(u,K))) \subseteq U\big(S(K'_{n_0},K'),K'\big) \overset{\eqref{Equ:U-L-K}}{=} \bigcup_{u'\in S(K'_{n_0},K')^{(0)}} \mathrm{St}(u',K'). $$ 
   
   \begin{figure}[h]
         \includegraphics[width=0.68\textwidth]{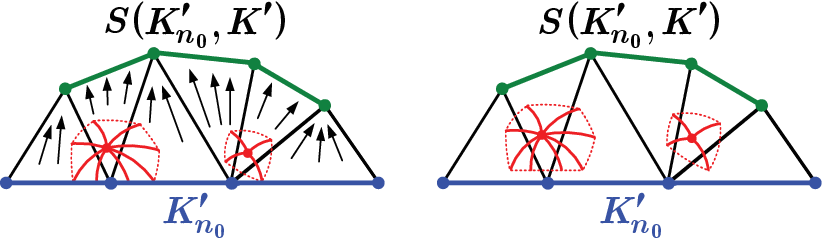}\\
          \caption{Pushing an open star away from $|K'_{n_0}|$}\label{p:Pushing-Star}
      \end{figure}
      
     Since the restriction of $G^{n_0}$ to each simplex is linear, we can easily prove the following claim.
      \n
      
   \begin{itemize}
   \item[\textbf{Claim-1:}] For a set $B \subseteq |N(K'_{n_0},K')  |\backslash |K'_{n_0}|$ and $\varepsilon\in (0,1)$,
     $G^{n_0}_{\varepsilon}(B)$ is contained in $\mathrm{St}(u',K')$ for some vertex $u'\in S(K'_{n_0},K')$ if and only if $B$ itself is contained in $\mathrm{St}(u',K')$.
   \end{itemize} \n
   
   Next, consider the continuous extension of
   $G^{n_0}_{\varepsilon_u}\circ f$ from the open star
   $\mathrm{St}(u,K)$ to the closed star
   $\overline{\mathrm{St}}(u,K)$ along the rays emitting from $u$, denoted by $\overline{G^{n_0}_{\varepsilon_u}\circ f}$. By our construction, 
    $ \overline{G^{n_0}_{\varepsilon_u}\circ f}\big(\overline{\mathrm{St}}(u,K)) \big)$ is the closure of 
   $G^{n_0}_{\varepsilon_u}(f(\mathrm{St}(u,K)))
    $ in $X'$, which is still contained in $U\big(S(K'_{n_0},K'),K'\big) $. 
    Then by the existence of simplicial approximation for the continuous map 
      $$ \overline{G^{n_0}_{\varepsilon_u} \circ f} :  
    |\mathrm{star}(u,K)|=  \overline{\mathrm{St}}(u,K) \rightarrow U\big(S(K'_{n_0},K'),K'\big) ,$$
       there exists an
   integer $m_u$ such that: for any $l\geq m_u$
   and every vertex $v$ of $Sd^{l}(\mathrm{star}(u,K))$, the open star $\mathrm{St}\big(v, Sd^{l}(\mathrm{star}(u,K)) \big)$ is mapped by $\overline{G^{n_0}_{\varepsilon_u} \circ f}$ into the open star 
  $\mathrm{St}(v',K')$ of some vertex $v'\in S(K'_{n_0},K')$. \n
  
  Since there are only finitely many
  simplices in $K$, we can define
  $$ \varepsilon= \underset{\xi(u)=n_1}{\mathrm{max}} \varepsilon_u \in (0,1), \ \ \ m= \underset{\xi(u)=n_1}{\mathrm{max}} m_u. $$
  Moreover, since $\varrho$ is a simplicial approximation
  of $f$, we can define a simplicial approximation of
   $f$, denoted by
    $$\varrho^{m}: Sd^m(K)\rightarrow K'$$
     with
    $\varrho^m(v)=\varrho(u)\in K'$ where $u$ is a vertex of $K$ with $v\in \mathrm{St}(u,K)$. Note that such a vertex $u$ is not unique in general.
   \n
  
  \begin{itemize}
   \item[\textbf{Claim-2:}] We can define $\varrho^m$ properly so that it preserves the weight of any vertex 
  of $Sd^m(K)$ with weight less or equal to $n_1$.  
   \end{itemize}\n        
   
    Let $v$ denote a vertex in $Sd^m(K)$ with weight $Sd^m(\xi)(v)=\lambda(v)\leq n_1$ below. \n
    \begin{itemize}
\item[Case 1.] If $\lambda(v)=n_1$, then $v$ must lie in   $\mathrm{St}(u,K)$ for some vertex $u$ of $K$ with $\xi(u)=n_1$ and so
 $\mathrm{St}\big(v, Sd^{m}(\mathrm{star}(u,K)) \big) \subseteq \mathrm{St}(u,K)$. Moreover, by the preceding argument, $G^{n_0}_{\varepsilon} \circ f$ maps
 $\mathrm{St}\big(v, Sd^{m}(\mathrm{star}(u,K)) \big)$ into $\mathrm{St}(v',K')$ for some vertex $v'\in S(K'_{n_0},K')$. Then by Claim-1, we can assert that
  $$ f \big(\mathrm{St}\big(v, Sd^{m}(\mathrm{star}(u,K)) \big) \big) \subseteq \mathrm{St}(v',K').$$
  So we can define $\varrho^m$ at such a vertex $v$ by
   $\varrho^m(v)=v'$. Notice that $\xi'(v')\geq n_1>n_0$ since $v'\in S(K'_{n_0},K')$. But on the other hand,
   $\xi'(v')$ cannot be greater than $n_1$ since otherwise
   $f(v)\in \mathrm{St}(v',K')$ would have weight
   greater than $n_1$ as well, which contradicts the assumption that $f$ is weight-preserving. Therefore, 
   we must have $\xi'(v')=n_1=\lambda(v)$.
     \n
 
    \item[Case 2.] If $\lambda(v)<n_1$, then $v\in |K_{n_0}|$. So
   there exists a vertex $v^*$ in $K_{n_0}$ with
    $$   v\in \mathrm{St}(v^*,K)\ \text{and}\  
    \lambda(v)=\xi(v^*).$$ 
    This implies that $
    \mathrm{St}(v,Sd^{m}(\mathrm{star}(u,K))) \subseteq \mathrm{St}(v^*,K)$. 
         In addition, since $\varrho$
    is assumed to preserve the weights
   of all the vertices of $K_{n_0}$, we obtain 
    $$  \xi'(\varrho(v^*))=\xi(v^*)=\lambda(v).$$
     Moreover, since $\varrho$ is a W-simplicial approximation of $f$, 
    $$f\big(\mathrm{St}(v,Sd^{m}(\mathrm{star}(u,K)))\big) \subseteq f\big(\mathrm{St}(v^*,K)\big)
    \subseteq \mathrm{St}(\varrho(v^*),K'). $$
      So we can define $\varrho^m$ at such a vertex $v$ by
   $\varrho^m(v)=\varrho(v^*)$. This proves Claim-2.
    \end{itemize}
 \n 
 
  By repeating the above modification of $\varrho$
   for all the weight values in $\mathrm{Im}(\xi)$ greater than $n_0$ from small to large, we will obtain a weight-preserving simplicial map $\varrho^N: (Sd^N(K),Sd^N(\xi))\rightarrow (K',\xi')$ for some large enough integer $N$ which is a simplicial approximation
    of $f$.
  \end{proof}

  \begin{thm}[Existence of W-simplicial Approximation] \label{Thm:W-Simpl-Approx-Exist}
 Suppose $(X,\lambda)$ and $(X',\lambda')$ are two weighted polyhedra of the same type and
$f: (X,\lambda)\rightarrow (X',\lambda')$ is a W-continuous map. Let $(K,\xi)$ and $(K',\xi')$ be divisibly weighted triangulations of $(X,\lambda)$ and $(X',\lambda')$, respectively.  
 \begin{itemize}
 \item[(a)] There always exists a large enough integer $m$ such that $f$ has a W-simplicial approximation
  $\varrho: (Sd^m(K),Sd^m(\xi))\rightarrow (K',\xi')$.  Moreover, if the map $f$ is weight-preserving, we can choose $\varrho$ to be weight-preserving as well
 and $\overline{\varrho}$ is strongly W-homotopic to $f$.\n
  
 \item[(b)] If $\varrho': (Sd^m(K),Sd^m(\xi))\rightarrow (K',\xi')$ is another W-simplicial approximation of $f$, then $\varrho'$ is contiguous to $\varrho$ and so 
  $$\varrho_*=
    \varrho'_*: H_*\big(Sd^m(K),\partial^{Sd^m(\xi)}\big )\rightarrow H_*(K',\partial^{\xi'}).$$
    
  \item[(c)] If $f$ is weight-preserving and
     $\varrho$ and $\varrho'$ are both weight-preserving simplicial approximations of $f$, then $\varrho$ and $\varrho'$ are strongly contiguous.   
 \end{itemize}
  \end{thm}
  \begin{proof}
   (a): The existence of such a W-simplicial approximation
 $\varrho$ of $f$ follows from 
   Theorem~\ref{Thm:Simp-Approx} and 
   Lemma~\ref{Lem:W-simp-Approx-0} immediately.
   When $f$ is weight-preserving, we can choose
   $\varrho$ to be weight-preserving by the proof of Lemma~\ref{Lem:Weight-Preserv-Simp-Approx}. Moreover,
   the homotopy $H$ between $\overline{\varrho}$ and $f$
   defined in~\eqref{Equ:Homotopy-varrho-f} 
   is weight-preserving in this case since all the points in
   $|\mathrm{Car}_{K'}(f(x))|^{\circ}$ have the same weight as $f(x)$. So $\overline{\varrho}$ is strongly W-homotopic to $f$.
    \n

    (b) and (c): For a simplex $\sigma=\{v_0,\cdots,v_n\}$ of $Sd^m(K)$, choose a point
    $x\in |\sigma|^{\circ}$. Then since $\varrho$ and $\varrho'$
   are both W-simplicial approximations of $f$,    
    the vertices
   $\varrho(v_0),\cdots,\varrho(v_n), \varrho'(v_0),\cdots, \varrho'(v_n)$ must all belong to the simplex $\mathrm{Car}_{K'}(f(x))$. Then by Lemma~\ref{Lem:Contiguity}, $\varrho'$ is contiguous to $\varrho$. Moreover, if
    $\varrho$ and $\varrho'$ are both weight-preserving, by Lemma~\ref{Lem:Contiguity} $\varrho$ is strongly contiguous to $\varrho'$. 
  \end{proof}

  \begin{defi} \label{Defi:Induced-Homology}
    Let 
  $f: (X,\lambda)\rightarrow (X',\lambda')$ be a W-continuous map between two weighted polyhedra $(X,\lambda)$ and $(X',\lambda')$ of the same type. Let $(K,\xi)$ and $(K',\xi')$ be divisibly weighted triangulations of $(X,\lambda)$ and $(X',\lambda')$, respectively. If
  $\varrho: Sd^m(K) \rightarrow K'$
  is a W-simplicial approximation of $f$, define 
  \begin{equation} \label{Equ:f-star}
   f_*:= \varrho_* \circ Sd^m_* : 
   H_*( K, \partial^{\xi}) \rightarrow 
    H_*( Sd^m(K),\partial^{Sd(\xi)}) 
   \rightarrow H_*(K', \partial^{\xi'}).
   \end{equation}
     
Note that here $\varrho$ may not induce a morphism
  from $(K, \widehat{Sd^m(\xi)})$ to $(K', \widehat{\xi}')$ if $\varrho$ is not weight-preserving. But when $f$ is weight-preserving, we can choose $\varrho$ to be weight-preserving by Theorem~\ref{Thm:W-Simpl-Approx-Exist}. Then since
   $(K, \widehat{Sd^m(\xi)})$ and $(K', \widehat{\xi}')$
   have the same weight at each vertex as 
    $(K, Sd^m(\xi))$ and $(K', \xi')$, respectively, 
   we obtain a weight-preserving simplicial map 
       $$ \widehat{\varrho}: (Sd^m(K),\widehat{Sd^m(\xi)})\rightarrow (K',\widehat{\xi'}).$$
     By composing with the chain map $\widehat{Sd}_{\#}$ (see Theorem~\ref{Thm:Sd-Homology-Isom-1}), we obtain
     \[ \widehat{\varrho}_{\#}\circ \widehat{Sd}_{\#}^m : 
   C_*(K, \partial^{\widehat{\xi}}) \rightarrow  C_*(Sd^m(K),\partial^{\widehat{Sd^m(\xi)}})\rightarrow C_*(K',\partial^{\widehat{\xi'}}) . \] 
    which further induces a map on the homology groups, denoted by
     \begin{equation} \label{Equ:f-star-hat}
      \widehat{f}_*:= \widehat{\varrho}_* \circ \widehat{Sd}^m_* : 
   H_*( K, \partial^{\widehat{\xi}}) \rightarrow 
    H_*( Sd^m(K),\partial^{\widehat{Sd^m(\xi)}}) 
   \rightarrow H_*(K', \partial^{\widehat{\xi'}}).
   \end{equation}
      \end{defi}
  
 By the following lemma, both $f_*$ and $\widehat{f}_*$ are well-defined (i.e.
 independent on the W-simplicial approximation $\varrho$).
 
  \begin{lem} \label{Lem:Well-Def}
 Let $(X,\lambda)$ and $(X',\lambda')$ be two weighted polyhedra of the same type. Suppose $f: (X,\lambda)\rightarrow (X',\lambda')$ is a W-continuous map. If 
 $$\varrho: (Sd^m(K),Sd^m(\xi)) \rightarrow (K',\xi'),\ \   \kappa: (Sd^{m+r}(K), Sd^{m+r}(\xi)) \rightarrow (K',\xi')$$
  are both W-simplicial approximations of $f$, then we have
  $$\varrho_* \circ Sd^m_* = \kappa_* \circ Sd^{m+r}_*.$$
  If moreover $f$, $\varrho$ and $\kappa$ are all weight-preserving maps, then we also have 
  $$ \widehat{\varrho}_* \circ \widehat{Sd}^m_* = 
  \widehat{\kappa}_* \circ \widehat{Sd}^{m+r}_*.$$
   \end{lem}
  \begin{proof}
  By the proof of Theorem~\ref{Thm:Sd-Homology-Isom-1}, we have the following diagram
  \[  \xymatrix{
      H_*(K,\xi) \ar[d]_{Sd^m_*} &  &  H_*(K',\xi')  \\
        H_*(Sd^m(K),Sd^m(\xi)) \ar[urr]_{\varrho_*} 
         \ar@<-.5ex>[d]_{Sd^r_*}&  & \\
     H_*(Sd^{m+r}(K), Sd^{m+r}(\xi)) \ar[uurr]_{\kappa_*} \ar@<-.5ex>[u]_{(\pi^{\xi}_*)^r} & &
                 }                  
  \]
   where $\pi^{\xi}$ is defined as in~\eqref{Equ:sigma-cond} which is weight-preserving. Clearly,
   $$ \varrho \circ (\pi^{\xi})^r: 
       (Sd^{m+r}(K), Sd^{m+r}(\xi)) \rightarrow (K',\xi')$$  is
   also a W-simplicial approximation of $f$. So by 
   Theorem~\ref{Thm:W-Simpl-Approx-Exist}\,(b), 
   $$\kappa_* = 
   \varrho_* \circ (\pi^{\xi}_*)^r.$$
    Then we obtain
   \begin{equation} \label{Equ:composite-equal}
     \kappa_* \circ Sd^{m+r}_* = 
     \varrho_* \circ (\pi^{\xi}_*)^r \circ Sd^{r}_* \circ Sd^{m}_* \overset{\eqref{Equ:Iso-Comp-H-2}}{=}
     \varrho_* \circ Sd^m_*.
    \end{equation}
      
  If $f$, $\varrho$ and $\kappa$ are all weight-preserving, then  
   $\varrho \circ (\pi^{\xi})^r$ is strongly contiguous
   to $\kappa$ by Theorem~\ref{Thm:W-Simpl-Approx-Exist}\,(c).  This implies: 
    $$\kappa_* = 
   \varrho_* \circ (\pi^{\xi}_*)^r, \ \ \
   \widehat{\kappa}_* = 
   \widehat{\varrho}_* \circ (\widehat{\pi}^{\xi}_*)^r.$$
  Similarly to~\eqref{Equ:composite-equal}, we can prove $ \widehat{\kappa}_* \circ \widehat{Sd}^{m+r}_* =\widehat{\varrho}_* \circ \widehat{Sd}^m_*$ from the following diagram:
   \[  \xymatrix{
      H_*(K,\widehat{\xi}) \ar[d]_{\widehat{Sd}^m_*} &  &  H_*(K',\widehat{\xi}')  \\
        H_*\big(Sd^m(K),\widehat{Sd^m(\xi)} \big) \ar[urr]_{\widehat{\varrho}_*} 
         \ar@<-.5ex>[d]_{\widehat{Sd}^r_*}&  & \\
     H_*\big(Sd^{m+r}(K), \widehat{Sd^{m+r}(\xi)} \big) \ar[uurr]_{\widehat{\kappa}_*} \ar@<-.5ex>[u]_{(\widehat{\pi}^{\xi}_*)^r} & &
                 }                  
  \]
    \end{proof}

  \begin{lem}\label{Lem:Functorial}
   Suppose $(X,\lambda), (X',\lambda')$ and $(X'',\lambda'')$ are weighted polyhedra of the same type. 
   \begin{itemize}
    \item[(a)] For any divisibly weighted triangulation $(K,\xi)$ of $(X,\lambda)$ and the identity map
  $\mathrm{id}_X : X\rightarrow X$, the induced 
  map $(\mathrm{id}_X)_* :  H_*(K,\partial^{\xi}) \rightarrow
   H_*(K,\partial^{\xi}) $ and $(\widehat{\mathrm{id}}_X)_* :  H_*(K,\partial^{\widehat{\xi}}) \rightarrow
   H_*(K,\partial^{\widehat{\xi}}) $
   are both the identity maps.\n
  
  \item[(b)] For W-continuous maps $f: (X,\lambda)\rightarrow (X',\lambda')$ and $g: (X',\lambda')\rightarrow (X'',\lambda'')$,
  \begin{align*}
     (g\circ f)_* &= g_*\circ f_* : H_*(K, \partial^{\xi})
  \rightarrow  H_*(K'', \partial^{\xi''}),
  \end{align*}
  where $(K,\xi)$, $(K',\xi')$ and $(K'',\xi'')$
  are divisibly weighted triangulations of
  $(X,\lambda)$, $(X',\lambda')$ and  $(X'',\lambda'')$, respectively. If moreover, $f$ and $g$ are both weight-preserving, then we also have
  \begin{align*}
     (\widehat{g\circ f})_* &= \widehat{g}_*\circ \widehat{f}_* : H_*(K, \partial^{\widehat{\xi}})
  \rightarrow  H_*(K'', \partial^{\widehat{\xi''}}),
  \end{align*}
  \end{itemize}  
  \end{lem}
  \begin{proof}
   (a) Take $\mathrm{id}_K: (K,\xi)\rightarrow (K,\xi)$ as the W-simplicial approximation of $\mathrm{id}_{X}$.\n
   
   (b) Let $\psi: (Sd^r(K'),Sd^r(\xi'))\rightarrow (K'',\xi'')$ be a W-simplicial approximation of $g$.
   Then using $(Sd^r(K'),Sd^r(\xi'))$ as the triangulation
   of $(X',\lambda')$, we can further obtain a 
    W-simplicial approximation $\varphi: (Sd^m(K),Sd^m(\xi)) \rightarrow (Sd^r(K'),Sd^r(\xi'))$ for $f$. Then since $\pi^{\xi}$ is a weight-preserving simplicial map, the composite map
$$(\pi^{\xi})^r\circ \varphi: (Sd^m(K),Sd^m(\xi)) \rightarrow (K',\xi')$$ is also a W-simplicial approximation of $f$ (see the following diagram).
   \[        
        \xymatrix{
           H_*(K,\partial^{\xi})\ar[d]_{Sd^m_*} \ar[r]^{f_*}
                &  H_*(K', \partial^{\xi'}) \ar@<.5ex>[d]^{Sd^r_*} \ar[r]^{g_*} & H_*(K'',\partial^{\xi''}) \\
         H_*(Sd^m(K), \partial^{Sd^m(\xi)})  \ar[r]^{\varphi_*} &  H_*(Sd^r(K'), \partial^{Sd^r(\xi')}) \ar@<.5ex>[u]^{(\pi^{\xi}_*)^r} \ar[ur]_{\psi_*}
                 } 
   \]
   
   Since
   $\psi\circ\varphi: (Sd^m(K),Sd^m(\xi))\rightarrow (K'',\xi'')$ is a W-simplicial approximation of $g\circ f$, we have $ (g\circ f)_* = (\psi\circ\varphi)_*\circ Sd^m_*$.
   On the other hand,
   \begin{align*}
     g_*\circ f_* &=(\psi_*\circ Sd^r_* )\circ   ((\pi^{\xi}_*)^r\circ \varphi_* \circ Sd^m_*)\\
     &\overset{\eqref{Equ:Iso-Comp-H-2}}{=} \psi_*\circ \varphi_* \circ Sd^m_* = (\psi\circ\varphi)_*\circ Sd^m_*
     =(g\circ f)_*.
   \end{align*}
   The proof for $ (\widehat{g\circ f})_* = \widehat{g}_*\circ \widehat{f}_*$ is completely parallel
   by arguing with the following diagram where we choose both $\varphi$ and $\psi$ to be weight-preserving:
   \[        
        \xymatrix{
           H_*(K,\partial^{\widehat{\xi}})\ar[d]_{\widehat{Sd}^m_*} \ar[r]^{\widehat{f}_*}
                &  H_*(K', \partial^{\widehat{\xi'}}) \ar@<.5ex>[d]^{\widehat{Sd}^r_*} \ar[r]^{\widehat{g}_*} & H_*(K'',\partial^{\widehat{\xi''}}) \\
         H_*(Sd^m(K), \partial^{\widehat{Sd^m(\xi)}})  \ar[r]^{\widehat{\varphi}_*} &  H_*(Sd^r(K'), \partial^{\widehat{Sd^r(\xi')}}) \ar@<.5ex>[u]^{(\widehat{\pi}^{\xi}_*)^r} \ar[ur]_{\widehat{\psi}_*}
                 }  
   \]
   \end{proof}

  \begin{thm} \label{Thm:Isomor}
   Suppose $f: (X,\lambda) \rightarrow (X',\lambda')$ is an isomorphism between two weighted polyhedra
   $(X,\lambda)$ and $(X',\lambda')$. If
   $(K,\xi)$ and $(K',\xi')$ are divisibly weighted triangulations of $(X,\lambda)$  and  $(X',\lambda')$, respectively, then the induced homomorphisms $f_*:  H_*( K, \partial^{\xi})  
   \rightarrow H_*(K', \partial^{\xi'})$ and 
   $\widehat{f}_*:  H_*(K, \partial^{\widehat{\xi}})  
   \rightarrow H_*(K', \partial^{\widehat{\xi'}})$ are both isomorphisms.
  \end{thm}
  \begin{proof}
   Let $g=f^{-1}: (X',\lambda')\rightarrow (X,\lambda)$.
   Then by Lemma~\ref{Lem:Functorial}, 
   $$ g_* \circ f_* = (g\circ f)_* = (\mathrm{id}_X)_*, \ \ 
   f_* \circ g_* = (f\circ g)_* = (\mathrm{id}_{X'})_* $$
are both identity maps.  So $f_*$ is an isomorphism. 
The proof for $\widehat{f}_*$ is parallel.
  \end{proof}
  
   We can obtain the following corollary immediately from Theorem~\ref{Thm:Isomor}.  
 
  \begin{cor} \label{Cor:Indepen}
   If $(K,\xi)$ and $(K',\xi')$ are divisibly weighted
   triangulations of the same weighted polyhedron $(X,\lambda)$, then the identity map $\mathrm{id}_X : X\rightarrow X$ determines isomorphisms
   $$(\mathrm{id}_X)_*:  H_*(K, \partial^{\xi})  
  \rightarrow H_*(K', \partial^{\xi'}), \ \
  (\widehat{\mathrm{id}}_X)_*:  H_*(K, \partial^{\widehat{\xi}})  
  \rightarrow H_*(K', \partial^{\widehat{\xi'}}).$$
   \end{cor} 
  
  The above corollary implies that the AW-homology and DW-homology of a weighted polyhedron in Definition~\ref{Defi:Pseudo-Orbi-Homol} are independent on the divisibly weighted triangulation we choose. So these two notions are both well-defined.   
 Moreover, we can conclude from Theorem~\ref{Thm:Isomor} that AW-homology and DW-homology groups are invariants of weighted polyhedra under isomorphisms.
  
   \begin{cor} \label{Cor:Isom-AW-DW}
    Suppose $f: (X,\lambda) \rightarrow (X',\lambda')$ is an isomorphism between two weighted polyhedra
   $(X,\lambda)$ and $(X',\lambda')$. 
   \begin{itemize}
   \item If $(X,\lambda)$ and $(X',\lambda')$ are of ascending type, then $f$ induces isomorphisms:
    $$\qquad f_*: H^{AW}_*(X,\lambda)\rightarrow H^{AW}_*(X',\lambda'), \ \ \widehat{f}_*: H^{DW}_*(X,\lambda)\rightarrow H^{DW}_*(X',\lambda').$$  
      \item If $(X,\lambda)$ and $(X',\lambda')$ are of descending type, then $f$ induces isomorphisms:
       $$\qquad f_*: H^{DW}_*(X,\lambda)\rightarrow H^{DW}_*(X',\lambda'), \ \ \widehat{f}_*: H^{AW}_*(X,\lambda)\rightarrow H^{AW}_*(X',\lambda').$$ 
     \end{itemize}   
   \end{cor}
    
  \n
  
  Next, we prove the invariance of AW-homology and DW-homology with respect to (strong) W-homotopy equivalences of weighted polyhedra.

 \begin{lem} \label{Lem:W-homotopy-Induce}
  Let $(X,\lambda)$ and $(X',\lambda')$ be weighted polyhedra of the same type. 
  Suppose two W-continuous maps 
   $f,g : (X,\lambda)\rightarrow (X',\lambda')$ are  W-homotopic.
   \begin{itemize}
    \item[(a)] If $(X,\lambda)$ and $(X',\lambda')$
    are of ascending type, then
    $$  f_*=g_*: H^{AW}_*(X,\lambda)\rightarrow H^{AW}_*(X',\lambda').$$
    
    \item[(b)] If $(X,\lambda)$ and $(X',\lambda')$
    are of descending type, then
    $$  f_*=g_*: H^{DW}_*(X,\lambda)\rightarrow H^{DW}_*(X',\lambda').$$  
    
    \item[(c)] If $f$ and $g$ are both weight-preserving and are strongly W-homotopic, then they induce the same homomorphisms on both AW-homology groups and DW-homology groups. 
  \end{itemize}
 \end{lem}
 \begin{proof}
   Let $G : (X \times [0, 1],\lambda\times\mathbf{1}) \rightarrow (X',\lambda')$ be a W-homotopy from $f$ to $g$.     
   Suppose $(K,\xi)$ and $(K',\xi')$ are divisibly weighted triangulations of $(X,\lambda)$ and $(X',\lambda')$, respectively. So $(K\times [0,1],\xi\times\mathbf{1})$ is a divisibly weighted triangulation of $(X \times [0, 1],\lambda\times\mathbf{1})$. Then $G$ induces a homomorphism   (see~\eqref{Equ:f-star})
    $$ G_*: H_*(K\times [0,1],\partial^{\xi\times\mathbf{1}}) \rightarrow H_*(K',\partial^{\xi'}).$$
   In addition, let $i_0, i_1: X \rightarrow X\times [0,1]$
   be the maps:
   \[ i_0(x)=(x,0), \  i_1(x)=(x,1), \ x\in X. \]
   Then $f=G \circ i_0$ and $g= G \circ i_1$.  Besides,
   we also consider $i_0$ and $i_1$ as simplicial maps
   from $(K,\xi)$ to $(K\times [0,1],\xi\times\mathbf{1})$ which are both weight-preserving.\n
  
 \textbf{Claim:} $(i_0)_{\#}, (i_1)_{\#}: (C_*(K),\partial^{\xi}) \rightarrow (C_*(K\times [0,1]),\partial^{\xi\times\mathbf{1}})$ are chain homotopic.
 \n
 
 Consider the function $\Phi$ assigning, to each simplex
 $\sigma$ of $K$, the subchain complex $(C_*(\sigma\times [0,1]),\partial^{\xi\times\mathbf{1}})$ of $(C_*(K\times [0,1]),\partial^{\xi\times\mathbf{1}})$. By Proposition~\ref{Prop:K-I-Product} and Lemma~\ref{Lem:Wt-Div-Simplex}, 
  $$H_*(\sigma\times [0,1],\partial^{\xi\times\mathbf{1}})  \cong H_*(\sigma,\partial^{\xi})\cong \begin{cases}
   \Z ,  &  \text{if $j=0$}; \\
   0 ,  &  \text{if $j \geq 1$}.
 \end{cases}$$
  
  So $\Phi$ is an (algebraic) acyclic carrier from
 $(C_*(K),\partial^{\xi})$ to $(C_*(K\times [0,1]),\partial^{\xi\times\mathbf{1}})$.
 Moreover, both $(i_0)_{\#}$ and $(i_1)_{\#}$ are 
 carried by $\Phi$ since both $i_0(\sigma)=\sigma\times \{0\}$ and $i_1(\sigma)=\sigma\times \{1\}$ belong to
 $\Phi(\sigma)$. Then it follows from Theorem~\ref{Thm:Acyclic-Carrier} that
 $(i_0)_{\#}$ is chain homotopic to $(i_1)_{\#}$.
 The claim is proved.\n
 
 By the above claim, we obtain $(i_0)_*=(i_1)_*:
 H_*(K,\partial^{\xi}) \rightarrow H_*(K\times [0,1],\partial^{\xi\times\mathbf{1}})$.
 So by Lemma~\ref{Lem:Functorial} and 
 the definitions of AW-homology and DW-homology,
  $$ f_* = G_* \circ (i_0)_* = G_*\circ(i_1)_* = g_*.$$ 
  This proves (a) and (b).\n
  
  As for (c), since $i_0$ and $i_1$ are both weight-preserving, they also induce chain maps from
  $(C_*(K),\partial^{\widehat{\xi}})$ to $(C_*(K\times [0,1]),\partial^{\widehat{\xi\times\mathbf{1}}})$, denoted by:
  \[ (\widehat{i_0})_{\#}, (\widehat{i_1})_{\#}: (C_*(K),\partial^{\widehat{\xi}}) \rightarrow (C_*(K\times [0,1]),\partial^{\widehat{\xi\times\mathbf{1}}}) \]
  Observe that $\widehat{\xi\times\mathbf{1}}$ equals 
  $\widehat{\xi}\times\mathbf{1}$ on $K\times [0,1]$. So
  we can use a similar argument as the proof of the above claim to show that 
  $(\widehat{i_0})_{\#}$ and  $(\widehat{i_1})_{\#}$ are chain homotopic. So  $(\widehat{i_0})_{\#}$ and  $(\widehat{i_1})_{\#}$ induce the same homomorphism on homology, i.e.
  $$(\widehat{i_0})_*=(\widehat{i_1})_*:
 H_*(K,\partial^{\widehat{\xi}}) \rightarrow H_*(K\times [0,1],\partial^{\widehat{\xi\times\mathbf{1}}}). $$
  Moreover, since $f$ and $g$ are strongly W-homotopic,
  the homotopy $G$ between $f$ and $g$ can be taken to
 be weight-preserving. So $G$ also induces a map (by~\eqref{Equ:f-star-hat})
  $$\widehat{G}_*:  H_*(K\times [0,1],\partial^{\widehat{\xi\times\mathbf{1}}}) \rightarrow H_*(K',\partial^{\widehat{\xi'}}).$$
  Then we obtain
   $$\widehat{f}_* = \widehat{G}_* \circ (\widehat{i_0})_* = \widehat{G}_*\circ(\widehat{i_1})_* = \widehat{g}_*
   :  H_*(K,\partial^{\widehat{\xi}}) \rightarrow H_*(K',\partial^{\widehat{\xi'}}).$$
 This fact together with (a) and (b) imply (c). 
 \end{proof}

 \begin{thm} \label{Thm:Homotopy-Invar}
  Suppose two weighted polyhedra $(X,\lambda)$ and $(X',\lambda')$ of the same type are W-homotopy equivalent.\n 
 \noindent $\mathrm{(a)}$ If $(X,\lambda)$ and $(X',\lambda')$
    are of ascending type, then
    $H^{AW}_*(X,\lambda)\cong H^{AW}_*(X',\lambda')$.
    \n
\noindent  $\mathrm{(b)}$ If $(X,\lambda)$ and $(X',\lambda')$
    are of descending type, then
    $  H^{DW}_*(X,\lambda)\cong H^{DW}_*(X',\lambda')$. \n
    
\noindent  $\mathrm{(c)}$ If $(X,\lambda)$ and $(X',\lambda')$ are 
    strongly W-homotopy equivalent, then
 $$ H^{AW}_*(X,\lambda)\cong H^{AW}_*(X',\lambda'),\ \
 H^{DW}_*(X,\lambda)\cong H^{DW}_*(X',\lambda').$$
 \end{thm}
 \begin{proof}
  This follows immediately from Lemma~\ref{Lem:Functorial} and Lemma~\ref{Lem:W-homotopy-Induce}. 
 \end{proof}

 \section{AW-homology and DW-homology with coefficients} \label{Sec:Relation} 
   
    Let $(K,\mu)$ be a weighted simplicial complex. Given any abelian group $G$,  we can apply the tensor product functor
   $\otimes G$  to $(C_*(K),\partial^{\mu})$ to obtain a chain complex denoted by
    $$ (C_*(K;G),\partial^{\mu}) := (C_*(K)\otimes G, \partial^{\mu}\otimes \mathrm{id}_G). $$
    Let $H_*(K,\partial^{\mu};G)$ denote the homology group of $(C_*(K;G),\partial^{\mu})$.
   By the algebraic universal coefficient theorem of homology (see Hatcher~\cite[\S\,3.A]{Hatcher02}), 
    \begin{align*} 
    H_j(K,\partial^{\mu};G) &\cong   \big( H_j(K,\partial^{\mu})\otimes G\big)\, \scalebox{1.5}{$\oplus$}\, \mathrm{Tor}\big(H_{j-1}(K,\partial^{\mu}),G\big),\,  j\in \Z.
    \end{align*}
    
 If $A$ is a simplicial subcomplex of $K$,
 we can similarly define $H_*(K,A, \partial^{\mu};G)$ which is the homology group of the chain complex
 $$(C_*(K,A;G),\partial^{\mu}) := \big( (C_*(K)\slash C_*(A))
   \otimes G, \partial^{\mu}\otimes \mathrm{id}_G \big).$$
   
  For a weighted polyhedron $(X,\lambda)$, the \emph{AW-homology and DW-homology of $(X,\lambda)$ with $G$-coefficients} are defined by the corresponding notions 
   of a divisibly weighted triangulation $(K,\xi)$ of $(X,\lambda)$, respectively, denoted by $H^{AW}_*(X,\lambda;G)$ and
   $H^{DW}_*(X,\lambda;G)$. Then by the universal coefficient theorem,
   \begin{align*}
    H^{AW}_j(X,\lambda;G) & \cong   \big( H^{AW}_j(X,\lambda)\otimes G\big)\, \scalebox{1.5}{$\oplus$}\, \mathrm{Tor}\big(H^{AW}_{j-1}(X,\lambda),G\big),\,  j\in \Z; \\
      H^{DW}_j(X,\lambda;G) & \cong   \big( H^{DW}_j(X,\lambda)\otimes G\big)\, \scalebox{1.5}{$\oplus$}\, \mathrm{Tor}\big(H^{DW}_{j-1}(X,\lambda),G\big),\,  j\in \Z.
    \end{align*}
    
   The following theorem tells us that with coefficients in some special fields, the AW-homology and DW-homology of a weighted polyhedron $(X,\lambda)$ 
    are both isomorphic to the ordinary simplicial (or singular) homology of $X$.\n
     
     \begin{thm}\label{thm:Weighted-Homology-Iso}
       Let $(X,\lambda)$ be a weighted polyhedron. If the character of a field $\mathbb{F}$ is relatively prime to the weights of all points of $X$, then 
       $ H^{AW}_*(X,\lambda; \mathbb{F})$ and 
       $ H^{DW}_*(X,\lambda; \mathbb{F})$ are both isomorphic to $H_*(X;\mathbb{F})$.
     \end{thm}
     \begin{proof}
      Let $(K,\xi)$ be a divisibly weighted triangulation of $(X,\lambda)$. For any $n\geq 0$, we can define
       a linear map for the $n$-skeleton $K^{(n)}$ of $K$,     
     as follows
      \begin{align} \label{Def:Xi}
      \Xi: \big((C_*(K^{(n)};\mathbb{F}),\partial^{\xi}\big) &\longrightarrow C_*(K^{(n)};\mathbb{F}).\\
       \sigma &\longmapsto 
        \begin{cases}
   w(\sigma)\cdot\sigma ,  &  \text{if $\xi$ is ascending}; \\
   \frac{1}{w(\sigma)} \cdot \sigma ,  &  \text{if $\xi$ is descending}.
        \end{cases}  \notag
      \end{align}   
   It is easy to check that $\Xi$ is a chain map by the definition of $\partial^{\xi}$ (see~\eqref{Equ:weighted-Boun-AW} and~\eqref{Equ:weighted-Boun-DW}). 
 Note that here $\frac{1}{w(\sigma)}$ is valid because of the assumption on the character of 
 $\mathbb{F}$.\n
      
  By the long exact sequence of weighted simplicial homology
     (\cite[Theorem 2.3]{Daw90}), we obtain a commutative diagram as follows (the coefficients $\mathbb{F}$ are omitted).
    \scriptsize     \[   \xymatrix{
           H_{j+1}(K^{(n)},K^{(n-1)}, \overline{\partial}^{\xi})\ar[d]^{\Xi_*} \ar[r]
                & H_j(K^{(n-1)},\partial^{\xi}) \ar[d]^{\Xi_*} \ar[r] & H_j(K^{(n)},\partial^{\xi}) \ar[d]^{\Xi_*} \ar[r] &   H_{j}(K^{(n)},K^{(n-1)},\overline{\partial}^{\xi}) \ar[d]^{\Xi_*} \ar[r]  & H_{j-1}(K^{(n-1)},\partial^{\xi}) \ar[d]^{\Xi_*}  \\
           H_{j+1}(K^{(n)},K^{(n-1)}) \ar[r]    & H_j(K^{(n-1)}) \ar[r] & H_j(K^{(n)}) \ar[r] & H_{j}(K^{(n)},K^{(n-1)}) \ar[r] &
           H_{j-1}(K^{(n-1)}).
                 } 
                 \]
         \normalsize
         
   For any $n$-simplex $\sigma$ in $K$,
   we can directly compute from the definition that
 $$\Xi_*: H_n(\sigma,\partial \sigma, \overline{\partial}^{\xi};\mathbb{F}) \cong \mathbb{F}  \rightarrow  \mathbb{F} \cong
  H_n(\sigma,\partial \sigma;\mathbb{F})$$ is given by 
   $ \mathbb{F} \xlongrightarrow{\cdot \xi(\sigma)}\mathbb{F}$ if $\xi$ is ascending and $ \mathbb{F} \xlongrightarrow{\cdot \frac{1}{\xi(\sigma)}} \mathbb{F}$ if $\xi$ is descending. 
   So the homomorphism
  $\Xi_*: H_n(\sigma,\partial \sigma,\overline{\partial}^{\xi};\mathbb{F})  \rightarrow
  H_n(\sigma,\partial \sigma;\mathbb{F})$
  is an isomorphism. It follows that 
  $\Xi_* : H_*(K^{(n)},K^{(n-1)}, \overline{\partial}^{\xi};\mathbb{F})
  \rightarrow H_*(K^{(n)},K^{(n-1)}) $ is an isomorphism.  
  Then by the five-lemma (see~\cite[p.\,129]{Hatcher02}), we can inductively prove that $H_*(K^{(n)},\partial^{\xi};\mathbb{F})$ is isomorphic to $H_*(K^{(n)};\mathbb{F})$
         for all $n\geq 0$.         
         So $\Xi_* : H_*(K,\partial^{\xi}; \mathbb{F})\rightarrow H_*(K;\mathbb{F})$ is an isomorphism. Applying the same argument
         to the inversion $\widehat{\xi}$ of $\xi$, we also obtain an isomorphism
         from $ H_*(K,\partial^{\widehat{\xi}}; \mathbb{F})$ to $H_*(K;\mathbb{F})$. Then the theorem follows.
     \end{proof}
     
  In particular, for a weighted polyhedron $(X,\lambda)$ we have isomorphisms with respect to the rational coefficients $\mathbb{Q}$:
    \begin{equation} \label{Equ:Q-Coeff-Iso}
    H^{AW}_*(X,\lambda; \mathbb{Q})\cong  H^{DW}_*(X,\lambda; \mathbb{Q})\cong H_*(X;\mathbb{Q}). 
     \end{equation}  
     
   The above isomorphisms imply that 
   $$H^{AW}_j(X,\lambda) \cong \Z^{\beta_j(X)} \oplus T^{AW}(X,\lambda), \ H^{DW}_j(X,\lambda) \cong \Z^{\beta_j(X)} \oplus T^{DW}(X,\lambda) $$
   where $\beta_j(X)$ is the $j$-th Betti number of $X$ and $T^{AW}(X,\lambda)$ and $T^{DW}(X,\lambda)$
   are torsion groups.   
   So it is the torsion parts of $H^{AW}_*(X,\lambda)$ and $H^{DW}_*(X,\lambda)$ that can really give us new information of the weighted polyhedron that comes from the weight function $\lambda$.

   \section{Examples} \label{Sec:Example}
   
   \subsection{Computation strategies}\ \n
   
   First of all, the excision and Mayer-Vietoris sequence in weighted simplicial homology (see~\cite[Theorem 2.3 and Corollary 2.3.1]{Daw90}) naturally induce the
  \emph{excision} and \emph{Mayer-Vietoris sequence} in
  AW-homology and DW-homology of weighted polyhedra, which are the basic tools for the computation.
  In addition, it is shown in~\cite{WuRenWuXia20} that the method of discrete Morse theory 
  can also be used to compute weighted simplicial homology.
    \n

  Note that the definition of weighted homology groups also makes sense for any $\Delta$-complex
  with a weight function (either ascending or descending). The reader is referred to~\cite{Hatcher02} for the definition of $\Delta$-complex. 
 So a very useful strategy for our computation is: when we compute the AW-homology and DW-homology of a weighted polyhedron $(X,\lambda)$ via a divisibly weighted triangulation $(K,\xi)$, 
  we can replace the simplicial complex $K$ by
  a $\Delta$-complex $L$ as long as $(L,\xi)$ also consists of divisibly weighted simplices.
  We readily call such $(L,\xi)$ a \emph{divisibly weighted $\Delta$-triangulation} of $(X,\lambda)$.
In many cases,  using $\Delta$-complexes can significantly reduce the number of generators of a weighted simplicial chain complex and hence simplify the computation.  
   \n
   
  Another useful notion for our computation is
  algebraic mapping cone (see Davis and Kirk~\cite[Section 11]{DavKirk01}).
  
  \begin{defi}[Algebraic Mapping Cone] \label{Def:Alg-Cone}
   Let $\mathcal{C}_*=(C_p,\partial_p)$ and 
 $\mathcal{C}'_*=(C'_p,\partial'_p)$ be chain complexes.
 For a chain map $\phi:\mathcal{C}_*\rightarrow \mathcal{C}'_*$, the \emph{algebraic mapping cone} of $\phi$
 is the chain complex $Cone(\phi)$ where
 \begin{equation} \label{Equ:Alg-Mapping-Cone}
   Cone(\phi)_p = C_{p-1}\oplus C'_{p}, \ \forall p\in \Z, \ \text{with boundary map}\
     d_p=\begin{pmatrix}
       -\partial_{p-1}  &  0  \\
        \phi &  \partial'_p 
   \end{pmatrix}.
 \end{equation}

     In particular, the \emph{algebraic cone} on
     a chain complex $\mathcal{C}_*$ is 
     $$ Cone(\mathcal{C}_*) := Cone(\mathrm{id}_{\mathcal{C}_*}:
     \mathcal{C}_* \rightarrow \mathcal{C}_*). $$
  \end{defi}   
     
 \begin{lem}[\text{see~\cite[Lemma 11.26]{DavKirk01}}] \label{Lem:Cone-vanish}
    The algebraic cone on a free chain complex is 
    always acyclic.
  \end{lem}
  
 \begin{rem}
  For a simplex $\sigma$, the ordinary simplicial chain complex $C_*(Sd(\sigma))$ is isomorphic to the algebraic cone of
  $C_*(Sd(\partial \sigma))$.  But given a weight function $\mu$ on $\sigma$, the weighted simplicial chain complex $\big(C_*(Sd(\sigma)),\partial^{Sd(\mu)}\big)$ may not be acyclic. So in general, $\big(C_*(Sd(\sigma)),\partial^{Sd(\mu)}\big)$  
 may not be isomorphic to the algebraic cone on $\big(C_*(\partial \sigma),\partial^{Sd(\mu)}\big)$. 
 Indeed, for a simplex $\tau$ of $Sd(\partial \sigma)$, $Sd(\mu)(b_{\sigma}\cdot\tau)=\mu(\sigma)$ may not agree with $Sd(\mu)(\tau)$. 
  \end{rem}

     In the following, we compute the AW-homology and DW-homology of some concrete examples of orbifolds and pseudo-orbifolds in dimension one and two. The reader is referred to~\cite{Thurs77,Scott83,Choi12} for more information on two-dimensional orbifolds. 
     We will use ``$\mathrm{gcd}$'' and ``$\mathrm{lcm}$'' as the abbreviations for greatest common divisor and least common multiple, respectively.\n
      
        \subsection{One-dimensional orbifolds and
        pseudo-orbifolds}\ \n
      
    \begin{exam}[One dimensional compact orbifolds]\label{Exam:1-dim-orbifold}
     It is well known that any one-dimensional compact connected orbifold is isomorphic to one of the following cases (see~\cite{Choi12}).     \begin{itemize}
      \item[(a)] $S^1$ or $[0,1]$ without singular points.\n
     \item[(b)] An interval $\mathcal{I}=[v_1,v_2]\subset \R^1$ with only one singular point $v_1$
     where $G_{v_1}= \Z_2$. \n
     
     \item[(c)] An interval 
     $\mathcal{I}'=[v_1,v_2]\subset \R^1$ with two singular points
     $v_1$, $v_2$ where $G_{v_1}=G_{v_2}=\Z_2$.
 
     \end{itemize}
     
     Note that $\mathcal{I}$ is a descending divisibly weighted $1$-simplex while $\mathcal{I}'$ is not. By Definition~\ref{Def:AW-DW-Orbifold}, we can easily obtain
     $$H^{AW}_j(\mathcal{I})
     \cong H^{AW}_j (\mathcal{I}')
      \cong \begin{cases}
   \Z ,  &  \text{if $j=0$}; \\
   0,  &  \text{if $j\geq 1$}.
 \end{cases}$$
 $$ 
 H^{DW}_j(\mathcal{I})
      \cong \begin{cases}
   \Z,  &  \text{if $j=0$}; \\
   0,  &  \text{if $j\geq 1$}.
 \end{cases} \ \  \ \
H^{DW}_j (\mathcal{I}')
      \cong \begin{cases}
   \Z\oplus \Z\slash 2\Z ,  &  \text{if $j=0$}; \\
   0,  &  \text{if $j\geq 1$}.
 \end{cases}  $$
 
    \end{exam}

    There are much more one-dimensional pseudo-orbifolds based on an interval than orbifolds because the weights of
    the singular points in a pseudo-orbifold could be arbitrary positive integers. For example, let $\mathcal{I}_{(k_1,k_2)}$
    denote the pseudo-orbifold
    $([v_1,v_2],\lambda)$ (see Figure~\ref{Fig:Example-Interval-Sing}) with only two singular points
    $v_1$ and $v_2$ where 
    $$\lambda(v_1)=k_1\geq 2, \ \
    \lambda(v_2)=k_2\geq 2.$$ 
       
     \begin{figure}[h]
        \begin{equation*}
        \vcenter{
            \hbox{
                  \mbox{$\includegraphics[width=0.38\textwidth]{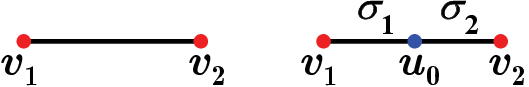}$}
                 }
           }
     \end{equation*}
   \caption{An interval with two singular points at the ends
       } \label{Fig:Example-Interval-Sing}
   \end{figure}  
   
     We can add a vertex $u_0$
    to the interval to obtain a divisibly weighted triangulation of $\mathcal{I}_{(k_1,k_2)}$ with two $1$-simplices $\sigma_1$ and $\sigma_2$. Then the weighted chain complexes for the AW-homology and DW-homology of $\mathcal{I}_{(k_1,k_2)}$ are, respectively:
    \[ \Z^2 \overset{\bordermatrix{%
	       &  \sigma_1   & \sigma_2   \cr
	 u_0    &  k_1       & k_2     \cr
	 v_1    &  1         &  0    \cr
	 v_2    &  0         & 1     
    } }{\xlongrightarrow[\mathrm{AW}]{\qquad \ \qquad\ \qquad\ \quad}} \Z^3, \ \ \ \ \ \Z^2 \overset{\bordermatrix{%
	       &  \sigma_1   & \sigma_2   \cr
	 u_0    &  1       &  1     \cr
	 v_1    &  k_1         &  0    \cr
	 v_2    &  0         & k_2     
    } }{\xlongrightarrow[\mathrm{DW}]{\qquad \ \qquad\ \qquad\ \quad}} \Z^3. \]
   Then it is easy to obtain
     $$H^{AW}_j(\mathcal{I}_{(k_1,k_2)})
     \cong \begin{cases}
   \Z ,  &  \text{if $j=0$}; \\
   0,  &  \text{if $j\geq 1$}.
 \end{cases} \ \ \ H^{DW}_j (\mathcal{I}_{(k_1,k_2)})
      \cong \begin{cases}
   \Z\oplus \Z\slash \mathrm{gcd}(k_1,k_2)\Z ,  &  \text{if $j=0$}; \\
   0,  &  \text{if $j\geq 1$}.
 \end{cases}  $$
 
 \begin{exam}[Line segment with finitely many singular points in the interior]\label{Exam:Segment}
   \ \n
    Let $I_{(k_1,k_2,\dots,k_n)}$ denote 
    the pseudo-orbifold $([a,b],\lambda)$ with $n$ singular points $v_1,\cdots, v_n$ where $ a<v_1<\cdots <v_n<b$ and
    the weight of $v_i$ is $k_i$, $1\leq i \leq n$. Next, we prove by induction that 
    \begin{equation} \label{Equ:AW-I-n}
        H^{AW}_i(I_{(k_1,\dots,k_n)})=
        \begin{cases}
            \Z \oplus \Z_{k_1} \oplus \dots \oplus \Z_{k_n}, & \text{if $i=0$;} \\ 
            0, & \text{otherwise}.
        \end{cases}
    \end{equation}
    where either $[a]$ or $[b]$ is a generator of the 
    $\Z$-factor.
    For $k=1$, $I_{(k_1)}$ is as shown in Figure~\ref{Fig:Line-Seg-1} where the natural triangulation is divisibly weighted.
        \begin{figure}[h]
   	\begin{equation*}
   		\vcenter{
   			\hbox{
   				\mbox{$\includegraphics[width=0.148\textwidth]{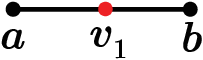}$}
   			}
   		}
   	\end{equation*}
   	\caption{An interval with only one singular point in the interior
   	} \label{Fig:Line-Seg-1}
   \end{figure} 
    
        It is easy to see that 
         $H^{AW}_0(I_{(k_1)})\cong \Z\oplus \Z_{k_1}$
        where $[a-b]$ is a generator of $\Z_{k_1}$, and $[a]$ (or $[b]$) is a generator of $\Z$. 
          Assume that we have prove~\eqref{Equ:AW-I-n} when the number of singular points is less than $n$. For $I_{(k_1,k_2,\dots,k_n)}$, we can subdivide $[a,b]$ into $n+1$ $1$-simplices to obtain a divisibly weighted triangulation as shown in Figure~\ref{Fig:Line-Seg-2}.
           \begin{figure}[h]
   	\begin{equation*}
   		\vcenter{
   			\hbox{
   				\mbox{$\includegraphics[width=0.54\textwidth]{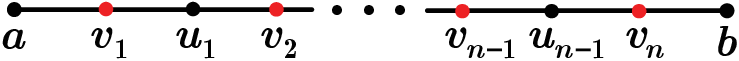}$}
   			}
   		}
   	\end{equation*}
   	\caption{An interval with $n$ singular point in the interior
   	} \label{Fig:Line-Seg-2}
   \end{figure} 
   Consider the interval $[a,u_{n-1}]$ as $I_{(k_1,\dots,k_{n-1})}$, and $[u_{n-1},b]$ as $I_{(k_n)}$. Then by the Mayer-Vietoris sequence for
   $I_{(k_1,\dots,k_n)} = I_{(k_1,\dots,k_{n-1})} \cup I_{(k_n)}$, we obtain
        \[
            H_0^{A W}(\{u_{n-1}\}) \longrightarrow H_0^{A W}(I_{(k_1, \ldots, k_{n-1})}) \oplus H_0^{A W}(I_{(k_n)}) \longrightarrow H_0^{A W}(I_{(k_1, \ldots, k_n)}) \longrightarrow 0.
        \]
    By our induction hypothesis, 
    $H_0^{A W}(I_{(k_1, \ldots, k_{n-1})})\cong \Z\oplus \Z_{k_1}\oplus\cdots\oplus \Z_{k_{n-1}}$ and $H_0^{A W}(I_{(k_n)})\cong \Z\oplus \Z_{k_n}$, where  
        the $\Z$-factor of $H_0^{A W}(I_{(k_1,\cdots,k_{n-1})})$ is generated by $[u_{n-1}]=[a]$ while the $\Z$-factor of $H_0^{A W}(I_{(k_n)})$ is generated by $[u_{n-1}]=[b]$.
        Then the induction step follows from the above long exact sequence. 
   \end{exam}
   
   \begin{rem}
    From the above proof, we can see that $[u_{i}-u_{i+1}]$ is a generator of the $\Z_{k_i}$-factor of $H^{AW}_0(I_{(k_1,\dots,k_n)})$. Then $[a-b]$ is the sum of the generator of the torsion part of $H^{AW}_0(I_{(k_1,\dots,k_n)})$.
   \end{rem}
   
 It is not so easy to compute $H^{DW}_*(I_{(k_1, \ldots, k_{n})})$ by a direct induction on $n$. We will compute $H^{DW}_*(I_{(k_1,\dots,k_n)})$ later (see~\eqref{Equ:DW-I-n}) using the result of the next example.

 \begin{exam}[Circle with finitely many singular points]\label{Exam:1-dim-pseudo-orbifold}
        In Figure~\ref{Fig:Example-Cirle-Sing}, we have a 
   $1$-dimensional pseudo-orbifold $(S^1,\lambda)$, denoted by $\mathcal{S}^1_{(k_1,\cdots,k_n)}$, which is a circle with $n$ singular points $v_1,\cdots, v_n$, $n\geq 1$, whose
  weights $\lambda(v_i)= k_i$, $k_i\geq 2$, $i=1,\cdots,n$. By adding a vertex $u_i$ between $v_i$ and $v_{i+1}$ for each $i$, we obtain a divisibly weighted triangulation $(L,\eta)$ of 
   $\mathcal{S}^1_{(k_1,\cdots,k_n)}$.
    \begin{figure}[h]
   	\begin{equation*}
   		\vcenter{
   			\hbox{
   				\mbox{$\includegraphics[width=0.66\textwidth]{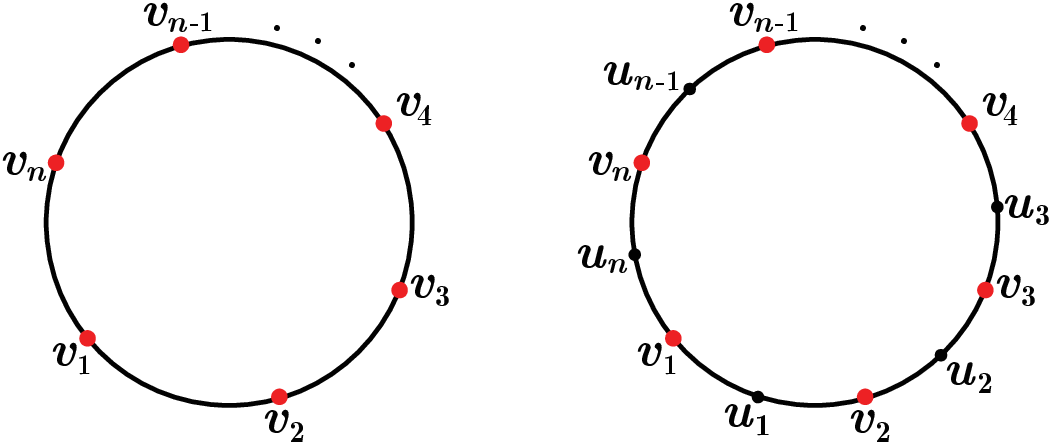}$}
   			}
   		}
   	\end{equation*}
   	\caption{A circle with $m$ singular points
   	} \label{Fig:Example-Cirle-Sing}
   \end{figure}  
   
   We can think of $\mathcal{S}^1_{(k_1,\dots,k_n)}$ as the union of $I_{(k_1,\dots,k_n)}$ and $I=[a,b]$ (an interval without singular points) where $I_{(k_1,\dots,k_n)} \cap I=\{a,b\}$. Then by the Mayer-Vietoris sequence of AW-homology groups, 
   \begin{align*}
 &  H_0^{A W}(\{a, b\}) \xlongrightarrow{i_*} H_0^{A W}\left(I_{\left(k_1, \ldots, k_n\right)}\right) \oplus H_0^{A W}(I) \longrightarrow H_0^{A W}\left(\mathcal{S}_{\left(k_1, \ldots, k_n\right)}^1\right) \longrightarrow 0 \\
 & \quad \ \ \Z\oplus \Z \ \ \longmapsto \ \ (\Z \oplus \Z_{k_1} \oplus \dots \oplus \Z_{k_n} )\oplus \Z  
   \end{align*}
   Choose the generators of $H^{AW}_0(\{a,b\} )$ to be $[a]$ and $[a]-[b]$. The image of $[a]$ in $H_0^{AW} (I_{(k_1,\cdots,k_n)} )$ and $H_0^{AW} (I )$ are both generators of the free abelian part, while the image of $[a]-[b]$ is exactly the sum of the generator of torsion part of $H_0^{AW} (I_{(k_1,\cdots,k_n)} )$. So we obtain
    \begin{equation} \label{Equ:AW-S-n}
            H^{AW}_i(\mathcal{S}^1_{(k_1,\cdots,k_n)})=
            \begin{cases}
                \Z \oplus \Z_{k_1}\langle c_1 \rangle \oplus \cdots \oplus \Z_{k_n}\langle c_n \rangle \slash \langle c_1+\ldots + c_n \rangle, & \text{if $i=0$}; \\ 
                \Z, & \text{if $i=1$};\\
                0, & \text{otherwise}.
            \end{cases}
    \end{equation}
    where $\mathbb{Z}_{k_i}\left\langle c_i\right\rangle$ denotes the cyclic group  $\mathbb{Z}_{k_i}$ whose generator is  $c_i$.
    Let 
    \begin{align}
      G_{\{k_1,\cdots, k_n\}} &= \text{the 
    torsion subgroup of}\ H^{AW}_0(\mathcal{S}^1_{(k_1,\cdots,k_n)}). \\
    &= \Z_{k_1}\langle c_1 \rangle \oplus \cdots \oplus \Z_{k_n}\langle c_n \rangle \slash \langle c_1+\ldots + c_n \rangle. \notag
    \end{align}
     The order of $G_{\{k_1,\cdots,k_n\}}$ is 
   \begin{equation*} 
    |G_{\{k_1,\cdots,k_n\}}| =\frac{k_1k_2\cdots k_n}{\mathrm{lcm}(k_1,k_2,\cdots,k_n)}.
\end{equation*}
     In particular, $G_{\{k_1,\cdots,k_n\}}$ is trivial if and only if $n=1$ or $k_1,k_2,\cdots, k_n$ are pairwise relatively prime.\n
 
   Another way to compute $H^{AW}_*(\mathcal{S}^1_{(k_1,\cdots,k_n)})$ is to simplify the chain complex $C_*(L,\partial^{\eta})$. From the triangulation $(L,\eta)$ and the isomorphisms in~\eqref{Equ:Q-Coeff-Iso}, we can derive that
     \[ H^{AW}_j(\mathcal{S}^1_{(k_1,\cdots,k_n)}) \cong
       \begin{cases}
   \text{cokernel of}\ \Z^n \overset{\alpha}{\longrightarrow} \Z^n ,  &  \text{if $j=0$}; \\
     \ \Z ,   &    \text{if $j=1$}; \\
     \ 0, & \text{if $j\geq 2$},
 \end{cases} 
    \]   
     where $\alpha$ is the linear map represented by the $n\times n$ matrix $ A_{(k_1,\cdots,k_n)}$ defined by
       \[ \quad \quad \
    A_{(k_1,\cdots,k_n)} = 
    \begin{pmatrix}
        -k_1 &  k_2 & \cdots & 0 & 0 \\  
         0  &  -k_2  & \cdots & 0 & 0  \\
       \vdots & \vdots & \ddots & \vdots & \vdots \\
           0 & 0 & \cdots & -k_{n-1} & k_n  \\
       k_1 & 0 & \cdots & 0 & -k_n 
   \end{pmatrix}_{n\times n}. \qquad\qquad
    \]
  The matrix $A_{(k_1,\cdots,k_n)}$ is simplified from the $2n\times 2n$ matrix
  representing the boundary map $C_1(L,\partial^{\eta})\rightarrow C_0(L,\partial^{\eta})$.
   Moreover, we can similarly compute
   \[ H^{DW}_j(\mathcal{S}^1_{(k_1,\cdots,k_n)}) \cong
       \begin{cases}
   \text{cokernel of}\ \Z^n \overset{\alpha^t}{\longrightarrow} \Z^n ,  &  \text{if $j=0$}; \\
     \ \Z ,   &    \text{if $j=1$}; \\
     \ 0, & \text{if $j\geq 2$}.
 \end{cases} 
    \]   
    where $\alpha^t$ is the linear map represented by the transpose $A^t$
    of $A$. Therefore, there is an isomorphism
 $$H^{DW}_j(\mathcal{S}^1_{(k_1,\cdots,k_n)}) \cong 
 H^{AW}_j(\mathcal{S}^1_{(k_1,\cdots,k_n)}), \, \forall j\geq 0.$$ 
 In particular, by~\eqref{Equ:AW-S-n} we have 
 \begin{equation} \label{Equ:S1-AW}
   H^{DW}_0(\mathcal{S}^1_{(k_1,\cdots,k_n)}) \cong H^{AW}_0(\mathcal{S}^1_{(k_1,\cdots,k_n)}) \cong \Z \oplus G_{\{k_1,\cdots,k_n\}}.
   \end{equation} 
  \end{exam}

 \begin{rem}\label{Rem:0-generator}
  Although $H^{AW}_0(\mathcal{S}^1_{(k_1,\cdots,k_n)}) $ is isomorphic to $H^{DW}_0(\mathcal{S}^1_{(k_1,\cdots,k_n)})$, their generators are different. In fact, $H^{AW}_0(\mathcal{S}^1_{(k_1,\cdots,k_n)}) $ is always generated by the regular points in $\mathcal{S}^1_{(k_1,\cdots,k_n)}$
 while $H^{DW}_0(\mathcal{S}^1_{(k_1,\cdots,k_n)}) $ is always generated by singular points. For example when $n=2$ (see Figure~\ref{Fig:Example-Cirle-2}), $H^{AW}_0(\mathcal{S}^1_{(k_1,k_2)})$ is generated by regular points
 $u_1$ and $u_1 - u_2$, while $H^{DW}_0(\mathcal{S}^1_{(k_1,k_2)})$ is generated by singular points 
 $v_1$ and $\frac{k_1}{\mathrm{gcd}(k_1,k_2)}v_1 - \frac{k_2}{\mathrm{gcd}(k_1,k_2)}v_2$.
 \end{rem}
  \begin{figure}[h]
        \begin{equation*}
        \vcenter{
            \hbox{
                  \mbox{$\includegraphics[width=0.46\textwidth]{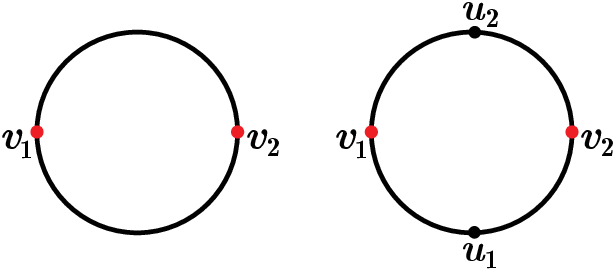}$}
                 }
           }
     \end{equation*}
   \caption{A circle with two singular points
       } \label{Fig:Example-Cirle-2}
   \end{figure}  
 
  Using the above calculation of 
  $H^{DW}_*(\mathcal{S}^1_{(k_1,\cdots,k_n)})$, we 
  can deduce that
           \begin{equation} \label{Equ:DW-I-n}    H^{DW}_i(I_{(k_1,\dots,k_n)})=
        \begin{cases}
            \Z \oplus G_{\{k_1,\cdots,k_n\}}, & \text{if $i=0$;} \\ 
            0, & \text{otherwise}.
        \end{cases}
    \end{equation}
  Indeed, the Mayer-Vietoris sequence of DW-homology groups for the decomposition $\mathcal{S}^1_{(k_1,\dots,k_n)}=I_{(k_1,\dots,k_n)} \cup I$ gives
 \begin{align*}
 &  H_0^{DW}(\{a, b\}) \xlongrightarrow{i_*} H_0^{DW}\left(I_{\left(k_1, \ldots, k_n\right)}\right) \oplus H_0^{DW}(I) \longrightarrow H_0^{DW}\left(\mathcal{S}_{\left(k_1, \ldots, k_n\right)}^1\right) \longrightarrow 0 \\
 & \quad \ \ \Z\oplus \Z \ \ \ \longmapsto \ \ \ (\Z \oplus G_{\{k_1,\cdots,k_n\}}) \oplus \Z  
   \end{align*}
 Choose the generators of $H^{DW}_0(\{a,b\} )$ to be $[a]$ and $[a]-[b]$. The image of $[a]$ in $H_0^{DW} (I_{(k_1,\cdots,k_n)} )$ and $H_0^{DW} (I )$ are both generators of the free abelian part, while the image of $[a]-[b]$ is $0$. So  $H^{DW}_0(I_{(k_1,\cdots,k_n)}) \cong \Z \oplus G_{\{k_1,\cdots,k_n\}}$ which gives~\eqref{Equ:DW-I-n}.
 \n
 
    \subsection{Two-dimensional orbifolds and
        pseudo-orbifolds}\ \n

   \begin{exam}\label{Exam:n-gon-Homology}
   In Figure~\ref{Fig:Example-n-gon}, we have a 
   $2$-dimensional pseudo-orbifold $(D^2,\lambda)$, denoted by $\mathcal{D}^2_{(k_1,\cdots,k_n)}$, which is a $2$-disk $D^2$ with $n$ singular points $v_1,\cdots, v_n$, $n\geq 1$, on the boundary of $D^2$ whose
  weights $\lambda(v_i)= k_i\geq 2$, $i=1,\cdots,n$. \n  
   \begin{figure}[h]
        \begin{equation*}
        \vcenter{
            \hbox{
                  \mbox{$\includegraphics[width=0.76\textwidth]{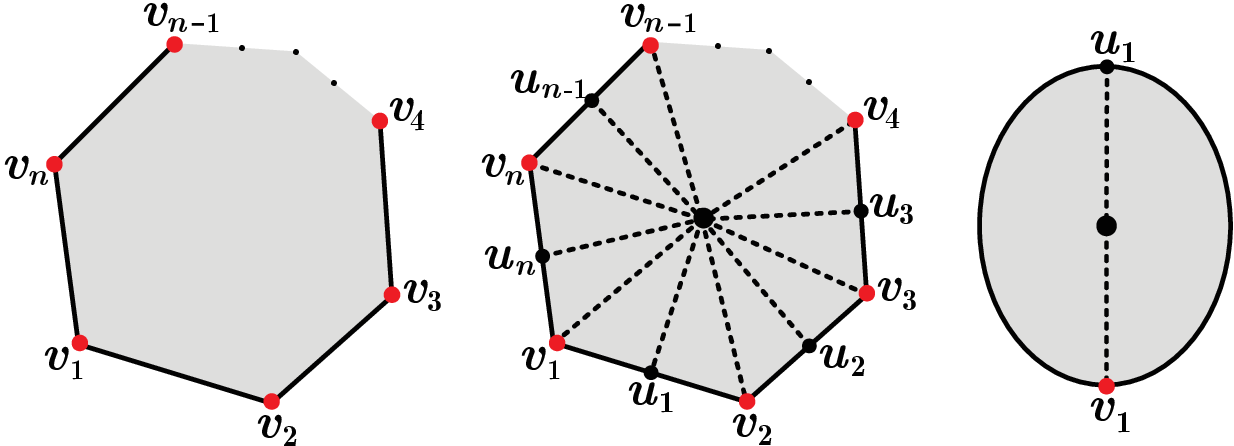}$}
                 }
           }
     \end{equation*}
   \caption{A $2$-disk with $n$ singular points on the boundary
       } \label{Fig:Example-n-gon}
   \end{figure}  
 Let $(K,\xi)$ denote the descending divisibly weighted triangulation of $\mathcal{D}^2_{(k_1,\cdots,k_n)}$ as shown in Figure~\ref{Fig:Example-n-gon}. The inversion $\widehat{\xi}$ is an ascending divisible weight on $K$. It is not hard to check that the weighted simplicial chain complex $\big(C_*(K),\partial^{\widehat{\xi}}\big)$ is isomorphic to the algebraic cone of
  $\big(C_*(K|_{\partial D^2}),\partial^{\widehat{\xi}}\big)$ (see Definition~\ref{Def:Alg-Cone}).
  So by Lemma~\ref{Lem:Cone-vanish}, $(C_*(K),\partial^{\widehat{\xi}})$ is acyclic and hence
    \[ H^{AW}_j(\mathcal{D}^2_{(k_1,\cdots,k_n)})  \cong
       \begin{cases}
   \Z ,  &  \text{if $j=0$}; \\
     0,   &    \text{if $j\geq 1$}.
 \end{cases} 
    \]  
 Meanwhile, $(C_*(K),\partial^{\xi})$ may not be acyclic. A direct calculation shows that 
 \[ H^{DW}_j(\mathcal{D}^2_{(k_1,\cdots,k_n)})  \cong
       \begin{cases}
     \Z\oplus G_{\{k_1,\cdots,k_n\}},  &  \text{if $j=0$}; \\
     0,   &    \text{if $j\geq 1$}.
 \end{cases} 
    \]  
Moreover, a simple analysis of the weighted simplicial chain complexes of  $\mathcal{S}^1_{(k_1,\cdots,k_n)}$ and 
$\mathcal{D}^2_{(k_1,\cdots,k_n)}$ shows that the inclusion map
  $\mathcal{S}^1_{(k_1,\cdots,k_n)} \hookrightarrow 
  \mathcal{D}^2_{(k_1,\cdots,k_n)}$ induces an isomorphism $ H^{DW}_0(\mathcal{S}^1_{(k_1,\cdots,k_n)}) 
  \xlongrightarrow{\cong} H^{DW}_j(\mathcal{D}^2_{(k_1,\cdots,k_n)})$.    
  \end{exam}

  \begin{exam} \label{Exam:Teardrop}
    The \emph{teardrop} orbifold $\mathcal{S}^2_{(k)}$ is the $2$-sphere with only one singular point $v_0$
    (see Figure~\ref{Fig:Example-Teardrop})
whose local group $G_{v_0}=\Z\slash k\Z$ ($k\geq 2$). A chart around $v_0$ consists of an open disk $\tilde{U}\subset \R^n$ and
an action of $\Z\slash k\Z$ on $\tilde{U}$ by rotations.
$\mathcal{S}^2_{(k)}$ is a typical example of ``bad'' orbifold.

 \begin{figure}[h]
        \begin{equation*}
        \vcenter{
            \hbox{
                  \mbox{$\includegraphics[width=0.35\textwidth]{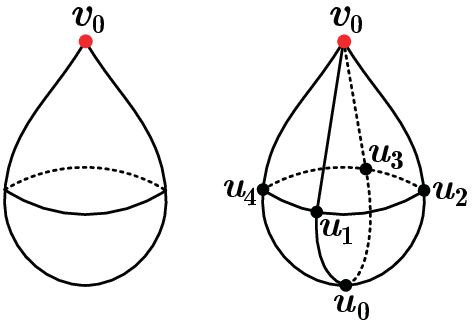}$}
                 }
           }
     \end{equation*}
   \caption{Teardrop orbifold
       } \label{Fig:Example-Teardrop}
   \end{figure}

 Using the triangulation adapted to $\mathcal{S}^2_{(k)}$ given in Figure~\ref{Fig:Example-Teardrop}, we can compute 
    \[  
   H^{AW}_j(\mathcal{S}^2_{(k)}) \cong 
    H^{DW}_j(\mathcal{S}^2_{(k)}) \cong
       \begin{cases}
   \Z,  &  \text{if $j=0$}; \\
     0 ,   &    \text{if $j=1$}; \\
   \Z,  &  \text{if $j = 2$};\\
   0, & \text{if $j \geq 3$}.
 \end{cases} 
    \]    
   \end{exam}

  \begin{exam} \label{Exam:Football}
    In Figure~\ref{Fig:Example-Football}, we have an orbifold $\mathcal{S}^2_{(k_1,k_2)}$ which is a $2$-sphere with two isolated singular points $v_1$, $v_2$ with local groups
    $G_{v_1}=\Z\slash k_1\Z$, $G_{v_2}=\Z\slash k_2\Z$.
    Usually, $\mathcal{S}^2_{(k_1,k_2)}$ is called
    a \emph{football} if $k_1=k_2$ and a \emph{spindle} 
    if $k_1\neq k_2$. \n    
   
     Using the triangulation adapted to
    $\mathcal{S}^2_{(k_1,k_2)}$ in the middle picture of Figure~\ref{Fig:Example-Football},
    we can compute its AW-homology groups:
    \[ H^{AW}_j(\mathcal{S}^2_{(k_1,k_2)}) \cong
       \begin{cases}
   \Z,  &  \text{if $j=0$}; \\
     \Z\slash \mathrm{gcd}(k_1,k_2)\Z ,   &    \text{if $j=1$}; \\
   \Z,  &  \text{if $j= 2$};\\
   0, & \text{if $j\geq 3$}.
 \end{cases} 
    \]
    
     If $\mathrm{gcd}(k_1,k_2)\geq 2$, the generator of 
    the factor $\Z\slash \mathrm{gcd}(k_1,k_2)\Z$ in
    $H^{AW}_1(\mathcal{S}^2_{(k_1,k_2)})$ can be represented by
    $[\overline{u_1u_2} + \overline{u_2u_3} + \overline{u_3u_4} +\overline{u_4u_1}]$.\n
    
    \begin{figure}[h]
        \begin{equation*}
        \vcenter{
            \hbox{
                  \mbox{$\includegraphics[width=0.62\textwidth]{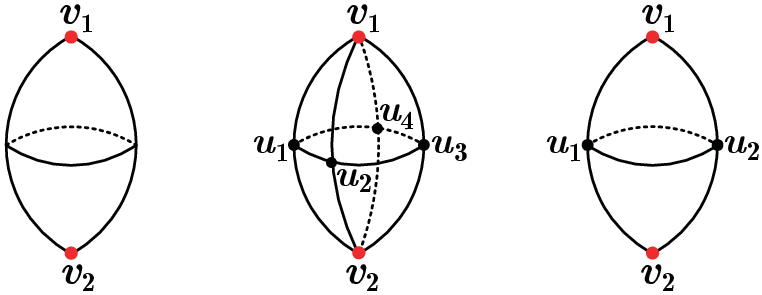}$}
                 }
           }
     \end{equation*}
   \caption{A $2$-sphere with two singular points
       } \label{Fig:Example-Football}
   \end{figure}  
        
        Notice that 
    the right picture of Figure~\ref{Fig:Example-Football} is a $\Delta$-complex triangulation of $\mathcal{S}^2_{(k_1,k_2)}$ which can also be used for our calculation. Moreover, we can compute
      \[ H^{DW}_j(\mathcal{S}^2_{(k_1,k_2)}) \cong
       \begin{cases}
  \Z \oplus \Z\slash \mathrm{gcd}(k_1,k_2)\Z,  &  \text{if $j=0$}; \\
   0   ,   &    \text{if $j=1$}; \\
   \Z,  &  \text{if $j= 2$};\\
   0, & \text{if $j\geq 3$}.
 \end{cases} 
    \]
  The generators of $H^{DW}_0(\mathcal{S}^2_{(k_1,k_2)})$ can be taken to be $v_1$ and $\frac{k_1}{\mathrm{gcd}(k_1,k_2)}v_1 - \frac{k_2}{\mathrm{gcd}(k_1,k_2)}v_2$.
     
      \end{exam}

  \begin{exam}\label{Exam:Sphere-n-points}
  Let $\mathcal{S}^2_{(k_1,\cdots,k_n)}$, $n\geq 1$, denote the pseudo-orbifold $(S^2,\lambda)$ with $n$ isolated singular points $v_1,\cdots, v_n$ whose
  weights are $\lambda(v_i)=k_i \geq 2$, $i=1,\cdots,n$. 
 We can think of $\mathcal{S}^2_{(k_1,\cdots,k_n)}$ as the gluing of two copies of $\mathcal{D}^2_{(k_1,\cdots,k_n)}$ along their boundaries, that is
 \[  \mathcal{S}^2_{(k_1,\cdots,k_n)} = \mathcal{D}^2_{(k_1,\cdots,k_n)} \cup_{\mathcal{S}^1_{(k_1,\cdots,k_n)}} \mathcal{D}^2_{(k_1,\cdots,k_n)}. \]
 Then by Mayer-Vietoris sequence
 and the result from~\eqref{Equ:S1-AW},
 we can easily compute
 \[ H^{AW}_j(\mathcal{S}^2_{(k_1,\cdots,k_n)})  \cong
       \begin{cases}
   \Z,  &  \text{if $j=0$}; \\
  G_{\{k_1,\cdots,k_n\}},   &    \text{if $j=1$}; \\
   \Z,  &  \text{if $j= 2$};\\
   0, & \text{if $j\geq 3$}. 
 \end{cases}  
    \]
    \[ H^{DW}_j(\mathcal{S}^2_{(k_1,\cdots,k_n)})  \cong
       \begin{cases}
   \Z\oplus G_{\{k_1,\cdots,k_n\}},  &  \text{if $j=0$}; \\
  0,   &    \text{if $j=1$}; \\
   \Z,  &  \text{if $j= 2$};\\
   0, & \text{if $j\geq 3$}.
 \end{cases} 
    \]
 Moreover, from the Mayer-Vietoris sequence we can see that\n
 
 $\bullet$ The connecting homomorphism 
 $$H^{AW}_1(\mathcal{S}^2_{(k_1,\cdots,k_n)}) \xlongrightarrow{d_1} H^{AW}_0(\mathcal{S}^1_{(k_1,\cdots,k_n)})\cong \Z\oplus G_{\{k_1,\cdots,k_n\}}$$ maps $H^{AW}_1(\mathcal{S}^2_{(k_1,\cdots,k_n)})$ isomorphically onto the torsion subgroup of
 $H^{AW}_0(\mathcal{S}^1_{(k_1,\cdots,k_n)})$.
\n
$\bullet$ The inclusion map $\mathcal{S}^1_{(k_1,\cdots,k_n)} \hookrightarrow \mathcal{S}^2_{(k_1,\cdots,k_n)}$ induces 
 an isomorphism 
 $$H^{DW}_0(\mathcal{S}^1_{(k_1,\cdots,k_n)}) \xlongrightarrow{\cong} H^{DW}_0(\mathcal{S}^2_{(k_1,\cdots,k_n)}).$$
  
\end{exam}
 
 \begin{rem}
   From Example~\ref{Exam:1-dim-orbifold},  
  \ref{Exam:Teardrop}, \ref{Exam:Football}
  and~\ref{Exam:Sphere-n-points}, we can see that a slight change of the weight of a singular point on a pseudo-orbifold may cause drastic changes to its AW-homology and DW-homology.\n
  \end{rem}

  More generally, we can compute the AW-homology and DW-homology of any pseudo-orbifold  with finitely many singular points whose underlying space is a closed surface.
      
    \begin{exam} \label{Exam:surface-sing}
   Let $\mathcal{X}_{(k_1,\cdots,k_n)}$, $n\geq 1$, denote the 
   pseudo-orbifold $(\Sigma,\lambda)$ where $\Sigma$ is a closed connected surface with $n$ singular points $v_1,\cdots, v_n$ whose weights are
 $\lambda(v_i) = k_i\geq 2$, $i=1,\cdots, n$. By thinking of $\mathcal{X}_{(k_1,\cdots,k_n)}$ as the connected sum of 
  $\mathcal{S}^2_{(k_1,\cdots,k_n)}$ with the surface $\Sigma$ (with no singular points) along a regular circle, we can easily compute the AW-homology and DW-homology of $\mathcal{X}_{(k_1,\cdots,k_n)}$ from the previous results on $\mathcal{S}^2_{(k_1,\cdots,k_n)}$ using 
  Mayer-Vietoris sequence:
  \[  H^{AW}_j(\mathcal{X}_{(k_1,\cdots,k_n)})\cong \begin{cases}
   \Z,  &  \text{if $j=0$}; \\
   H_1(\Sigma) \oplus G_{\{k_1,\cdots,k_n\}},  & \text{if $j=1$}; \\
     \Z ,  &  \text{if $j=2$}; \\
    0 ,  &  \text{if $j\geq 3$}. \\
 \end{cases} \]
 \[  H^{DW}_j(\mathcal{X}_{(k_1,\cdots,k_n)})\cong \begin{cases}
   \Z \oplus G_{\{k_1,\cdots,k_n\}} ,  & \text{if $j=0$}; \\
   H_1(\Sigma),  &  \text{if $j=1$}; \\
     \Z ,  &  \text{if $j=2$}; \\
    0 ,  &  \text{if $j\geq 3$}. \\
 \end{cases} \]
    where $H_1(\Sigma)$ is the ordinary singular homology group of $\Sigma$ in dimension $1$.   
    \end{exam}
    
     \begin{rem} \label{Rem:Agree}
     The AW-homology of $\mathcal{X}_{(k_1,\cdots,k_n)}$ is isomorphic to the $t$-singular homology  of $\mathcal{X}_{(k_1,\cdots,k_n)}$ (see~\cite[Theorem 11.1]{TakYok06}). Moreover, $H^{AW}_1(\mathcal{X}_{(k_1,\cdots,k_n)})$ is exactly the abelianization of the 
     orbifold fundamental group $\pi^{orb}_1(\mathcal{X}_{(k_1,\cdots,k_n)})$ of $\mathcal{X}_{(k_1,\cdots,k_n)}$ (see~\cite[p.\,424]{Scott83} for a presentation of $\pi^{orb}_1(\mathcal{X}_{(k_1,\cdots,k_n)}$)).
    \end{rem}
       
 \begin{exam}
 Let $\widetilde{\mathcal{D}}^2_{(l_1,\cdots,l_m)}$, $m\geq 1$,
  denote the pseudo-orbifold $(D^2,\widetilde{\lambda})$ with $m$ isolated singular points $z_1,\cdots, z_m$
  in the interior of $D^2$ whose weights are 
  $$\widetilde{\lambda}(z_j)=l_j\geq 2,\ j=1,\cdots,m.$$ 
  Since the boundary of $\widetilde{\mathcal{D}}^2_{(l_1,\cdots,l_m)}$ is a regular circle,
 we can glue a regular $2$-disk $D^2$ to $\widetilde{\mathcal{D}}^2_{(l_1,\cdots,l_m)}$
along its boundary, which gives $\mathcal{S}^2_{(l_1,\cdots,l_m)}$.
 Then by Mayer-Vietoris sequence, we can easily obtain
 \[ H^{AW}_j(\widetilde{\mathcal{D}}^2_{(l_1,\cdots,l_m)})   \cong
       \begin{cases}
   \Z,  &  \text{if $j=0$}; \\
  G_{\{l_1,\cdots,l_m\}},   &    \text{if $j=1$}; \\
  0,  &  \text{if $j \geq 2$}.
 \end{cases} 
 \]
 \[ H^{DW}_j(\widetilde{\mathcal{D}}^2_{(l_1,\cdots,l_m)})  \cong
       \begin{cases}
   \Z\oplus G_{\{l_1,\cdots,l_m\}},  &  \text{if $j=0$}; \\
   0,   &    \text{if $j\geq 1$}.
 \end{cases} 
    \]
  \end{exam}  
    
 By comparing the above computations for $\mathcal{D}^2_{(k_1,\cdots,k_n)}$ and $\widetilde{\mathcal{D}}^2_{(l_1,\cdots,l_m)}$, we see that moving the singular points
 from the boundary of a $2$-disk to its interior may change the AW-homology while preserving the DW-homology.\n
 
  It is also possible for us to compute the AW-homology and DW-homology of a pseudo-orbifold 
  based on $D^2$ which has singular points both in the interior and on the boundary of $D^2$.
  But it is difficult to write an explicit formula
  for the general case.

  \section{Product structure on DW-cohomology} \label{Sec:Product-Cohomology}
  
  \subsection{AW-cohomology and DW-cohomology}\ \n
  
  Let $(K,\mu)$ be a weighted simplicial complex. For an abelian group $G$, the \emph{weighted simplicial cochain complex} with $G$-coefficients is obtained by applying
  the $\mathrm{Hom}(-, G)$ functor to 
  the weighted simplicial chain complex $(C_*(K),\partial^{\mu})$, denoted by $(C^*(K;G),\delta^{\mu})$,
  where $\delta^{\mu}$ is the coboundary map determined by
    $\partial^{\mu}$.
  Then the \emph{weighted simplicial cohomology group of $(K,\mu)$} with $G$-coefficients is
   $$ H^*(K,\delta^{\mu};G):= H^*((C^*(K;G),\delta^{\mu})).$$    
   In practice, we mainly use $G=\Z$ as the coefficients 
   because of Theorem~\ref{thm:Weighted-Homology-Iso}.\n
   
   If $\varrho: (K,\mu) \rightarrow (K',\mu')$ is a
  a morphism of weighted simplicial complexes, define
  $$\varrho^{\#}: C^n(K';G)\rightarrow C^n(K;G), \ n\geq0, $$ 
 by: for any $\phi\in C^n(K';G)$ and any oriented
  $n$-simplex $[v_0,\cdots,v_n]\in C_n(K)$,
  $$\varrho^{\#}(\phi)([v_0,\cdots, v_n])= 
  \phi(\varrho_{\#}[v_0,\cdots, v_n]).$$
 It is easy to see that $\varrho^{\#}$ is a cochain map
 from $(C^*(K';G),\delta^{\mu'})$ to $(C^*(K;G),\delta^{\mu})$ and then induces a map on the weighted simplicial cohomology groups
 $$ \varrho^*: H^*(K',\delta^{\mu'};G)\rightarrow  
 H^*(K,\delta^{\mu};G). $$
  
  \begin{defi}[AW-cohomology and DW-cohomology]
   Suppose $(X,\lambda)$ is a weighted polyhedron with a divisibly weighted triangulation $(K,\xi)$. Then we define the \emph{AW-cohomology} and \emph{DW-cohomology} of $(X,\lambda)$ with $G$-coefficients to be $H^*(K,\delta^{\xi};G)$
   and $H^*(K,\delta^{\widehat{\xi}};G)$, respectively if
   $\xi$ is ascending and contrariwise if $\xi$ is descending. Denote the AW-cohomology of $(X,\lambda)$
   by $H^*_{AW}(X,\lambda;G)$ and the DW-cohomology 
    by $H^*_{DW}(X,\lambda;G)$. 
  \end{defi}
    
     By the universal coefficient theorem for cohomology  (see~\cite[Theorem 3.2]{Hatcher02}), we can compute the AW-cohomology and DW-cohomology groups of $(X,\lambda)$ from its AW-homology and DW-homology groups, respectively.
    \n
    
  If $f:(X,\lambda)\rightarrow (X',\lambda')$ is a $W$-continuous map of weighted polyhedra, then by the above definitions
$f$ induces group homomorphisms
$$f^*: H^*_{AW}(X',\lambda';G)\rightarrow H^*_{AW}(X,\lambda;G), \ \ f^*: H^*_{DW}(X',\lambda';G)\rightarrow H^*_{DW}(X,\lambda;G). $$
   
   \n

\subsection{Weighted cup product} \ \n

  We discover that there is a natural product structure on
  the DW-cohomology of a weighted polyhedron defined
 below, which generalizes the cup product $\cup$ of 
 ordinary simplicial cochains induced by Alexander-Whitney diagonal. Moreover, we will prove in Section~\ref{Subsec:Graded-Comm} that this product
  is graded commutative on the DW-cohomology.

 \begin{defi}[Weighted Cup Product] \label{Def:Product-DW}
 Let $(K,\xi)$ be a descending divisibly weighted simplicial complex and $R$ be a commutative ring with unit.
 Choose a total ordering $\prec$ of the vertices of $K$. Then for any cochains $\phi\in C^p(K;R)$ and
  $\psi\in C^q(K;R)$, let $\phi \Cup \psi\in
  C^{p+q}(K;R)$
  be the cochain whose value on an oriented
  $(p+q)$-simplex $\sigma =[v_0,\cdots,v_{p+q}]$ of $K$ with $v_0\prec\cdots\prec v_{p+q}$ is defined by
  \small
  \begin{equation}  \label{Equ:Prod-Wt-Homol}
  \phi\Cup \psi (\sigma) := \frac{\xi([v_0,\cdots,v_p]) \xi([v_p,\cdots,v_{p+q}])}{\xi([v_0,\cdots,v_{p+q}])} \, \phi([v_0,\cdots,v_p])
 \cdot \psi([v_p,\cdots,v_{p+q}]).
 \end{equation}
 \normalsize
 The coefficient on the right hand side of~\eqref{Equ:Prod-Wt-Homol} is always integral because $\xi$ is descending. Besides, the product of an integer with an element of $R$ is defined by the abelian group structure of $R$.
  We call $\Cup$ the \emph{weighted cup product} on $(C^*(K;R),\delta^{\xi})$ with respect to $\prec$. 
\end{defi}

 Note that the above definition depends on the ordering
 $\prec$ of the vertices of $K$ because $\mu$, $\phi$ and $\psi$ may take different values on different faces of $[v_0,\cdots,v_{p+1}]$. So we should have use
a notation like $\Cup^{\prec}$ instead of $\Cup$ to indicate the dependence on $\prec$. But later we will prove in Theorem~\ref{Thm:Iso-Cup-Cohom} that the ring structure induced by $\Cup$ on the weighted cohomology group $H^*(K,\delta^{\xi};R)$ is
actually independent on the choice of $\prec$ up to ring isomorphisms. So we will not attach $\prec$ in our notation and just remember that we need to
 fix a total ordering of the vertices of $K$ when
 writing the formula.
 A convenient choice of $\prec$ for $(K,\xi)$ is to order the vertices of $K$ according to their weight values. \n

 \begin{lem} \label{Lem:Cobound}
 Let $(K,\xi)$ be a descending divisibly weighted simplicial complex. For any $\phi\in C^p(K ; R) $ and $\psi\in C^q(K ; R)$, we have
   \begin{equation} \label{Equ:Cobound-formula}
   	 \delta^{\xi}(\phi\Cup \psi) =\delta^{\xi} \phi \Cup\psi + (-1)^p\phi \Cup \delta^{\xi}\psi.
   \end{equation}  
  \end{lem}
  \begin{proof}
   For any $(p+q+1)$-simplex $\sigma=[v_0,\cdots,v_{p+q+1}]$ of $K$ with $v_0\prec\cdots\prec v_{p+q}$,
\small
\begin{align*}
  &\, \delta^{\xi}(\phi\Cup \psi)(\sigma) = \phi\Cup \psi (\partial^{\xi} \sigma) \\
  = &\, \phi\Cup \psi \Big( \sum^{p+q+1}_{i=0} (-1)^i \frac{\xi([v_0,\cdots,\widehat{v}_i,\cdots,v_{p+q+1}])}{\xi([v_0,\cdots,v_{p+q+1}])}\, [v_0,\cdots,\widehat{v}_i,\cdots,v_{p+q+1}] \Big) \\
  = &   \sum^{p}_{i=0} (-1)^i         
       \frac{\xi([v_0,\cdots,\widehat{v}_i,\cdots,v_{p+1}])\xi([v_{p+1},\cdots,v_{p+q+1}])}{\xi([v_{0},\cdots,v_{p+q+1}])}\,  \phi([v_0,\cdots,\widehat{v}_i,\cdots,v_{p+1}]) 
  \cdot \psi([v_{p+1},\cdots,v_{p+q+1}])\\
  +& \sum^{p+q+1}_{i=p+1} (-1)^i \frac{\xi([v_0,\cdots,v_p])        
        \xi([v_p,\cdots,\widehat{v}_i,\cdots, v_{p+q+1}])}{\xi([v_{0},\cdots,v_{p+q+1}])}\,  \phi([v_0,\cdots,v_p]) 
  \cdot \psi([v_p,\cdots, \widehat{v}_i,\cdots,v_{p+q+1}]).
  \end{align*}
  \normalsize
  On the other side, we have
  \small  \begin{align*}
       &\quad  \ (\delta^{\xi}\phi\Cup \psi)(\sigma) \\
       &= \frac{\xi([v_0,\cdots,v_{p+1}])        
        \xi([v_{p+1},\cdots,v_{p+q+1}])}{\xi([v_{0},\cdots,v_{p+q+1}])}\, \delta^{\xi}\phi([v_0,\cdots,v_{p+1}]) 
  \cdot \psi([v_{p+1},\cdots,v_{p+q+1}])
     \end{align*}  
 \normalsize
   where 
 \small \begin{align*}
    \delta^{\xi}\phi([v_0,\cdots,v_{p+1}]) 
      & =\sum^{p+1}_{i=0} (-1)^i \frac{\xi([v_0,\cdots,\widehat{v}_i,\cdots,v_{p+1}])}{\xi([v_0,\cdots,v_{p+1}])}
      \, \phi([v_0,\cdots,\widehat{v}_i,\cdots,v_{p+1}]).
  \end{align*}
 \normalsize
  Then we have
 \small  \begin{align*}
       &\quad  \ (\delta^{\xi} \phi\Cup \psi)(\sigma) \\
       &= \sum^{p+1}_{i=0} (-1)^i         
       \frac{\xi([v_0,\cdots,\widehat{v}_i,\cdots,v_{p+1}])\xi([v_{p+1},\cdots,v_{p+q+1}])}{\xi([v_{0},\cdots,v_{p+q+1}])}\,  \phi([v_0,\cdots,\widehat{v}_i,\cdots,v_{p+1}]) 
  \cdot \psi([v_{p+1},\cdots,v_{p+q+1}]).
     \end{align*} 
      \normalsize
     Similarly, we have 
  \small
   \begin{align*}
       &\quad  \ (-1)^p (\phi\Cup \delta^{\xi}\psi)(\sigma) \\
       &= \sum^{p+q+1}_{i=p} (-1)^i \frac{\xi([v_0,\cdots,v_p])        
        \xi([v_p,\cdots,\widehat{v}_i,\cdots, v_{p+q+1}])}{\xi([v_{0},\cdots,v_{p+q+1}])}\,  \phi([v_0,\cdots,v_p]) 
  \cdot \psi([v_p,\cdots, \widehat{v}_i,\cdots,v_{p+q+1}]).
     \end{align*} 
      \normalsize
  
  If we add the above two expressions up, the last term of the first sum cancels the first term of the second sum, and the remaining terms give exactly $\delta^{\xi}(\phi\Cup \psi)(\sigma) $.
  The lemma is proved.
  \end{proof}
  
  By Lemma~\ref{Lem:Cobound}, the above product $\Cup$ on cochains induces the weighted cup product on 
  $H^*(K,\delta^{\xi};R)$:
   $$ H^p(K,\delta^{\xi};R) \times
   H^q(K,\delta^{\xi};R) \overset{\Cup}{\longrightarrow} H^{p+q}(K,\delta^{\xi};R), \ p,q\in \Z.  $$
 
 Then by Definition~\ref{Defi:Pseudo-Orbi-Homol}, this determines the weighted cup product $\Cup$ on the DW-cohomology $H^*_{DW}(X,\lambda;R)$ of a weighted polyhedron $(X,\lambda)$, that is
 $$ H^p_{DW}(X,\lambda;R) \times
  H^q_{DW}(X,\lambda;R) \overset{\Cup}{\longrightarrow} H^{p+q}_{DW}(X,\lambda;R), \ p,q\in \Z.  $$
\n

 Moreover, it is easy to check that the weighted cup product $\Cup$ is \emph{natural}
 in the sense that for 
 any morphism $\varrho: (K,\xi)\rightarrow (K',\xi')$ between two descending divisibly weighted simplicial complexes $(K,\xi)$ and $(K',\xi')$, the following diagram commutes
 \[ \xymatrix{
          H^p(K',\delta^{\xi'};R) \times  H^q(K',\delta^{\xi'};R) \ar[d]_{\varrho^*\times\varrho^*} \ar[r]^{\qquad\quad \Cup}
                &  H^{p+q}(K',\delta^{\xi'};R) \ar[d]^{\varrho^*}  \\
           H^p(K,\delta^{\xi};R)\times H^q(K,\delta^{\xi};R) \ar[r]^{\qquad\quad \Cup} & H^{p+q}(K,\delta^{\xi};R)
                 }.  \]
                 
Correspondingly, if $f:(X,\lambda)\rightarrow (X',\lambda')$ is a $W$-continuous map of weighted polyhedra,
the following diagram commutes:
     \[ \xymatrix{
          H^p_{DW}(X',\lambda';R) \times  H^q_{DW}(X',\lambda';R) \ar[d]_{f^*\times f^*} \ar[r]^{\qquad\qquad \Cup} &  H^{p+q}_{DW}(X',\lambda';R) \ar[d]^{f^*}  \\
     H^p_{DW}(X,\lambda;R)\times   H^q_{DW}(X,\lambda;R) \ar[r]^{\qquad\quad\ \ \Cup} &   H^{p+q}_{DW}(X,\lambda;R)
                 }.  \]

 \begin{rem}
 	 It seems to us that there is no meaningful product structure on AW-cohomology  with integral coefficients.
 	 Indeed, the naive extension of the formula in~\eqref{Equ:Prod-Wt-Homol} to an ascending weighted simplicial complex cannot guarantee the coefficients
 	 in the formula to be integral and make the coboundary formula in~\eqref{Equ:Cobound-formula} hold simultaneously.
 \end{rem}  
 
\subsection{Ordered simplex and ordered simplicial chain}
 \ \n
 To prove that the product $\Cup$ in Definition~\ref{Def:Product-DW} always defines isomorphic ring structures on $H^*(K,\delta^{\xi};R)$ with respect to  different choices of the total orderings of the vertices of $K$, 
 we need to use the following notions from~\cite[\S 13]{Munk84} and define some parallel notions for weighted simplicial complexes.
 
 \begin{defi}[Ordered Simplex and Ordered Simplicial Chain Complex]\label{Def:Ord-Simp}
   Let $K$ be a simplicial complex. An \emph{ordered
   $n$-simplex} of $K$ is an $n+1$ tuple $\langle v_0,\cdots, v_n\rangle$ of vertices of $K$, where $v_i$ are vertices of a simplex of $K$ but need not be distinct.
   Let $\widehat{C}_n(K)$ denote the free abelian group generated by the ordered $n$-simplices
of $K$; it is called the \emph{$n$-dimensional ordered simplicial chain group of $K$}. Define
the boundary map
$ \hat{\partial}: \widehat{C}_n(K) \rightarrow \widehat{C}_{n-1}(K) $ by
\begin{equation} \label{Equ:Boundary-Ord-Simplex}
  \hat{\partial} \big( \langle v_0,\cdots, v_n \rangle \big) = \sum^n_{i=0} (-1)^i \langle v_0,\cdots, \widehat{v}_i,\cdots, v_n\rangle.  
  \end{equation}
It is easy to check that $\hat{\partial}$ is well-defined and $\hat{\partial}\circ \hat{\partial}=0$ for all $n\geq 1$.
We call $(\widehat{C}_*(K),\hat{\partial})$ the 
\emph{ordered simplicial chain complex} of $K$ and
let 
$$\widehat{H}_*(K)= \text{the homology group of}\   
 \big(\widehat{C}_*(K),\hat{\partial}\big).$$
 We call $\widehat{H}_*(K)$ the \emph{ordered simplicial
 homology} of $K$. In addition, let
 \begin{equation*}
  \sigma_{v_0\cdots v_n} = \text{the simplex of} \ K
  \ \text{spanned by all the different vertices among}\
    v_0,\cdots, v_n. 
 \end{equation*}
 \end{defi}

 Note that the ordinary simplicial chain group $C_*(K)$ can be considered
 as the quotient group of the ordered simplicial chain group $\widehat{C}_*(K)$. Indeed, for all $n\geq 0$ there is a natural epimorphism
   \[ \Psi_n: \widehat{C}_n(K)\rightarrow C_n(K) \]
 defined by
 \begin{equation} \label{Equ:Psi-def}
   \Psi_n \big( \langle v_{0},\cdots, v_{n}\rangle \big)= \begin{cases}
  [v_{0},\cdots, v_{n}],  &  \text{if $v_{0},\cdots, v_{n}$ are all distinct}; \\
   0,  &  \text{otherwise}.
 \end{cases} 
 \end{equation}
       
 On the other hand, given a total ordering $\prec$ of all the
 vertices of $K$, we can define a monomorphism
 $\Theta_n: C_n(K)\rightarrow \widehat{C}_n(K)
 $ for all $n\geq 0$ by: for any simplex
 $\sigma=\{v_{0},\cdots, v_{n}\}$ of $K$ with $v_{0} \prec \cdots \prec v_{n}$,
 \begin{equation}\label{Equ:Theta-def}
  \Theta_n \big([v_{0},\cdots, v_{n}]\big) =  \langle v_{0},\cdots, v_{n}\rangle.
 \end{equation}
  It is easy to 
 check that $\Psi=\{\Psi_n\}_{n\geq 0}$ and $\Theta=\{\Theta_n\}_{n\geq 0}$ are both chain maps.
 Moreover, according to~\cite[Theorem 13.6]{Munk84}, $\Psi$ and $\Theta$ are chain homotopy inverses of each other. So there is a group isomorphism 
 \begin{equation} \label{Equ:Iso-Hom-order}
  H_j(K)\cong \widehat{H}_j(K), \ \forall j\geq 0. 
  \end{equation}
  
  Given any abelian group $G$, it is natural to define
  \emph{ordered 
  simplicial chain complex} and \emph{ordered 
  simplicial cochain complex} of $K$ by applying the
  tensor product functor $\otimes G$ and $\mathrm{Hom}(-, G)$ functor to $\widehat{C}_*(K)$, respectively, denoted by
  $$ \big(\widehat{C}_*(K;G),\hat{\partial} \big) = \big(\widehat{C}_*(K)\otimes G, \hat{\partial} \otimes \mathrm{id}_G \big), 
   \ \
     \big(\widehat{C}^*(K;G),\hat{\delta}\big) =
  \mathrm{Hom} \big((\widehat{C}_*(K),\hat{\partial}),
   G\big). $$
 Moreover, let
  $$ \widehat{H}_*(K;G) := \text{homology group of}\ \big(\widehat{C}_*(K;G),\hat{\partial} \big),  $$
  $$ \widehat{H}^*(K;G) := \text{cohomology group of}\ \big(\widehat{C}^*(K;G),\hat{\delta} \big).  $$ 
  We call $\widehat{H}_*(K;G)$ and $\widehat{H}^*(K;G)$  the \emph{ordered simplicial homology} and \emph{ordered simplicial cohomology of $K$ with $G$-coefficients}, respectively.  By the isomorphism
    in~\eqref{Equ:Iso-Hom-order} and the universal coefficient theorems, we obtain isomorphisms
   \begin{equation} \label{Equ:Iso-CoHom-order}
  H_j(K;G)\cong \widehat{H}_j(K;G), \ \   H^j(K;G)\cong \widehat{H}^j(K;G), \ \forall j\geq 0. 
   \end{equation}
   \n
  Next, we extend the above definitions to weighted simplicial complexes. Let $(K,\mu)$ be a weighted simplicial complex. Define the weight 
  of any ordered simplex $\langle v_0,\cdots,v_n \rangle$  of $K$ by
  $$\mu\big( \langle v_0,\cdots,v_n \rangle \big)= \mu(\sigma_{v_0 \cdots v_n}).$$
   
 In addition, we can modify the boundary map $\hat{\partial}$ on $\widehat{C}_*(K)$ into a boundary map
  $ \hat{\partial}^{\mu}: \widehat{C}_n(K)\rightarrow \widehat{C}_{n-1}(K) $ using the similar formula as $\partial^{\mu}$
  given by~\eqref{Equ:weighted-Boun-AW} and~\eqref{Equ:weighted-Boun-DW} (depending on $\mu$
  is ascending or descending). So we obtain a chain complex $\big(\widehat{C}_*(K), \hat{\partial}^{\mu} \big)$.
 Note that the augmentation $\varepsilon^{\mu}$ of
  $(C_*(K), \partial^{\mu})$ defined in Section~\ref{Subsec:Augmentation} gives the augmentation of $\big(\widehat{C}_*(K), \hat{\partial}^{\mu}\big)$.  
   Moreover, $\hat{\partial}^{\mu}$ induces a coboundary map $\hat{\delta}^{\mu}$ on the cochain complex $\widehat{C}^*(K;G)$.  Let
\[ \widehat{H}_*(K, \hat{\partial}^{\mu};G) := \ \text{the homology
group of}\ \big(\widehat{C}_*(K;G),\hat{\partial}^{\mu} \big); \] 
\[ 
\widehat{H}^*(K,\hat{\delta}^{\mu};G) :=\ \text{the cohomology group of}\ \big(\widehat{C}^*(K;G), \hat{\delta}^{\mu} \big). \]
 As usual, we will omit the coefficient $G$ if $G=\Z$.\n

The functorial properties of (weighted) simplicial homology naturally extend to the ordered (weighted) simplicial homology. In particular, we have the following theorem
which is parallel to Theorem~\ref{Thm:Contig}.

\begin{thm}  \label{Thm:Contig-Hom}
   If two morphisms $\varrho_0,\varrho_1: (K,\mu)\rightarrow (K',\mu')$ of weighted simplicial complexes 
   are contiguous, then the induced chain maps
   $$(\varrho_0)_{\#}, (\varrho_1)_{\#}: 
   \big(\widehat{C}_*(K),\hat{\partial}^{\mu} \big) \rightarrow \big(\widehat{C}_*(K'),\hat{\partial}^{\mu'} \big)$$ are chain homotopic, and hence
    $ (\varrho_0)_*= (\varrho_1)_*:  \widehat{H}_*(K,\hat{\partial}^{\mu})\rightarrow \widehat{H}_*(K',\hat{\partial}^{\mu'})$.
\end{thm}
\begin{proof}
  The proof is almost identical to
  to proof of Theorem~\ref{Thm:Contig} in~\cite{Daw90} 
if we replace ordinary simplices by ordered simplices there.
\end{proof}

 \begin{lem} \label{Lem:Wt-Div-Simplex-Ord}
  Let $\sigma$ be a simplex and $\xi$ be a divisible weight on $\sigma$. Then 
   $$
   \widehat{H}_j(\sigma,\hat{\partial}^{\xi}) \cong \begin{cases}
   \Z ,  &  \text{if $j=0$}; \\
   0 ,  &  \text{if $j \geq 1$}.
 \end{cases}
 $$
 \end{lem}
 \begin{proof}
 The proof is parallel to that of Lemma~\ref{Lem:Wt-Div-Simplex}.
   The identity map $\mathrm{id}_{\sigma}:\sigma\rightarrow \sigma$ is contiguous to a constant map, and so the lemma follows from Theorem~\ref{Thm:Contig-Hom}. 
 \end{proof}

 Moreover, we have the following theorem which 
  generalizes~\cite[Theorem 13.6]{Munk84}.
  
  \begin{thm}
  Let $(K,\xi)$ be a divisibly weighted simplicial complex.
 Then the two maps  
  $\Psi:  \big(\widehat{C}_*(K), \hat{\partial}^{\xi} \big)
  \rightarrow \big(C_*(K),\partial^{\xi} \big)$ and
  $\Theta: \big(C_*(K),\partial^{\xi} \big) \rightarrow
  \big(\widehat{C}_*(K), \hat{\partial}^{\xi} \big)$ are augmentation-preserving chain maps that are chain homotopy
inverses to each other. 
  \end{thm}
  \begin{proof}
   It is routine to check from the definitions that $\Psi$ and $\Theta$ are both chain maps with respect to the weighted boundary $\partial^{\xi}$ and $\hat{\partial}^{\xi}$. In addition,  they are both augmentation-preserving since they are the identity map on $0$-chains. Moreover, it is obvious that
    $$\Psi\circ \Theta  = \mathrm{id}_{C_*(K)}: (C_*(K),\partial^{\xi})\rightarrow (C_*(K),\partial^{\xi}).$$
  So it remains to prove that $\Theta \circ \Psi$ is chain homotopic to $\mathrm{id}_{\widehat{C}_*(K)}$.\n
  
  Let $\prec$ be a total ordering of all the vertices of $K$. Then for any ordered simplex $\langle v_{0},\cdots,v_{n} \rangle$ of $K$, 
  \begin{align} \label{Equ:Comp-Psi-Theta}
  & \ \Theta_n\circ \Psi_n  \big(\langle v_{0},\cdots,v_{n} \rangle \big) \\
   =& \begin{cases}
  (-1)^{t(v_{0} \cdots v_{n})}  \langle v_{i_0},\cdots, v_{i_n}\rangle,  &  \text{if $v_{0},\cdots, v_{n}$ are all distinct}; \\
   0,  &  \text{otherwise},
 \end{cases} \notag
  \end{align}
  where the sequence $v_{i_0} \cdots v_{i_n}$ is a permutation of $v_{0} \cdots v_{n}$ with $v_{i_0} \prec\cdots\prec v_{i_n}$, and
  $t(v_{0},\cdots,v_{n})$ is the number of 
  swaps in turning $v_{0} \cdots v_{n}$
  to $v_{i_0}\cdots v_{i_n}$. In addition, since $(K,\xi)$ is  divisibly weighted, $(\sigma_{v_{0} \cdots v_{n}},\xi)$ is a divisibly weighted simplex. 
   So by Lemma~\ref{Lem:Wt-Div-Simplex-Ord}, $\big(\widehat{C}_*(\sigma_{v_{0} \cdots v_{n}}), \hat{\partial}^{\xi}\big)$ is acyclic. Then we obtain an algebraic acyclic carrier $\Phi$ on $\big(\widehat{C}_*(K),\hat{\partial}^{\xi}\big)$ defined by
   \begin{equation} \label{Equ:Acyclic-Car-Order}
     \Phi\big(\langle v_{0},\cdots,v_{n} \rangle\big):= \big(\widehat{C}_*(\sigma_{v_{0} \cdots v_{n}}),\hat{\partial}^{\xi}\big). \end{equation}
  Obviously, $\Phi$ carries the chain map $\Theta \circ \Psi$. Then since $\Phi$ also carries $\mathrm{id}_{\widehat{C}_*(K)}$, it follows from Theorem~\ref{Thm:Acyclic-Carrier} that $\Theta \circ \Psi$ is chain homotopic to $\mathrm{id}_{\widehat{C}_*(K)}$.
  \end{proof}
  
   \begin{cor} \label{Cor:Iso-Hom-Order}
   Let $(K,\xi)$ be a divisibly weighted simplicial complex. For any abelian group $G$, the chain maps $\Psi$ and $\Theta$ induce group isomorphisms:
 $$\Psi_*: \widehat{H}_*(K,\hat{\partial}^{\xi};G) \rightarrow H_*(K,\partial^{\xi};G), \ \ \ \Theta_*: H_*(K,\partial^{\xi};G)\rightarrow \widehat{H}_*(K,\hat{\partial}^{\xi};G);  $$
 $$ \Psi^*: H^*(K,\delta^{\xi};G) \rightarrow 
 \widehat{H}^*(K,\hat{\delta}^{\xi};G), \ \ \  \Theta^*: \widehat{H}^*(K,\hat{\delta}^{\xi};G)\rightarrow
 H^*(K,\delta^{\xi};G).$$
  \end{cor}
  
  The definition of weighted cup product $\Cup$ in~\eqref{Equ:Prod-Wt-Homol} can be naturally extended to
 $\big( \widehat{C}^*(K;R), \hat{\delta}^{\xi}\big)$ where
 $(K,\xi)$ is a descending divisibly weighted simplicial complex and $R$ is a commutative ring with unit. Indeed, for any cochains $ \hat{\phi}\in \widehat{C}^p(K;R)$ and
  $\hat{\psi}\in \widehat{C}^q(K;R)$, let $\hat{\phi} 
 \, \hat{\Cup}\, \hat{\psi} \in \widehat{C}^{p+q}(K;R)$
  be the cochain whose value on an ordered
  $(p+q)$-simplex $\langle v_0,\cdots,v_{p+q}\rangle$ of $K$ is defined by 
  \begin{align}  \label{Equ:Cup-Order-def}
 & \ \ \ \ \hat{\phi}\, \hat{\Cup} \,\hat{\psi} \big(\langle v_0,\cdots,v_{p+q}\rangle\big) \\
  &:= \frac{\xi \big(\langle v_0,\cdots,v_p\rangle\big) \xi \big(\langle v_p,\cdots,v_{p+q}\rangle\big)}{\xi \big(\langle v_0,\cdots,v_{p+q}\rangle\big)} \, \hat{\phi}\big(\langle v_0,\cdots,v_p\rangle\big)
 \cdot \hat{\psi}\big(\langle v_p,\cdots,v_{p+q} \rangle\big). \notag
 \end{align}
 
 Note that we do not need to order the vertices of $K$ to define $\hat{\Cup}$. So this is a more intrinsic way to define the weighted cup product as suggested by the following theorem.
  
  \begin{thm} \label{Thm:Iso-Cup-Cohom}
   Let $(K,\xi)$ be a descending divisibly weighted simplicial complex. Then given any total ordering $\prec$ of the vertices of $K$, the chain map $\Theta$
   induces a ring isomorphism $\Theta^*:
   \big(\widehat{H}^*(K,\hat{\delta}^{\xi};R),\hat{\Cup} \big)\rightarrow \big(H^*(K,\delta^{\xi};R),\Cup\big)$.  
  \end{thm}
  \begin{proof}
    By Corollary~\ref{Cor:Iso-Hom-Order}, $\Theta$ already induces an additive isomorphism. So
    we only need to verify that $\Theta$ is a ring homomorphism. Denote by
    $$\Theta^{\#}: \big(\widehat{C}^*(K;R),\hat{\delta}^{\xi} \big) \rightarrow \big(C^*(K;R),\delta^{\xi}\big)$$
    the cochain map determined by $\Theta$. \n
        
    For any $ \hat{\phi}\in \widehat{C}^p(K;R)$,
  $\hat{\psi}\in \widehat{C}^q(K;R)$ and any oriented 
  $(p+q)$-simplex $[v_0,\cdots,v_{p+q}]$ of $K$ with
  $v_0\prec \cdots \prec v_{p+q}$, we have
  \begin{align*}
  &\ \ \ \ \Theta^{\#}(\hat{\phi}\, \hat{\Cup}\,\hat{\psi})\big( [v_0,\cdots,v_{p+q}] \big) = 
      \hat{\phi}\, \hat{\Cup}\,\hat{\psi} \big( 
      \Theta \big( [v_0,\cdots,v_{p+q}] \big) \big) 
    \overset{\eqref{Equ:Theta-def}}{=} \hat{\phi}\, \hat{\Cup}\,\hat{\psi}  \big( \langle v_0,\cdots,v_{p+q} \rangle \big) \\
    &\overset{\eqref{Equ:Cup-Order-def}}{=} \frac{\mu\big(\langle v_0,\cdots,v_p\rangle\big) \mu\big(\langle v_p,\cdots,v_{p+q}\rangle\big)}{\mu\big(\langle v_0,\cdots,v_{p+q}\rangle\big)} \, \hat{\phi}\big(\langle v_0,\cdots,v_p\rangle\big)
 \cdot \hat{\psi}\big(\langle v_p,\cdots,v_{p+q} \rangle\big) \\
 &= \frac{\mu\big(\langle v_0,\cdots,v_p\rangle\big) \mu\big(\langle v_p,\cdots,v_{p+q}\rangle\big)}{\mu\big(\langle v_0,\cdots,v_{p+q}\rangle\big)} \, \hat{\phi}\big(\Theta\big([v_0,\cdots,v_p]\big)\big)
 \cdot \hat{\psi}\big(\Theta\big([v_p,\cdots,v_{p+q}]\big)\big) \\
 &= \frac{\mu\big(\langle v_0,\cdots,v_p\rangle\big) \mu\big(\langle v_p,\cdots,v_{p+q}\rangle\big)}{\mu\big(\langle v_0,\cdots,v_{p+q}\rangle\big)} \, \Theta^{\#}(\hat{\phi})\big([v_0,\cdots,v_p]\big)
 \cdot \Theta^{\#}(\hat{\psi})\big([v_p,\cdots,v_{p+q}]\big) \\
 &\overset{\eqref{Equ:Prod-Wt-Homol}}{=} \Theta^{\#}(\hat{\phi})\Cup \Theta^{\#}(\hat{\psi}) \big( [v_0,\cdots,v_{p+q}] \big).
  \end{align*}
  So $\Theta^*$ is a ring homomorphism and the theorem is proved.  
  \end{proof}
  
 By the above theorem,  the ring structure induced by $\Cup$ on $H^*(K,\delta^{\xi};R)$ is independent on the choice of the ordering $\prec$ up to ring isomorphisms.

 \begin{rem}
  The cochain map $\Psi^{\#}$ induced by $\Psi$ is not a ring homomorphism from
   $\big(H^*(K,\delta^{\xi};R),\Cup\big)$ to
  $\big(\widehat{H}^*(K,\hat{\delta}^{\xi};R),\hat{\Cup} \big) $. Therefore, the map $\Psi^*$ is only an additive isomorphism.
 \end{rem}
\n
  
\subsection{Graded commutativity of the weighted cup product} \label{Subsec:Graded-Comm}
\ \n
  Now we are ready to prove that the weighted cup product
   $\Cup$
 on DW-cohomology is graded commutative. The proof is parallel to the proof of graded commutativity of
 singular cohomology in~\cite[Theorem 3.11]{Hatcher02}
 where here ordered simplices play the role of
 singular simplices.
 \n
 
 First of all, for a divisibly weighted simplicial complex $(K,\xi)$, we introduce an auxiliary map $\rho$ on $\widehat{C}_*(K)$ as follows: for any ordered simplex
 $\langle v_0,\cdots, v_n\rangle$ of $K$, 
 \begin{equation} \label{Equ:rho-Def}
  \rho \big(\langle v_0,\cdots, v_n\rangle \big) := \varepsilon_n \langle v_n,\cdots, v_0\rangle \ \text{where}\ 
 \varepsilon_n=(-1)^{\frac{n(n+1)}{2}}.
 \end{equation}
 
\begin{rem}
 Here we have to use ordered simplices instead of oriented simplices to define $\rho$. This is because
 $\rho$ would become the identity map if we
 replace $\langle v_0,\cdots, v_n\rangle$ in~\eqref{Equ:rho-Def} by an oriented simplex $[v_0,\cdots,v_n]$. This is another reason why we need to introduce ordered simplices in our study.
 \end{rem}
 \n
 
 \begin{lem} \label{Lem:Inver-Order}
  Suppose $(K,\xi)$ is a divisibly weighted simplicial complex. Then the map
  $\rho: \big(\widehat{C}_*(K), \hat{\partial}^{\xi}\big)\rightarrow 
  \big(\widehat{C}_*(K), \hat{\partial}^{\xi} \big)$ is a chain map 
  and $\rho$ is chain homotopic to the identity map
  $\mathrm{id}_{\widehat{C}_*(K)}$.  
 \end{lem}
 \begin{proof}
  Assume that $\xi$ is descending since the ascending case
  is completely parallel.
  For any ordered simplex
 $\langle v_0,\cdots, v_n\rangle$ of $K$, we have
 \begin{align*}
   \hat{\partial}^{\xi}\circ\rho\big( \langle v_0,\cdots, v_n\rangle \big) = &\, \varepsilon_n\,  \hat{\partial}^{\xi}\big( \langle v_n,\cdots, v_0\rangle \big)\\
   \overset{\eqref{Equ:Boundary-Ord-Simplex}}{=} &
 \varepsilon_n \sum^n_{i=0} (-1)^{i}\frac{\xi\big(\langle v_n,\cdots,\widehat{v}_{n-i}, \cdots, v_0 \rangle\big)}{\xi\big(\langle v_n,\cdots, v_0 \rangle\big)} \langle v_n,\cdots,\widehat{v}_{n-i}, \cdots, v_0 \rangle;
  \end{align*}
  \begin{align*}
  \rho\circ \hat{\partial}^{\xi} \big( \langle v_0,\cdots, v_n\rangle \big) = & \rho \Big(  
     \sum^n_{i=0} (-1)^i \frac{\xi\big(\langle v_0,\cdots,\widehat{v}_{i}, \cdots, v_n \rangle\big)}{\xi\big(\langle v_0,\cdots, v_n \rangle\big)} \langle v_0,\cdots,\widehat{v}_{i}, \cdots, v_n \rangle \Big)\\
     = & \varepsilon_{n-1} \Big(  
     \sum^n_{i=0} (-1)^{n-i} \frac{\xi\big(\langle v_0,\cdots,\widehat{v}_{n-i}, \cdots, v_n \rangle\big)}{\xi\big(\langle v_0,\cdots, v_n \rangle\big)} \langle v_0,\cdots,\widehat{v}_{n-i}, \cdots, v_n \rangle \Big).
   \end{align*}
 Since $\varepsilon_n = (-1)^n \varepsilon_{n-1}$, we obtain $\hat{\partial}^{\xi}\circ\rho = \rho\circ \hat{\partial}^{\xi}$, i.e. $\rho$ is a chain map.\n
 
 Moreover, $\rho$ is clearly augmentation-preserving and the acyclic carrier $\Phi$ defined in~\eqref{Equ:Acyclic-Car-Order} carries both $\rho$
 and $\mathrm{id}_{\widehat{C}_*(K)}$. So by
 Theorem~\ref{Thm:Acyclic-Carrier},
  $\rho$ is chain homotopic to $\mathrm{id}_{\widehat{C}_*(K)}$.   
 \end{proof}
 
 For a commutative ring $R$ with unit, by Lemma~\ref{Lem:Inver-Order} and the cohomology universal coefficient theorem,
 $\rho$ induces a cochain map $$\rho^{\#}: \big(\widehat{C}^*(K;R), \hat{\delta}^{\xi} \big) \rightarrow \big(\widehat{C}^*(K;R), \hat{\delta}^{\xi}\big)$$
 which is cochain homotopic to $\mathrm{id}_{\widehat{C}^*(K;R)}$. So the map $\rho^*$ induced by $\rho$ on the cohomology group is the identity map, that is
 \begin{equation} \label{Equ:rho-id}
   \rho^*=\mathrm{id}: \widehat{H}^p(K,\hat{\delta}^{\xi};R) \rightarrow \widehat{H}^p(K,\hat{\delta}^{\xi};R). 
\end{equation}

 \begin{lem} \label{Lem:rho-sign-commu}
Let $(K,\xi)$ be a descending divisibly weighted simplicial complex.
   For any cochains $ \hat{\phi}\in \widehat{C}^p(K;R)$,
  $\hat{\psi}\in \widehat{C}^q(K;R)$,
  \[ \rho^{\#}( \hat{\phi}\,\hat{\Cup}\,\hat{\psi} )
  =(-1)^{pq}  \rho^{\#}(\hat{\psi}) \,\hat{\Cup}\,  
   \rho^{\#}(\hat{\phi}).  \]
 \end{lem}
 \begin{proof}
   For any ordered simplex $\langle v_0,\cdots, v_{p+q}\rangle$ of $K$, we have
   \begin{align*}
    &\ \rho^{\#}( \hat{\phi}\,\hat{\Cup}\,\hat{\psi} )
     \big( \langle v_0,\cdots, v_{p+q}\rangle \big) 
    = \varepsilon_{p+q}\cdot \hat{\phi}\,\hat{\Cup}\,\hat{\psi} \big( 
      \langle v_{p+q},\cdots, v_{0}\rangle  \big) \\
     =&\, \varepsilon_{p+q}\cdot \frac{\xi \big(\langle v_{p+q},\cdots,v_q\rangle\big) \xi \big(\langle v_q,\cdots,v_{0}\rangle\big)}{\xi \big(\langle v_{p+q},\cdots,v_{0}\rangle\big)} \, \hat{\phi}\big(\langle v_{p+q},\cdots,v_q\rangle\big)
 \cdot \hat{\psi}\big(\langle v_q,\cdots,v_{0} \rangle\big).
   \end{align*}
   On the other hand,
 \begin{align*}
  &\ \rho^{\#}(\hat{\psi}) \,\hat{\Cup}\,  
   \rho^{\#}(\hat{\phi})  \big( \langle v_0,\cdots, v_{p+q}\rangle \big) \\
    =&\,  \frac{\xi \big(\langle v_{0},\cdots,v_q\rangle\big) \xi \big(\langle v_q,\cdots,v_{p+q}\rangle\big)}{\xi \big(\langle v_{0},\cdots,v_{p+q}\rangle\big)} \, \rho^{\#}(\hat{\psi})\big(\langle v_{0},\cdots,v_q\rangle\big)
 \cdot  \rho^{\#}(\hat{\phi}) \big(\langle v_q,\cdots,v_{p+q} \rangle\big)\\
 =&\, \varepsilon_{p} \varepsilon_{q}\cdot 
 \frac{\xi \big(\langle v_{0},\cdots,v_q\rangle\big) \xi \big(\langle v_q,\cdots,v_{p+q}\rangle\big)}{\xi \big(\langle v_{0},\cdots,v_{p+q}\rangle\big)} \,
 \hat{\psi}\big(\langle v_q,\cdots,v_{0} \rangle\big)
\cdot \hat{\phi}\big(\langle v_{p+q},\cdots,v_q\rangle\big).
 \end{align*}
 Then the lemma follows from the simple identity $\varepsilon_{p+q} = (-1)^{pq} \varepsilon_{p} \varepsilon_{q}$. 
 \end{proof}

 \begin{thm} \label{Thm:Graded-Comm-Cup}
 Let $(K,\xi)$ be a descending divisibly weighted simplicial complex.
  For any cohomology classes $[\hat{\phi}] \in \widehat{H}^p(K,\hat{\delta}^{\xi};R)$
  and $[\hat{\psi}]\in \widehat{H}^q(K,\hat{\delta}^{\xi};R)$,
  $$ [\hat{\phi}]\, \hat{\Cup}\, [\hat{\psi}] =(-1)^{pq} 
  [\hat{\psi}]\, \hat{\Cup}\, [\hat{\phi}]. $$
 \end{thm}
 \begin{proof}
   By the property of $\rho$ in ~\eqref{Equ:rho-id} and Lemma~\ref{Lem:rho-sign-commu}, we obtain
   \begin{align*}
     [\hat{\phi}]\, \hat{\Cup}\, [\hat{\psi}] =\rho^*([\hat{\phi}]\, \hat{\Cup}\, [\hat{\psi}])=
     \big[\rho^{\#} ( \hat{\phi} \, \hat{\Cup}\, \hat{\psi} ) \big] &= (-1)^{pq} \big[ \rho^{\#}(\hat{\psi}) \,\hat{\Cup}\, \rho^{\#}(\hat{\phi})  \big]\\
     &= (-1)^{pq} \big[ \rho^{\#}(\hat{\psi}) \big] \,\hat{\Cup}\, \big[\rho^{\#}(\hat{\phi})  \big] \\
     &=(-1)^{pq} 
  [\hat{\psi}]\, \hat{\Cup}\, [\hat{\phi}].
     \end{align*}
     So the product $\hat{\Cup}$ on $\widehat{H}^*(K,\hat{\delta}^{\xi};R)$ is graded commutative.     
 \end{proof}

 Then by the ring isomorphism $\Theta^*$ (see Theorem~\ref{Thm:Iso-Cup-Cohom}), we obtain the following corollary immediately.
 
 \begin{cor}\label{Cor:Graded-Comm-Cup}
  For a descending divisibly weighted simplicial complex
  $(K,\xi)$, the weighted cup product $\Cup$ on $H^*(K,\delta^{\xi};R)$ is graded commutative.
 \end{cor}
 
 It follows that for a weighted polyhedron $(X,\lambda)$,
 the weighted cup product $\Cup$ on the DW-cohomology
 $H^*_{DW}(X,\lambda;R)$ is graded commutative. \n
 
 Next, we give an example to show the difference between the DW-cohomology ring and the ordinary simplicial
 cohomology ring of a weighted polyhedron.
 
 \begin{figure}[h]
        \begin{equation*}
        \vcenter{
            \hbox{
                  \mbox{$\includegraphics[width=0.9\textwidth]{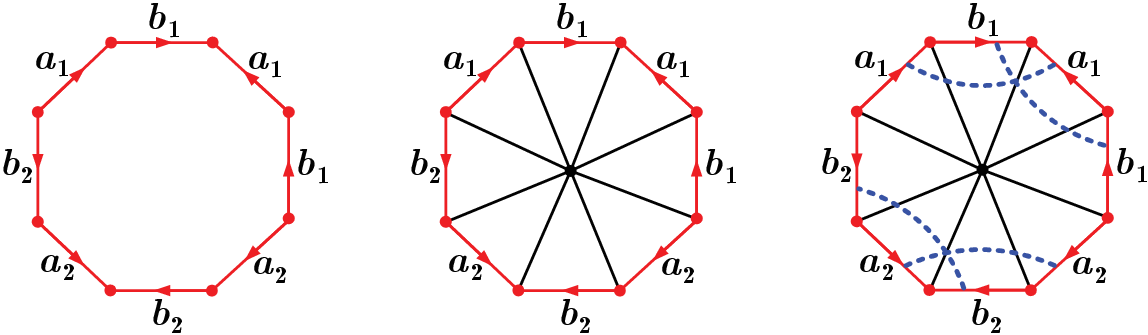}$}
                 }
           }
     \end{equation*}
   \caption{A pseudo-orbifold based on a genus $2$ closed orientable surface
       } \label{Fig:Example-Sigma-2}
   \end{figure}       
 
 \begin{exam} \label{Exam:DW-cohom-Surface}
   Let $\Sigma_2$ denote a closed orientable surface of genus $2$ which is glued from an octagon as shown in  Figure~\ref{Fig:Example-Sigma-2}. For any integer $k\geq 2$, let $\lambda_k$ be a weight function on $\Sigma_2$ which assigns $1$ to all the
   interior point of the octagon and assigns $k$ to
   those points from the edges on the boundary of the octagon. 
   Then the middle picture in Figure~\ref{Fig:Example-Sigma-2} is a divisibly weighted $\Delta$-triangulation of
   the pseudo-orbifold 
   $(\Sigma_2,\lambda_k)$, from which we can compute the 
   DW-cohomology group of $(\Sigma_2,\lambda_k)$:
   \[ H^p_{DW}(\Sigma_2,\lambda_k) \cong 
   H^p(\Sigma_2) \cong \begin{cases}
   \Z,  &  \text{$p=0,2$}; \\
   \Z\oplus\Z\oplus\Z\oplus\Z, & \text{$p=1$}; \\
   0,  &  \text{otherwise}.
 \end{cases} 
  \]

  The DW-cohomology ring $\big( H^*_{DW}(\Sigma_2,\lambda_k), \Cup \big)$, however, is not isomorphic as a graded ring to the ordinary simplicial cohomology ring
  $\big( H^*(\Sigma_2), \cup \big)$ of $\Sigma_2$. Indeed, let $\phi_1,\phi_2$, $\psi_1$ and $\psi_2$ denote the generators of
  $H^1_{DW}(\Sigma_2,\lambda_k)$ that are dual to
  $a_1,a_2, b_1$ and $b_2$, respectively. More specifically, $\phi_i$ (or $\psi_i$) has the value $1$ 
  on $a_i$ (or $b_i$) and the value $k$ on the two edges
   meeting the dotted arc that intersects $a_i$ (or $b_i$) (see the right picture in Figure~\ref{Fig:Example-Sigma-2}), and has the value $0$ on all other edges.   Let $\zeta$ denote the generator of $H^2_{DW}(\Sigma_2,\lambda_k)$ which has the value $1$ on all the $2$-simplices of $\Sigma_2$. Then we can directly compute from the $\Delta$-triangulation and the definition of $\Cup$ that the only nontrivial relations among the generators of $H^1_{DW}(\Sigma_2,\lambda_k)$ and $H^2_{DW}(\Sigma_2,\lambda_k)$ are:
  \[ \phi_1\Cup \psi_1 = \phi_2\Cup \psi_2 = k^2 \cdot \zeta. \] 
 So $\zeta$ is a multiplicative
  generator of $\big( H^*_{DW}(\Sigma_2,\lambda_k), \Cup \big)$ since $k\geq 2$.
  So there is no graded ring isomorphism between $\big( H^*_{DW}(\Sigma_2,\lambda_k), \Cup \big)$
 and $\big( H^*(\Sigma_2), \cup \big)$.      
 \end{exam}
 
 \n

\subsection{Weighted cap product}\ \n

 For a descending divisibly weighted simplicial complex $(K,\xi)$, we can also define a product 
 between elements of $H_*(K,\partial^{\xi};R)$
 and $H^*(K,\delta^{\xi};R)$ for any coefficient ring $R$. To avoid choosing a total ordering of the vertices of $K$, we give the definition of the product through ordered chains and cochains first.

 \begin{defi}[Weighted Cap Product]      
 	 \label{Def:Weighted-Cap-Prod}
 Let $(K,\xi)$ be a descending divisibly weighted simplicial complex. Define $R$-bilinear \emph{weighted cap product} 
  $$ \hat{\Cap}:  \widehat{C}_n(K;R)\times \widehat{C}^p(K;R) \rightarrow  \widehat{C}_{n-p}(K;R), \ 0\leq p\leq n $$ 
  by: for any cochain $\hat{\phi}\in \widehat{C}^p(K;R)$ and any ordered $n$-simplex $\langle v_0,\cdots,v_{n} \rangle$ of $K$, 
   \begin{equation*} 
   \langle v_0,\cdots,v_{n} \rangle  \,\hat{\Cap}\, \hat{\phi} := \frac{\xi\big(\langle v_0,\cdots, v_{p} \rangle\big)\xi\big(\langle v_p,\cdots, v_{n} \rangle\big)}{\xi\big(\langle v_0,\cdots, v_{n} \rangle\big)} \hat{\phi}\big(  \langle v_0,\cdots,v_{p}
   \rangle \big) \langle v_{p},\cdots,v_{n} \rangle. 
   \end{equation*}
 The coefficient on the right hand side is integral 
  because $\xi$ is a descending weight. 
 \end{defi}
 
 For any chain
 $\hat{\alpha} \in\widehat{C}_n(K;R)$,
 it is routine to check from the definitions that 
  \begin{equation} \label{Equ:Cap-Bound-Rel}
    \hat{\partial}^{\xi} \big( \hat{\alpha}\,\hat{\Cap}\,
    \hat{\phi} \big) = (-1)^p \big( ( \hat{\partial}^{\xi} \hat{\alpha})\,\hat{\Cap}\, \hat{\phi} - 
    \hat{\alpha}\,\hat{\Cap}\, \hat{\delta}^{\xi}\hat{\phi} \, \big).
  \end{equation}
 This implies that there is an induced product
 \begin{equation*} 
  \widehat{H}_n(K, \hat{\delta}^{\xi};R) \times \widehat{H}^p(K, \hat{\partial}^{\xi};R) \overset{\hat{\Cap}}{\longrightarrow} 
  \widehat{H}_{n-p}(K, \hat{\delta}^{\xi};R).
 \end{equation*} 
 Then by the isomorphisms $\Theta_*$ and $\Theta^*$
in Corollary~\ref{Cor:Iso-Hom-Order}, we obtain a product
  \begin{equation} 
  H_n(K, \delta^{\xi};R) \times H^p(K, \partial^{\xi};R) \overset{\Cap}{\longrightarrow} 
  H_{n-p}(K, \delta^{\xi};R).
 \end{equation} 
 We call ``$\Cap$'' the \emph{weighted cap product} of $(K,\xi)$.\n
  
 Moreover, the following lemma tells us that 
 the two products $\Cup$ and 
 $\Cap$ are compatible just as the ordinary
 cup product and cap product do.
 
 \begin{lem} \label{Lem:Cap-Cup-Compatible} 
 Let $(K,\xi)$ be a descending divisibly weighted simplicial complex. Then for any cochains $\hat{\phi}\in \widehat{C}^p(K;R)$,
 $\hat{\psi}\in \widehat{C}^q(K;R)$ and any chain
 $\hat{\alpha}\in\widehat{C}_n(K;R)$ with $p+q\leq n$, 
  $$ (\hat{\alpha} \,\hat{\Cap}\, \hat{\phi} ) \,\hat{\Cap}\, \hat{\psi} =  \hat{\alpha} \,\hat{\Cap}\,
  \big( \hat{\phi} \,\hat{\Cup}\, \hat{\psi} \big). $$
 \end{lem}
 \begin{proof}
  For any ordered $n$-simplex $\langle v_0,\cdots,v_{n} \rangle$ of $K$, we obtain 
 \begin{align*}
 &\ \big( \langle v_0,\cdots,v_{n} \rangle  \,\hat{\Cap}\, \hat{\phi} \big)  \,\hat{\Cap}\, \hat{\psi}\\
 =& \frac{\xi\big(\langle v_0,\cdots, v_{p} \rangle\big)\xi\big(\langle v_p,\cdots, v_{p+q} \rangle\big) \xi\big(\langle v_{p+q},\cdots, v_{n} \rangle\big)}{\xi\big(\langle v_0,\cdots, v_{n} \rangle\big)} \hat{\phi}\big(  \langle v_0,\cdots,v_{p} \big) 
 \hat{\psi}\big(  \langle v_p,\cdots,v_{p+q} \rangle \big) \langle v_{p+q},\cdots,v_{n} \rangle;
 \end{align*} 
 \begin{align*}
 &\  \big( \langle v_0,\cdots,v_{n} \rangle \big) \,\hat{\Cap}\, \big( \hat{\phi} \,\hat{\Cup}\, \hat{\psi} \big) \\
 =& \frac{\xi\big(\langle v_0,\cdots, v_{p+q} \rangle\big) \xi\big(\langle v_{p+q},\cdots, v_{n} \rangle\big)}{\xi\big(\langle v_0,\cdots, v_{n} \rangle\big)} 
 \big( \hat{\phi} \,\hat{\Cup}\, \hat{\psi} \big)\big(  \langle v_0,\cdots,v_{p+q} \rangle \big) 
  \langle v_{p+q},\cdots,v_{n} \rangle.
 \end{align*}
 Then it is easy see that $\big( \langle v_0,\cdots,v_{n} \rangle \big) \,\hat{\Cap}\, \big( \hat{\phi} \,\hat{\Cup}\, \hat{\psi} \big) = \big( \langle v_0,\cdots,v_{n} \rangle  \,\hat{\Cap}\, \hat{\phi} \big)  \,\hat{\Cap}\, \hat{\psi}$.
 \end{proof}
 
 By Lemma~\ref{Lem:Cap-Cup-Compatible} and the isomorphisms $\Theta_*$ and $\Theta^*$
in Corollary~\ref{Cor:Iso-Hom-Order}, we obtain the following corollary immediately.

\begin{cor} \label{Cor:Cap-Cup-Compatible}
Let $(K,\xi)$ be a descending divisibly weighted simplicial complex. Then for any cohomology classes $[\phi]\in H^p(K,\delta^{\xi};R)$, $[\psi]\in H^q(K,\delta^{\xi};R)$ and a homology class
 $[\alpha]\in H_n(K,\partial^{\xi};R)$ with $p+q\leq n$, 
  $$ ([\alpha] \Cap [\phi]) \Cap [\psi] =  [\alpha] \Cap
  ( [\phi] \Cup [\psi] ). $$
\end{cor}

By the above corollary, the product $\Cap$ induces a right $H^*(K,\delta^{\xi};R)$-module structure on $H_*(K,\partial^{\xi};R)$.
 \n
 
  By choosing a divisibly weighted triangulation of 
  a weighted polyhedron $(X,\lambda)$,  we can define
  the weighted cap product $\Cap$ of DW-homology and DW-cohomology $(X,\lambda)$, that is: for any $0\leq p\leq n$,
  \begin{equation*} 
  H^{DW}_n(X, \lambda;R) \times H^p_{DW}(X, \lambda;R) \overset{\Cap}{\longrightarrow} 
  H^{DW}_{n-p}(X, \lambda;R).
 \end{equation*} 
 \n
 
\begin{rem}
 We cannot define any $R$-bilinear product analogous to $\Cap$ between
 a AW-homology class and a DW-cohomology class (or between a DW-homology class and a AW-cohomology class). This is because  the relation in~\eqref{Equ:Cap-Bound-Rel} would fail when the chains and the cochains were of different types.
\end{rem}
  
 \vskip .4cm

 \section{Duality theorem of Weighted Homology and Cohomology} \label{Sec:Poincare-Duality}
 
 In this section, we prove the following theorem which is a direct generalization of Poincar\'e duality theorem in the category of weighted polyhedra. This is also an important reason why we study AW-homology and DW-homology at the same time.
 
   \begin{thm} \label{Thm:Poincare-Duality}
            Let $(X,\lambda)$ be a weighted polyhedra. If $X$ is an orientable compact homology $n$-manifold, then for all $0\leq p \leq n$, there are isomorphisms 
          $$
                H^p_{AW}(X,\lambda;G)\cong H_{n-p}^{DW}(X,\lambda;G),
           \ \
                H^p_{DW}(X,\lambda;G)\cong H_{n-p}^{AW}(X,\lambda;G).
           $$  
   If $X$ is non-orientable, the above isomorphisms are still valid for $G=\Z_2$.                   
     \end{thm} 
     
  \begin{rem}
    Even if the space $X$ is a closed orientable manifold, 
     $H^p_{AW}(X,\lambda)$ and $H_{n-p}^{AW}(X,\lambda)$ may not be isomorphic to each other, neither do $H^p_{DW}(X,\lambda;G)$ and $H_{n-p}^{DW}(X,\lambda)$. We can easily see this kind of phenomenon from Example~\ref{Exam:1-dim-pseudo-orbifold} and Example~\ref{Exam:Sphere-n-points} by a simple calculation via Universal Coefficient Theorem. 
  \end{rem}

   Recall a topological space $X$ is called a 
   \emph{homology $n$-manifold} if for each point $x$ of $X$, the local homology group 
   $$ H_i(X, X-x) =  \begin{cases}
      \Z,  &  \text{if $i=n$}; \\
       0,  &  \text{otherwise}.
 \end{cases} $$  
 Moreover, a compact triangulated homology $n$-manifold $X$ is called \emph{orientable} if it is possible to orient all the $n$-simplices $\sigma_i$ of $X$ so their sum $\gamma=\Sigma \sigma_i$ is a cycle. Such a cycle $\gamma$ is called an \emph{orientation cycle} for $X$.
 \n
 
 By the two versions of Universal Coefficient Theorem (see~\cite[Theorem 3.2]{Hatcher02} and~\cite[Theorem 56.1]{Munk84}), the two isomorphisms in Theorem~\ref{Thm:Poincare-Duality} imply each other.
 Moreover, if we have proven Theorem~\ref{Thm:Poincare-Duality} for 
 ascending type weighted-polyhedra, then the
 descending type case follows immediately by taking
 the adjoint (see~\eqref{Equ:adjoint-lambda}).
 So in the rest of this section, we only deal with ascending type weighted-polyhedra in our arguments.
\nn

Our proof of Theorem~\ref{Thm:Poincare-Duality} is analogous to the first proof of the Poincar\'e duality theorem in~\cite[\S 65]{Munk84}. The main idea is using the dual block decomposition of a homology $n$-manifold.

 \begin{defi}[Dual Block{~\cite[\S 64]{Munk84}}]\label{Def:Dual-Block}
        Let $K$ be a locally finite simplicial complex. Let $D(\sigma)$ denote the union of all the open simplices in $Sd(K)$ of the form 
        \begin{equation}
            [b_{\sigma_0},b_{\sigma_1},\cdots,b_{\sigma_k}]^{\circ}, \ \sigma =\sigma_0\subsetneq \dots \subsetneq \sigma_k.
        \end{equation}
        We call $D(\sigma)$ the \emph{dual block} to $\sigma$. We call the closure $\overline{D}(\sigma)$ of $D(\sigma)$ the \emph{closed block dual to} $\sigma$. 
        Let $\dot{D}(\sigma)=\overline{D}(\sigma)-D(\sigma)$. \n
   \begin{itemize}
   \item $\overline{D}(\sigma)$ consists of simplices of the form        
        \begin{equation*}
        [b_{\sigma_0},b_{\sigma_1},\cdots,b_{\sigma_k}], \ \sigma =\sigma_0\subsetneq \dots \subsetneq \sigma_k.
    \end{equation*}
    
    \item $\dot{D}(\sigma)$ consists of simplices of the form 
    \begin{equation*}
        [b_{\sigma_0},b_{\sigma_1},\cdots,b_{\sigma_k}], \ \sigma \subsetneq \sigma_0\subsetneq \dots \subsetneq \sigma_k.
    \end{equation*}
    \end{itemize}
    \end{defi}

   So by definition, $\overline{D}(\sigma)$ and $\dot{D}(\sigma)$ are both subcomplexes of $Sd(K)$.

     \begin{thm}[{\cite[Theorem 64.1]{Munk84}}]\label{Thm:DualBlock}
      Let $K$ be a locally finite simplicial complex that consists entirely of $n$-simplices and their faces. Let $\sigma$ be a $k$-simplex of $K$. Then:
      \begin{itemize}
      \item[(a)] The dual blocks are disjoint and their union is $|K|$.
      \n
      \item[(b)] $\overline{D}(\sigma)=b_{\sigma}*\dot{D}(\sigma)$ is a subcomplex of $Sd(K)$ of dimension $n-k$.\n
      
      \item[(c)] If $H_i(X, X-b_{\sigma}) \cong \mathbb{Z}$ for $i=n$ and vanishes otherwise, then $(\overline{D}(\sigma), \dot{D}(\sigma))$ has the homology of an $n-k$ cell modulo its boundary.
     \end{itemize}
    \end{thm}

  \begin{defi}[Dual Block Decomposition] \label{Def:Dual-Block-Decomp}
   Let $K$ be a locally finite simplicial complex that is a homology $n$-manifold. The collection of dual blocks $\{D(\sigma)\}_{\sigma\in K}$ are called the 
   \emph{dual block decomposition} of $|K|$. Let $K^*_p$ denote the union of all the dual blocks of dimension at most $p$, and called the \emph{dual $p$-skeleton} of $|K|$. The dual cellular chain complex $\mathcal{D}_*(K)$ of $|K|$ is defined by 
$$
\mathcal{D}_p(K)=H_p\left(K^*_p, K^*_{p-1}\right) .
$$
Its boundary operator $d$ is given by the homomorphism in the exact sequence of the triple $\left(K^*_p, K^*_{p-1}, K^*_{p-2}\right)$. 
 It is well known that the homology of 
 $(\mathcal{D}_*(K),d)$ is isomorphic to the simplicial (or singular) homology of $K$. 
 \end{defi}

 In the following, we will prove parallel results to Theorem~\ref{Thm:DualBlock} in the context of weighted simplicial complexes. But when we have a weight function on $K$, we need to be careful about which weight function 
 we should put on $Sd(K)$ that can induce appropriate
 weights on the dual blocks. 
 \n

   Let $(K,\xi)$ be a divisibly weighted simplicial complex. Then on $Sd(K)$, we can define three meaningful weight functions $Sd(\xi)$, $Sd(\widehat{\xi})$ and
 $\widehat{Sd (\widehat{\xi})}$.  For brevity, let 
 $$\widetilde{Sd}(\xi) := \widehat{Sd (\widehat{\xi})}.$$
 In general, $Sd(\xi)$ and $\widetilde{Sd}(\xi)$ are not equal on $Sd(K)$ (see Figure~\ref{Fig:Subdivision-Dual} for example).
 
  \begin{figure}[h]
        \begin{equation*}
        \vcenter{
            \hbox{
                  \mbox{$\includegraphics[width=0.98\textwidth]{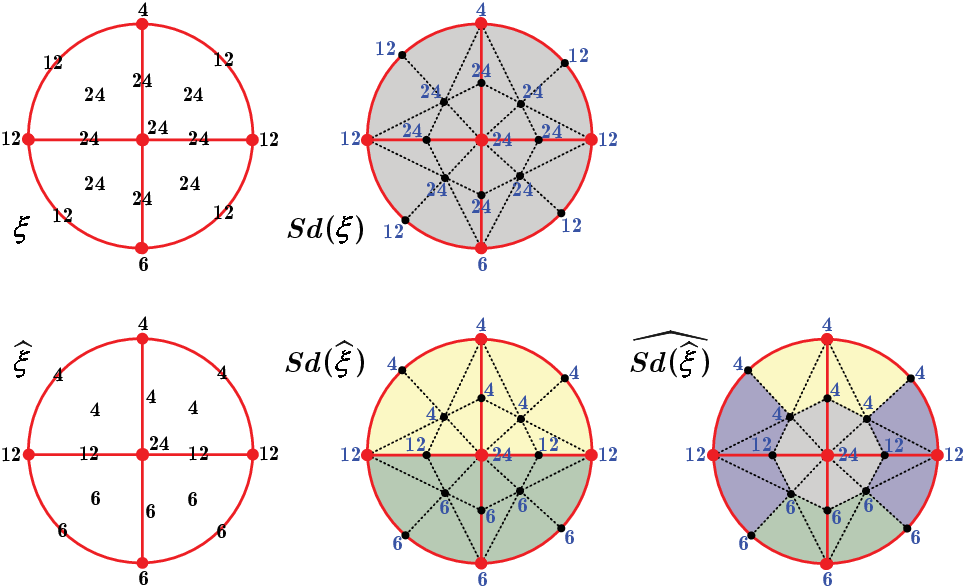}$}
                 }
           }
     \end{equation*}
   \caption{Induced weights on the barycentric subdivision of a divisibly weighted simplicial complex where here $\xi$ is ascending}\label{Fig:Subdivision-Dual}
   \end{figure}  
 
 \begin{lem} \label{Lem:Wt-Hom-Iso-Tilde-Sd}
  Let $(K,\xi)$ be a divisibly weighted simplicial complex. Then for every $p\geq 0$, there is an isomorphism of the weighted homology groups 
  $$ H_p\big(Sd(K),\partial^{Sd(\xi)} \big) \cong H_p\big(Sd(K), \partial^{\widetilde{Sd}(\xi)}\big).$$
   \end{lem}
\begin{proof}
 It follows from Theorem~\ref{Thm:Sd-Homology-Isom-1} 
 that there are isomorphisms:
    \begin{equation} \label{Equ:Isomorphisms}
    H_* \big(Sd(K),\partial^{Sd(\xi)} \big) \cong    H_*(K,\xi) = H_*(K, \widehat{\widehat{\xi}}) \cong H_*\big(Sd(K), \partial^{\widetilde{Sd}(\xi)} \big).
    \end{equation}
    Here we use the fact that the inversion of the inversion of $\xi$ goes back to $\xi$. 
\end{proof}

 Later in our proof of Theorem~\ref{Thm:Poincare-Duality}, we will mainly use 
  the weight function $\widetilde{Sd}(\xi)$ on
  $Sd(K)$. So let us first prove some properties of
  $\widetilde{Sd}(\xi)$.

   \begin{lem}\label{Lem:boundaryOfBarySimp}
        Let $(K,\xi)$ be a divisibly weighted simplicial complex where the weight $\xi$ is ascending. 
        If 
        $[b_{\sigma_0},b_{\sigma_1},\ldots,b_{\sigma_k}]$, $\sigma_0\subsetneq\sigma_1\subsetneq \dots \subsetneq \sigma_k$,
        is a generator of $C_p(Sd(K),\widetilde{Sd}(\xi))$, then 
        \begin{equation}
            \partial^{\widetilde{Sd}(\xi)}[b_{\sigma_0},b_{\sigma_1}, \ldots, b_{\sigma_k}]= \left(\frac{\widehat{\xi}(\sigma_0)}{\widehat{\xi}(\sigma_1)}-1\right) [b_{\sigma_1},\dots,b_{\sigma_k}] + \partial[b_{\sigma_0},b_{\sigma_1},\ldots,b_{\sigma_k}],
        \end{equation}
        where $\widehat{\xi}(\sigma_1)= \widetilde{Sd}(\xi)([b_{\sigma_1}, \ldots, b_{\sigma_k}])$.
    \end{lem}
    \begin{proof}
       By our assumption,  $\widehat{\xi}$ is a descending divisible weight on $K$. So
       $$ \widehat{\xi}(\sigma_k)\mid \cdots \mid \widehat{\xi}(\sigma_1) \mid \widehat{\xi}(\sigma_0). $$
           Then by Definition~\ref{Def:Bary-Subdiv-Weighted}, $Sd(\widehat{\xi})$ is a descending divisible weight on $Sd(K)$ with 
      $$Sd(\widehat{\xi})(b_{\sigma}) = \widehat{\xi}(\sigma),\ \text{for every simplex}\ \sigma \ \text{of}\ K.$$
       So we have
       \begin{equation} \label{Equ:Sd-hat-b}
        Sd(\widehat{\xi})(b_{\sigma_k}) \mid \cdots
        \mid Sd(\widehat{\xi})(b_{\sigma_1}) \mid Sd(\widehat{\xi})(b_{\sigma_0}). 
        \end{equation}
       Then since $\widetilde{Sd}(\xi)=\widehat{Sd(\widehat{\xi})}$ is an ascending divisible weight on $Sd(K)$ which agrees with 
       $Sd(\widehat{\xi})$ on each vertex of $Sd(K)$, we deduce from~\eqref{Equ:Sd-hat-b} that
       \begin{equation} \label{Equ:tilde-Sd}
        \widetilde{Sd}(\xi)\big([b_{\sigma_0}, b_{\sigma_1}, \ldots, b_{\sigma_k}]\big) = Sd(\widehat{\xi})(b_{\sigma_0}) =\widehat{\xi}(\sigma_0).
        \end{equation}
        Then we can compute
        \begin{equation*}
            \begin{aligned}
             \partial^{\widetilde{Sd}(\xi)}[b_{\sigma_0},
             b_{\sigma_1}, \ldots, b_{\sigma_k}]=&
             \frac{\widetilde{Sd}(\xi)\big([b_{\sigma_0}, b_{\sigma_1}, \ldots, b_{\sigma_k}]\big)}{ \widetilde{Sd}(\xi)\big([b_{\sigma_1}, \ldots, b_{\sigma_k}]\big)}                   
       [b_{\sigma_1},\ldots,b_{\sigma_k}] \\
             +& \ \sum_{i=1}^{k}(-1)^i
 \frac{\widetilde{Sd}(\xi)\big([b_{\sigma_0}, b_{\sigma_1}, \ldots, b_{\sigma_k}]\big)}{ \widetilde{Sd}(\xi)\big([b_{\sigma_0}, \ldots, \widehat{b}_{\sigma_i},\ldots, b_{\sigma_k}]\big)}
  [b_{\sigma_0},\ldots,\widehat{b}_{\sigma_i},\dots,b_{\sigma_k}] \\
     \overset{\eqref{Equ:tilde-Sd}}{=} &\ \frac{\widehat{\xi}(\sigma_0)}{\widehat{\xi}(\sigma_1)}[b_{\sigma_1},\dots,b_{\sigma_k}] + \sum_{i=1}^{k}(-1)^i \frac{\widehat{\xi}(\sigma_0)}{\widehat{\xi}(\sigma_0)}[b_{\sigma_0},\ldots,\widehat{b}_{\sigma_i},\dots,b_{\sigma_k}]\\ 
                = &\left(\frac{\widehat{\xi}(\sigma_0)}{\widehat{\xi}(\sigma_1)}-1\right) [b_{\sigma_1},\dots,b_{\sigma_k}] + \sum_{i=0}^{k}(-1)^i \frac{\widehat{\xi}(\sigma_0)}{\widehat{\xi}(\sigma_0)}[b_{\sigma_0},\dots,\widehat{b}_{\sigma_i},\dots,b_{\sigma_k}] \\
                =& \left(\frac{\widehat{\xi}(\sigma_0)}{\widehat{\xi}(\sigma_1)}-1\right)[b_{\sigma_1},\dots,b_{\sigma_k}] + \partial [b_{\sigma_0}, b_{\sigma_1},\ldots,b_{\sigma_k}].
            \end{aligned}
        \end{equation*}
        This proves the lemma.
     \end{proof}

   \begin{lem} \label{Lem:Homology-Dual-Block}
        Let $(K,\xi)$ be a divisibly weighted simplicial complex of ascending type and $|K|$ is a homology $n$-manifold. Then for any $k$-simplex $\sigma$ of $K$, 
        \begin{itemize}
        \item[(a)]
       $
            H_i\big(\overline{D}(\sigma), \dot{D}(\sigma), \partial^{\widetilde{Sd}(\xi)}\big)=
        \begin{cases}
            \Z, & \text{if $i=n-k$}; \\ 
            0, &  \text{if $i\neq n-k$}.
        \end{cases}
        $\n
    \item[(b)] The generator of $H_{n-k}\big(\overline{D}(\sigma), \dot{D}(\sigma), \partial^{\widetilde{Sd}(\xi)}\big)$ can be taken to be
     the sum of all the $(n-k)$-simplices of $\overline{D}(\sigma)$ with compatible orientations, denoted by $[\overline{D}(\sigma)]$.\n
     \item[(c)] For each $(n-k)$-simplex $s$ in $\overline{D}(\sigma)$, we have $\widetilde{Sd}(\xi)(s)=\widehat{\xi}(\sigma)$. 
     \end{itemize}
    \end{lem}
    \begin{proof}
     Consider the short exact sequence of simplicial chain complexes:
        \begin{equation*}
            0 \to  C_i(\dot{D}(\sigma)) \stackrel{i}{\to} C_i(\overline{D}(\sigma)) \stackrel{p}{\to} C_i(\overline{D}(\sigma),\dot{D}(\sigma)) \to  0 .
        \end{equation*}
        
        Let $\beta$ be a relative $i$-chain of $C_i(\overline{D}(\sigma),\dot{D}(\sigma))$ and $\alpha\in C_i(\overline{D}(\sigma)) $ be 
        a preimage of $\beta$ under the quotient homomorphism $p$. By Theorem~\ref{Thm:DualBlock}, we can write $\alpha$ as the join of $b_{\sigma}$ with a
         $(i-1)$-chain $c =\sum_j n_jc_j$ where
         each $c_j$ is a simplex of $\dot{D}(\sigma)$. 
         Using the notation from~\cite{Munk84}, we write $\alpha$ as
        \begin{equation*}
            \alpha=[b_{\sigma},c]=\sum_j n_j [b_{\sigma},c_j].
        \end{equation*} 
        Note that each vertex of $c_j$ is the barycenter 
        of some simplex  of $K$ that contains $\sigma$. 
        So from~\eqref{Equ:tilde-Sd}, we can deduce that
      \begin{equation} \label{Equ:Sd-c-j}
       \widetilde{Sd}(\xi)(c_j)\mid \widehat{\xi}(b_{\sigma})=\widehat{\xi}(\sigma) = \widetilde{Sd}(\xi)([b_{\sigma},c_j]).
       \end{equation}
        Then by Lemma~\ref{Lem:boundaryOfBarySimp}, 
         \begin{equation*}
            \begin{aligned}
                \partial^{\widetilde{Sd}(\xi)} [b_{\sigma},c] =& \sum_j n_j \partial^{\widetilde{Sd}(\xi)} [b_{\sigma},c_j] \\
                = &\sum_j n_j \left( \Big(\frac{\widehat{\xi} (\sigma)}{\widetilde{Sd}(\xi) (c_j)}-1\Big) c_j + \partial [b_{\sigma},c_j] \right) \\
                = & \left( \sum_j n_j  \Big(\frac{\widehat{\xi} (\sigma)}{\widetilde{Sd}(\xi) (c_j)}-1\Big) c_j \right) + \partial [b_{\sigma},c].
            \end{aligned}
        \end{equation*}
        Then we obtain
    \[ \partial^{\widetilde{Sd}(\xi)} [b_{\sigma},c] - \partial [b_{\sigma},c] = \sum_j n_j  \Big(\frac{\widehat{\xi} (\sigma)}{\widetilde{Sd}(\xi) (c_j)}-1\Big) c_j \in C_{j-1}(\dot{D}(\sigma)).\]
        This implies that
        $\partial^{\widetilde{Sd}(\xi)} [b_{\sigma},c]$ and $\partial [b_{\sigma},c]$ represent the same relative $(i-1)$-chain of $C_{i-1}(\overline{D}(\sigma),\dot{D}(\sigma))$, i.e. 
       $
            \partial^{\widetilde{Sd}(\xi)} \beta = \partial \beta \in C_{i-1}(\overline{D}(\sigma),\dot{D}(\sigma))$.
      Therefore, we have the following commutative diagram 
      \begin{equation*}
            \begin{tikzcd}
                \cdots \ \ C_i(\overline{D}(\sigma),\dot{D}(\sigma)) \arrow[r,"\partial^{\widetilde{Sd}(\xi)}"] \arrow[d, shift left=4mm, "="]& C_{i-1}(\overline{D}(\sigma),\dot{D}(\sigma)) \arrow[d,shift right=4mm,"="] \ \ \cdots\\ 
                \cdots \ \ C_i(\overline{D}(\sigma),\dot{D}(\sigma)) \arrow[r,"\partial"] & C_{i-1}(\overline{D}(\sigma),\dot{D}(\sigma)) \ \ \cdots
            \end{tikzcd}
        \end{equation*}
      In other words, the identity map of $C_*\big(\overline{D}(\sigma),\dot{D}(\sigma)\big)$ induces a chain map between $\big(C_*(\overline{D}(\sigma),\dot{D}(\sigma)),
   \partial^{\widetilde{Sd}(\xi)} \big)$ and 
   $\big( C_*(\overline{D}(\sigma),\dot{D}(\sigma)),\partial\big)$. So we have an isomorphism
        \begin{equation*}
            H_i\big(\overline{D}(\sigma),\dot{D}(\sigma),\partial^{\widetilde{Sd}(\xi)}\big) \cong H_i(\overline{D}(\sigma),\dot{D}(\sigma)), \ i\geq 0.
        \end{equation*}
        Then the claims in (a) and (b) follow from Theorem~\ref{Thm:DualBlock}(c). 
        \n
      Moreover, since each $(n-k)$-simplex $s$ in $\overline{D}(\sigma)$ is the join of $b_{\sigma}$ with some simplex in $\dot{D}(\sigma)$, the results in~\eqref{Equ:tilde-Sd} and~\eqref{Equ:Sd-c-j} 
      imply that $\widetilde{Sd}(\xi)(s)=\widehat{\xi}(b_{\sigma})=\widehat{\xi}(\sigma)$. This finishes the proof.
    \end{proof}
    
    \begin{rem}
     The homology groups $H_*\big(\overline{D}(\sigma), \dot{D}(\sigma),\partial^{Sd(\xi)}\big)$ may not satisfy Lemma~\ref{Lem:Homology-Dual-Block}(a) as
      $ H_*\big(\overline{D}(\sigma),\dot{D}(\sigma),\partial^{\widetilde{Sd}(\xi)}\big)$ do.
      This is the reason why in the proof of Theorem~\ref{Thm:Poincare-Duality}, we use $\widetilde{Sd}(\xi)$ instead of $Sd(\xi)$ as the weighted function on $Sd(K)$.
    \end{rem}

     \begin{lem}\label{Lem:Homology-Dual-Skeleton}
        Let $(K,\xi)$ be a divisibly weighted simplicial complex of ascending type and $|K|$ is a homology $n$-manifold. Then for each $0\leq p \leq n$, 
        \begin{equation*}
            H_i\big(K^*_p,K^*_{p-1},\partial^{\widetilde{Sd}(\xi)}\big)= 
            \begin{cases}
                0, &  \text{if $i\neq p$}; \\ 
                \Z^{l_p}, & \text{if $i=p$},
            \end{cases}
        \end{equation*}
        where $l_p$ is the number of $(n-p)$-simplices in $K$. Moreover, a free basis of $H_p\big(K^*_p,K^*_{p-1},\partial^{\widetilde{Sd}(\xi)}\big)$ is $\{ [\overline{D}(\sigma)] \,;\, \sigma \ \text{is any $(n-p)$-simplex
        of}\ K \}$.
    \end{lem}
     \begin{proof}
        By definition, 
        \[K^*_{p}=\bigcup_{\dim (\sigma) = n-p} \overline{D}(\sigma) ,\ \ \  K^*_{p-1}=\bigcup_{\dim (\tau) = n-p+1} \overline{D}(\tau).\]
       An $(n-p)$-simplex $\sigma$ is a face of an 
        $(n-p+1)$-simplex $\tau$ if and only if
          $\overline{D}(\tau)\subset \dot{D}(\sigma)$. 
          This implies 
       \begin{equation*}
            (K^*_p,K^*_{p-1})=\bigcup_{\dim (\sigma) =n-p} \big(\overline{D}(\sigma),\dot{D}(\sigma)\big).
        \end{equation*}
      So we have
        \[
            C_*\big(K_p^*,K_{p-1}^*,\partial^{\widetilde{Sd}(\xi)}\big) \cong \bigoplus_{\dim (\sigma) =n-p} C_*\big(\overline{D}(\sigma),\dot{D}(\sigma),\partial^{\widetilde{Sd}(\xi)}\big),
      \]
      \[
            H_* \big(  K^*_p,K^*_{p-1},\partial^{\widetilde{Sd}(\xi)} \big)\cong \bigoplus_{\dim (\sigma) =n-p}  H_*\big(\overline{D}(\sigma),\dot{D}(\sigma),\partial^{\widetilde{Sd}(\xi)} \big).
      \]
        Then the lemma follows from Lemma~\ref{Lem:Homology-Dual-Block}.
        \end{proof}
        
  From the above lemma, it is easy to prove the following theorem.   
  
  \begin{thm} \label{Thm:dual-cell}
        Let $(K,\xi)$ be a divisibly weighted simplicial complex of ascending type and $|K|$ is a homology $n$-manifold. Then for every $p\geq 0$,
        \begin{itemize}
        \item[(a)] The inclusion $K^*_p \hookrightarrow K^*_{p+1}$ induces group isomorphisms
        \begin{equation*}
            H_i\big(K^*_p,\partial^{\widetilde{Sd}(\xi)}\big) \cong H_i\big(K^*_{p+1},\partial^{\widetilde{Sd}(\xi)}\big), \ i \neq p, p+1.
        \end{equation*}
        
        \item[(b)]  The inclusion $K^*_p \hookrightarrow Sd(K)$ induces group isomorphisms
        \begin{equation*}
            H_i\big(K^*_p,\partial^{\widetilde{Sd}(\xi)}\big) \cong H_i\big(Sd(K),\partial^{\widetilde{Sd}(\xi)}\big), \  i\leq p-1.
        \end{equation*}
        
        \item[(c)] When $i \geq p+1$, we have 
        $H_i\big(K_p^*,\partial^{\widetilde{Sd}(\xi)}\big)=0$.
        \end{itemize}
    \end{thm}    
     
       For each $p\geq 0$, denote $H_p\big(K^*_p,K^*_{p-1},\partial^{\widetilde{Sd}(\xi)}\big)$ by $\mathcal{D}_p\big(K,\widetilde{Sd}(\xi)\big)$.\n

        Similarly to the ordinary cellular chain complex, we can construct a differential 
       $$d^{\xi}: \mathcal{D}_p\big(K,\widetilde{Sd}(\xi)\big) \rightarrow
       \mathcal{D}_{p-1}\big(K,\widetilde{Sd}(\xi)\big), \ \forall p\geq 0, $$ 
        such that the homology of the chain complex $\big(\mathcal{D}_*(K,\widetilde{Sd}(\xi)),d^{\xi} \big)$ is isomorphic to the homology of $\big(C_*(Sd(K)),\partial^{\widetilde{Sd}(\xi)}\big)$.   
   Indeed, consider the following diagram 
    \begin{equation*}
        \begin{tikzcd}
            H_p\big(K^*_p,K^*_{p-1},\partial^{\widetilde{Sd}(\xi)}\big) \arrow[r,"\partial_p^{\widetilde{Sd}(\xi)}"] & H_{p-1}\big(K^*_{p-1},\partial^{\widetilde{Sd}(\xi)}\big) \arrow[r] \arrow[d,"="]& H_{p-1}\big(K^*_{p}, \partial^{\widetilde{Sd}(\xi)}\big) \\ 
            H_{p-1}\big(K^*_{p-2},\partial^{\widetilde{Sd}(\xi)}\big) \arrow[r] & H_{p-1}\big(K^*_{p-1},\partial^{\widetilde{Sd}(\xi)}\big) \arrow[r,"i_{p-1}"] & H_{p-1}\big(K^*_{p-1},K^*_{p-2},\partial^{\widetilde{Sd}(\xi)}\big) 
        \end{tikzcd}
    \end{equation*}
    where the two horizontal rows are the long exact sequences of homology groups of $(K_p^*,K_{p-1}^*)$ and
    $(K_{p-1}^*,K_{p-2}^*)$. Then define the differential
    $d^{\xi}$ to be
   \begin{equation} \label{Equ:d-xi}
     d^{\xi}_p=i_{p-1}\circ \partial_p^{\widetilde{Sd}(\xi)} : \mathcal{D}_p\big(K,\widetilde{Sd}(\xi)\big) \to \mathcal{D}_{p-1}\big(K,\widetilde{Sd}(\xi)\big).
     \end{equation}

   \begin{lem} \label{Lem:Sd-tilde-Sd-Isom}
  For every $p\geq 0$, there is an isomorphism 
  $$H_p \big(\mathcal{D}_*\big(K,\widetilde{Sd}(\xi)\big),d^{\xi} \big)\cong H_p\big(Sd(K),\partial^{\widetilde{Sd}(\xi)}\big).$$
   \end{lem}
   \begin{proof}
   The argument is parallel to that for
   the isomorphism between the singular homology and the cellular homology of a CW-complex (see~\cite[Theorem 39.4]{Munk84} and~\cite[Theorem 39.5]{Munk84}). We have the following diagram based on the results in Theorem~\ref{Thm:dual-cell}.   
        
   {\scriptsize  \begin{equation*}\label{dualdf}
            \begin{tikzcd}[column sep=tiny]
            & & & 0 &        \\
               H_{p}\big(K_{p-1}^*,\partial^{\widetilde{Sd}(\xi)}\big)=0 \arrow[dr]& & H_{p} \big(K^*_{p+1},\partial^{\widetilde{Sd}(\xi)}\big)\arrow[ur, shorten=2.5mm] \\ 
              & H_{p} \big(K^*_{p},\partial^{\widetilde{Sd}(\xi)}\big) \arrow[dr,"i_p"] \arrow[ur]\\
              H_{p+1} \big(K^*_{p+1},K^*_{p},\partial^{\widetilde{Sd}(\xi)}\big) \arrow[rr,"d^{\xi}_{p+1}",dashrightarrow] \arrow[ur,"\partial^{\widetilde{Sd}(\xi)}_{p+1}"]& &H_{p} \big(K^*_{p}, K^*_{p-1},\partial^{\widetilde{Sd}(\xi)}\big) \arrow[rr,"d^{\xi}_p",dashrightarrow] \arrow[dr,"\partial_p^{\widetilde{Sd}(\xi)}"']& & H_{p-1} \big(K^*_{p-1}, K^*_{p-2},\partial^{\widetilde{Sd}(\xi)}\big)\\
              & & & H_{p-1} \big(K^*_{p-1},\partial^{\widetilde{Sd}(\xi)}\big) \arrow[ur,"i_{p-1}"'] \\
               & & H_{p-1}\big(K^*_{p-2},\partial^{\widetilde{Sd}(\xi)}\big)=0 \arrow[ur] & & 
              \end{tikzcd}.
        \end{equation*}}
          
           It is easy to derive that
   $$H_p\big(Sd(K),\partial^{\widetilde{Sd}(\xi)}\big)
   \cong H_{p} \big(K^*_{p+1},\partial^{\widetilde{Sd}(\xi)}\big) \cong H_{p} \big(K^*_{p}, K^*_{p-1},\partial^{\widetilde{Sd}(\xi)}\big) \big\slash
   \mathrm{Im}\big(\partial^{\widetilde{Sd}(\xi)}_{p+1}\big).$$
   
  Moreover, since $i_p$ and $i_{p-1}$ in the diagram are
   both injective, it is easy to see that
      $H_p\big(Sd(K),\partial^{\widetilde{Sd}(\xi)}\big) \cong  \ker(d^{\xi}_p)\slash \mathrm{Im}(d^{\xi}_{p+1}) =H_p\big(\mathcal{D}_*(K,\widetilde{Sd}(\xi)),d^{\xi} \big)$.  
   \end{proof}

   \begin{cor} \label{Cor:Dual-Isom-AW-Homol}
   If $(K,\xi)$ is a divisibly weighted triangulation of 
   an ascending type weighted polyhedron $(X,\lambda)$, then there is an isomorphism
    $$H_p\big(\mathcal{D}_*\big(K,\widetilde{Sd}(\xi)\big),d^{\xi} \big)\cong H^{AW}_p(X,\lambda), \ \forall p\geq 0.$$
    \end{cor}
  \begin{proof}  
  By Theorem~\ref{Thm:Sd-Homology-Isom-1} and Lemma~\ref{Lem:Wt-Hom-Iso-Tilde-Sd},
  $$H^{AW}_p(X,\lambda)= H_p(K,\partial^{\xi})\cong H_p\big(Sd(K),\partial^{Sd(\xi)} \big) \cong H_p\big(Sd(K),\partial^{\widetilde{Sd}(\xi)}\big).$$ 
  Then the corollary follows from Lemma~\ref{Lem:Sd-tilde-Sd-Isom}.
\end{proof}
   
 Next, we write an explicit formula for $ 
 d^{\xi}_p :\mathcal{D}_p\big(K,\widetilde{Sd}(\xi)\big) \rightarrow \mathcal{D}_{p-1}\big(K,\widetilde{Sd}(\xi)\big)$.
 By Lemma~\ref{Lem:Homology-Dual-Skeleton}, 
 it amounts to determine $d^{\xi}_p([\overline{D}(\sigma)])$ for each $(n-p)$-simplex $\sigma$ of $K$.
 We can compute
  directly from the definition~\eqref{Equ:d-xi} that
  \begin{equation} \label{Equ:d-xi-formula}
    d^{\xi}_p([\overline{D}(\sigma)]) = \sum_i \varepsilon_i \frac{\widehat{\xi}(\sigma)}{\widehat{\xi}(\tau_i)} [\overline{D}(\tau_i)] ,
    \end{equation}
  where the sum ranges over all the $(n-p+1)$-simplex
  $\tau_i$ of $K$ that contains $\sigma$ as a face. Here,
   $\varepsilon_i= \pm 1$ is the coefficient of $\sigma$  in the expression for $\partial \tau_i$, and the weight factor
   $\frac{\widehat{\xi}(\sigma)}{\widehat{\xi}(\tau_i)} $ follows from Lemma~\ref{Lem:Homology-Dual-Block}(c).
  \n
  
  Now, we are prepared to prove Theorem~\ref{Thm:Poincare-Duality}.
  
  \nn
  
 \noindent \textbf{\textit{Proof of Theorem~\ref{Thm:Poincare-Duality}}.}
  As we argued at the beginning of this section, it is sufficient for us to prove the isomorphism $H^p_{DW}(X,\lambda;G)\cong H_{n-p}^{AW}(X,\lambda;G)$
 for any ascending type weighted polyhedron $(X,\lambda)$. In addition, we will omit all the coefficients $G$
 in our argument below.\n
 
  First, choose $(K,\xi)$ to be a divisibly weighted triangulation of $(X,\lambda)$. By our assumption that $X$ is orientable, we can orient the $n$-simplices of $K$ so their sum is an $n$-cycle. Orient the other simplices of $K$ arbitrarily.  
  The oriented $p$-simplices $\sigma$ of $K$ form a basis of the simplicial chain group $C_p(X)$; their duals $\sigma^*$ form a basis of the cochain group $C^p(X)$. 
  Moreover, the orientation of a $p$-simplex $\sigma$ determines 
  the orientation of its dual block $D(\sigma)$, and hence a generator $[\overline{D}(\sigma)]$ of
 $H_{n-p}(\overline{D}(\sigma), \dot{D}(\sigma))$.\n
 
 By the proof of (ordinary) Poincar\'e duality
 theorem for the oriented homology $n$-manifold $X$ (see~\cite[Theorem 65.1]{Munk84}),
  one can use the dual block decomposition  $\{D(\sigma)\}_{\sigma\in K}$ of $X$ (see Definition~\ref{Def:Dual-Block-Decomp}) to obtain a commutative diagram 
   \begin{equation*} 
            \begin{tikzcd}
                \dots \arrow[r] & C^p(K) \arrow[d,"\phi"] \arrow[r,"\delta"] & C^{p+1}(K) \arrow[d,"\phi"] \arrow[r] & \dots \\ 
                \dots \arrow[r] & \mathcal{D}_{n-p}(K) \arrow[r,"d"] & \mathcal{D}_{n-p-1}(K) \arrow[r] & \dots
            \end{tikzcd}
        \end{equation*}
   where $\phi$ maps each $p$-simplex $\sigma^*\in C^p(X)$ to $[\overline{D}(\sigma)]\in H_{n-p}(\overline{D}(\sigma), \dot{D}(\sigma))$. This implies that the simplicial homology and cohomology groups of $K$ satisfies the duality relation $H^p(K)\cong H_{n-p}(\mathcal{D}(K),d) \cong H_{n-p}(K)$ for all $p\geq 0$.\n

   Here in the presence of a divisible weight $\xi$ on $K$, Lemma~\ref{Lem:Homology-Dual-Block}(b) tells us that $[\overline{D}(\sigma)]$ is also a generator of $H_{n-p}\big(\overline{D}(\sigma), \dot{D}(\sigma),\partial^{\widetilde{Sd}(\xi)}\big)$. Moreover,   
   Lemma~\ref{Lem:Homology-Dual-Skeleton} says that
  all the $[\overline{D}(\sigma)]$ with $\dim(\sigma)=p$ form a free basis of $\mathcal{D}_{n-p}\big(K,\widetilde{Sd}(\xi)\big) $.\n
  
   \begin{itemize}
   \item[\textbf{Claim:}] The
   map $\phi : \sigma^* \mapsto [\overline{D}(\sigma)]$ makes the following diagram commute as well: 
   \end{itemize}
   \begin{equation*}
            \begin{tikzcd}
                \dots \arrow[r] & C^p(K,\delta^{\widehat{\xi}}) \arrow[d,"\phi"] \arrow[r,"\delta^{\widehat{\xi}}"] & C^{p+1}(K,\delta^{\widehat{\xi}}) \arrow[d,"\phi"] \arrow[r] & \dots \\ 
                \dots \arrow[r] & \mathcal{D}_{n-p}\big(K,\widetilde{Sd}(\xi)\big) \arrow[r,"d^{\xi}"] & \mathcal{D}_{n-p-1}\big(K,\widetilde{Sd}(\xi)\big) \arrow[r] & \dots
            \end{tikzcd}
        \end{equation*}
   where $\widehat{\xi}$ is the inversion of $\xi$.     
  Indeed, if $\delta(\sigma) = \sum_i \varepsilon_i \tau_i$, then by definition
  $$\delta^{\widehat{\xi}}(\sigma^*) = \sum_i \varepsilon_i \frac{\widehat{\xi}(\sigma)}{\widehat{\xi}(\tau_i)} \tau^*_i, \ \ 
  \phi\circ \delta^{\widehat{\xi}}(\sigma^*)= 
   \sum_i \varepsilon_i \frac{\widehat{\xi}(\sigma)}{\widehat{\xi}(\tau_i)} [\overline{D}(\tau_i)].$$ 
   So we have 
  \begin{equation*} 
   \begin{aligned}
   d^{\xi}\circ \phi (\sigma^*) =  d^{\xi}([\overline{D}(\sigma)]) \overset{\eqref{Equ:d-xi-formula}}{=} 
   \sum_i \varepsilon_i \frac{\widehat{\xi}(\sigma)}{\widehat{\xi}(\tau_i)} [\overline{D}(\tau_i)] =  \phi\circ \delta^{\widehat{\xi}}(\sigma^*).
   \end{aligned}  
  \end{equation*}
  The commutativity of the above diagram implies that there is an isomorphism
  $$ H^p(K,\delta^{\widehat{\xi}})\cong H_{n-p}\big(\mathcal{D}_{n-p}(K,\widetilde{Sd}(\xi)),d^{\xi}\big), \ \forall p\geq 0.$$
  
  Then by Corollary~\ref{Cor:Dual-Isom-AW-Homol}, 
  $ H^p_{DW}(X,\lambda) = H^p(K,\delta^{\widehat{\xi}})\cong H^{AW}_{n-p}(X,\lambda)$ for all $p\geq 0$.
  This finishes the proof when $X$ is an orientable homology $n$-manifold. \n
  
  If $X$ is non-orientable but with the coefficient group $G=\Z_2$, the proof is parallel. \qed

 \begin{rem}
   It is shown in Takeuchi and Yokoyama~\cite{TakYor12} that under certain orientation assumption, the t-singular homology and
   the ws-singular cohomology of an orbifold are dual to each other. But since the relation between
  our weighted homology theory and Takeuchi and Yokoyama's t-singular homology theory is not known, it is not clear whether the duality between the t-singular homology and
  the ws-singular cohomology is equivalent to the duality between AW-homology and DW-cohomology in our Theorem~\ref{Thm:Poincare-Duality}.
 \end{rem}
   
    \vskip .3cm
      
   \section*{Acknowledgment}
  During the process of writing this paper, 
Lisu Wu is partially supported by  National Natural Science Foundation of China (Grant No.\,12201359) and  Natural Science Foundation of Shandong Province, China (Grant No.\,ZR2022QA028), and Li Yu is partially supported by
  National Natural Science Foundation of China (Grant No.\,11871266).

\end{document}